\documentclass[13pt,CJK]{article}
\usepackage{amsfonts,amsmath,amssymb,amsthm,mathrsfs}
\usepackage{latexsym}
\usepackage{graphicx}
\usepackage{indentfirst}
\usepackage{color}
\usepackage{fancyhdr}
\usepackage[colorlinks,linktocpage,linkcolor=blue]{hyperref}

\theoremstyle{plain}                       % default
\newtheorem{lemma}{Lemma}[section]
\newtheorem{theorem}[lemma]{Theorem}

\newtheorem{remark}[lemma]{Remark}
\newtheorem{definition}[lemma]{Definition}

\theoremstyle{remark}

\def\Xint#1{\mathchoice
  {\XXint\displaystyle\textstyle{#1}}%
  {\XXint\textstyle\scriptstyle{#1}}%
  {\XXint\scriptstyle\scriptscriptstyle{#1}}%
  {\XXint\scriptscriptstyle\scriptscriptstyle{#1}}%
  \!\int}
\def\XXint#1#2#3{{\setbox0=\hbox{$#1{#2#3}{\int}$}
  \vcenter{\hbox{$#2#3$}}\kern-.5\wd0}}

\def\dashint{\Xint-}
\begin{document}
\allowdisplaybreaks
\pagestyle{myheadings}\markboth{$~$ \hfill {\rm Q. Xu,} \hfill $~$} {$~$ \hfill {\rm  } \hfill$~$}
%\author{Li Wang
%%\thanks{Email: lwang10@lzu.edu.cn.}
%\quad Qiang Xu
%\thanks{Corresponding author}
%\thanks{Email: xuqiang09@lzu.edu.cn.}
%\quad Peihao Zhao
%%\thanks{Email: phzhao@lzu.edu.cn.}
%\\
%School of Mathematics and Statistics, Lanzhou University, \\
%Gansu, 730000, PR China.
%\vspace{0.5cm}
%}

\author{
Li Wang,\quad Qiang Xu, \quad Peihao Zhao\\
%School of Mathematics and Statistics, Lanzhou University, \\
%Lanzhou, 710000, China.\\
%\thanks{Corresponding author.}
%\thanks{Email: qiangxu@mis.mpg.de.}\\
%Max Planck Institute for Mathematics in the Sciences, \\
%Inselstrasse 22, 04103 Leipzig, Germany \vspace{0.2cm}\\
%Peihao Zhao\\
%School of Mathematics and Statistics, Lanzhou University, \\
%Lanzhou, 710000, China. \vspace{0.2cm}
}

%

%\author{Weiren Zhao
%\thanks{Email: xuqiang@math.pku.edu.cn.}
%\thanks{This work was supported by the National Natural Science Foundation of China (Grant No. 11471147).}

\title{\textbf{Quantitative Estimates for Homogenization
of Nonlinear Elliptic Operators in Perforated Domains}}
\maketitle
\begin{abstract}
This paper was devoted to study the quantitative homogenization
problems for nonlinear elliptic operators in perforated
domains. We obtained a sharp error estimate $O(\varepsilon)$ when the problem was anchored in the reference domain $\varepsilon\omega$.
If concerning a bounded perforated domain, one will see a bad influence from the boundary layers,
which leads to the loss of the convergence rate by $O(\varepsilon^{1/2})$.
Equipped with the error estimates, we developed both interior and boundary Lipschitz estimates at large-scales. As an application, we
received the so-called quenched Calder\'on-Zygumund estimates by Shen's
real arguments. To overcome some difficulties, we
improved the extension theory from (\cite[Theorem 4.3]{OSY}) to $L^p$-versions with  $\frac{2d}{d+1}-\epsilon<p<\frac{2d}{d-1}+\epsilon$ and $0<\epsilon\ll1$. Appealing to this, we established
Poincar\'e-Sobolev inequalities of local type on perforated domains.
Some of results in the present literature are new even for related linear elliptic models.
\\
\textbf{Key words.}
homogenization; perforated domains;
nonlinear elliptic operators;
convergence rates; large-scale Lipschitz estimates; quenched Calder\'on-Zygumund estimates.
\end{abstract}

\tableofcontents

\section{Introduction and main results}

\subsection{Hypothesises and main results}\label{subsec:1.1}
The aim of the present paper is to establish some error estimates
and large-scale Lipschitz estimates for
a class of monotone operators in periodically perforated domains,
arising in the homogenization theory.
More precisely, let $d\geq 2$ and $\Omega\subset\mathbb{R}^d$ be a bounded
Lipschitz domain (unless otherwise stated). Let $\omega\subset \mathbb{R}^{d}$
be an unbounded Lipschitz domain with 1-periodic structure
(we call it the reference domain).
In other words, if $l^{+}(y)$ denotes the characteristic
function of $\omega$, then $l^{+}$ is a 1-periodic function.
We denote $\varepsilon$-homothetic set
$\{x\in \mathbb{R}^{d}:x/\varepsilon\in\omega\}$
by $\varepsilon \omega$, and so
the function $l_{\varepsilon}^{+}(x)=l^{+}(x/ \varepsilon)$
represents the characteristic function of $\varepsilon\omega$.
Consider the following quasilinear elliptic equations in the divergence form with the mixed boundary conditions,
depending on a parameter $0<\varepsilon\ll 1$,
\begin{eqnarray}\label{pde:1.1}
\left\{\begin{aligned}
\mathcal{L}_{\varepsilon} u_\varepsilon \equiv
-\nabla\cdot A(x/\varepsilon,\nabla u_\varepsilon)
&= F &\qquad &\text{in}~~\Omega_{\varepsilon}, \\
\sigma_{\varepsilon}(u_{\varepsilon})&=0
&\qquad &\text{on}~~K_{\varepsilon},\\
 u_\varepsilon &= g &\qquad & \text{on}~~\Gamma_{\varepsilon},
\end{aligned}\right.
\end{eqnarray}
where $\Omega_{\varepsilon}:=\Omega\cap \varepsilon\omega
$, $\Gamma_{\varepsilon}:=\partial \Omega_{\varepsilon}
\cap\partial \Omega$, $K_{\varepsilon}:
=\partial \Omega_{\varepsilon}\cap\Omega$ and,
$\sigma_{\varepsilon}(u_{\varepsilon})=\vec{n}
\cdot A(x/\varepsilon,\nabla u_{\varepsilon})$ is
known as the conormal derivative of $u_\varepsilon$ on related boundaries.
Given three constants $\mu_0,\mu_1,\mu_2>0$, the function
$A\in C^{1}(\mathbb{R}^d\times\mathbb{R}^d;\mathbb{R}^d)$ and
additionally satisfies
the structure conditions below.
\begin{itemize}
\item[1.] For any $y,\xi,\xi^\prime\in\mathbb{R}^d$,
there hold the \emph{coerciveness} and \emph{growth} conditions
\begin{equation}\label{a:1}
\left\{\begin{aligned}
& \big<A(y,\xi)-A(y,\xi^\prime),\xi-\xi^\prime\big>
\geq\mu_0|\xi-\xi^\prime|^2;\\
&|A(y,\xi)-A(y,\xi^{\prime})|\leq\mu_{1}|\xi-\xi^{\prime}|.
\end{aligned}\right.
\end{equation}
\item[2.] For every $\xi\in\mathbb{R}^d$,
$A(\cdot,\xi)$ is \emph{1-periodic} and
\begin{equation}\label{a:2}
 A(y,0) = 0.
\end{equation}
\item[3.] The \emph{smoothness} assumption is also imposed, i.e.,
\begin{equation}\label{a:3}
|A(y,\xi)-A(y^\prime,\xi)|\leq \mu_2|y-y^\prime|^\tau|\xi|,
\end{equation}
where $\tau\in(0,1]$.
\end{itemize}
(It is not hard to verify that one may take
$A(y,\xi) = \frac{1+|\xi|^2}{1+b(y)|\xi|^2}\xi$ as a nontrivial example, such that it satisfies all the assumptions above, provided $b$ being a 1-periodic function with a suitable smoothness and boundedness assumption.)
%\begin{remark}
%\textcolor[rgb]{1.00,0.00,0.00}{ The effect of assumption \eqref{a:2} here is to make our statement simpler and we can remove it in the paper.}
%\end{remark}
We say $u_{\varepsilon}$ is a weak solution to \eqref{pde:1.1}
if there holds
\begin{equation}\label{pri:1.5}
  \int_{\Omega_{\varepsilon}}A(x/\varepsilon,\nabla u_\varepsilon)\cdot
  \nabla wdx
  =\int_{\Omega_{\varepsilon}} F wdx
\end{equation}
for any $w\in H^{1}(\Omega_{\varepsilon},\Gamma_{\varepsilon})$,
and $u_{\varepsilon}-g\in H^{1}(\Omega_{\varepsilon},
\Gamma_{\varepsilon})$. Here
$H^{1}(\Omega_{\varepsilon},\Gamma_{\varepsilon})$
denotes the closure in $H^{1}(\Omega_{\varepsilon})$
of $C^\infty(\mathbb{R}^d)$
with functions vanishing on $\Gamma_{\varepsilon}$
(see Subsection $\ref{subsec:1.2}$).
Under the assumptions $\eqref{a:1}$ and $\eqref{a:2}$, the existence and
uniqueness of the weak solution to \eqref{pde:1.1}
follows from Browder-Minty's theorem (see for example
\cite[Theorem 26.A]{Z}).
%and it is not hard to verify
%our assumptions on $A$ satisfy a strong monotone,
%hemicontinuous and coercive property
Moreover,
the following qualitative homogenization result had
been shown in V. Zhikov and M. Rychago's work
\cite{Zhikov,ZR}, i.e., there hold
that $l_{\varepsilon}^{+}u_\varepsilon \rightharpoonup u_0$
weakly in $L^2(\Omega)$, and
$l_{\varepsilon}^{+}\nabla u_\varepsilon
\rightharpoonup\nabla u_0$ with
$l_{\varepsilon}^{+}A(x/\varepsilon,
\nabla u_\varepsilon)\rightharpoonup\widehat{A}(\nabla u_0)$
weakly in $L^2(\Omega;\mathbb{R}^d)$.
Here $u_0$ is the solution to the effective (homogenized) equation
\begin{equation}\label{pde:1.3}
\left\{\begin{aligned}
\mathcal{L}_{0} u_0 \equiv
-\nabla\cdot\widehat{A}(\nabla u_0) &= F &\qquad&\text{in}~~\Omega, \\
 u_0 &= g &\qquad& \text{on}~\partial\Omega.
\end{aligned}\right.
\end{equation}
The function $\widehat{A}:\mathbb{R}^d\to\mathbb{R}^d$ is defined for every $\xi\in\mathbb{R}^d$ by
\begin{equation}\label{eq:1.1}
\widehat{A}(\xi) = \theta^{-1}\int_{Y\cap\omega} A(y,\xi+\nabla_y N(y,\xi))dy
\quad\text{and}\quad \theta = |Y\cap\omega|,
\end{equation}
in which $N(y,\xi)$ is the so-called corrector,
associated with the following cell problem:
\begin{equation}\label{pde:1.2}
\left\{\begin{aligned}
& \nabla\cdot A(\cdot,\xi+\nabla N(\cdot,\xi)) = 0 &\quad&\text{in}~ Y\cap \omega,\\
& \vec{n}\cdot A(\cdot,\xi+\nabla N(\cdot,\xi))=0 &\quad&\text{on}~Y\cap\partial\omega,\\
& N(\cdot,\xi)\in H^1_{\text{per}}(Y\cap \omega),
\quad& \dashint_{Y\cap \omega}&N(\cdot,\xi) = 0,
\end{aligned}\right.
\end{equation}
where the notation $\dashint_\Omega := \frac{1}{|\Omega|}\int_\Omega$ represents the
average of integral and $Y=(-\frac{1}{2},\frac{1}{2}]^{d}$.
%The existence and uniqueness of the weak solution to \eqref{pde:1.3} and \eqref{pde:1.2} also follows from
%the properties of $A(y,\xi)$, $\widehat{A}(\xi)$
%(see Lemma \ref{lemma:2.2}).

In order to investigate some quantitative
estimates, we introduce
some geometry assumptions on the reference domain $\omega$  as follows.
\begin{enumerate}
  \item[4. ] \emph{A separated property}. It's assumed that
%$(\mathbb{R}^d\backslash\omega)\cap Y\subset\subset \overline{Y}$
%in which $Y$ is the unit cube,
%and
any two connected components
of $\mathbb{R}^{d}\backslash\omega$
are separated by some positive distance.
Specifically, if $\mathbb{R}^{d}\backslash\omega
=\bigcup_{k=1}^{\infty}H_{k}$ in which $H_k$
is connected and bounded for each k,
then there exists a constant $\mathfrak{g}^{\omega}$ such that
\begin{equation}\label{g}
  0<\mathfrak{g}^{\omega}\leq\inf_{i\neq j}\bigg\{\text{dist}(H_{i},H_j)\bigg\}.
\end{equation}
  \item [5. ] \emph{Regular boundaries}. For each of the components $\{H_k\}$, the boundary of $H_{k}$
is additionally assumed to be $C^{1,\alpha}$ with $\alpha\in(0,1)$,
where the component $H_k$ is usually referred to as a ``hole'' in the context.
\end{enumerate}
Then we call $\omega$ a ``regular'' reference domain, if it satisfies
the above two conditions.

\vspace{0.2cm}

Now, the main results of the paper are stated as following.

\begin{theorem}[convergence rates]\label{thm:1.1}
Let $\omega$ be a regular reference domain.
Suppose that $\mathcal{L}_\varepsilon$ satisfies the conditions
$\eqref{a:1}$, $\eqref{a:2}$ and $\eqref{a:3}$.
Given
$F\in H^{1}(\Omega)$ ,
let
$u_\varepsilon\in H^{1}(\Omega_{\varepsilon})$
and $u_0\in H^1(\Omega)$ be the weak solution to
$\eqref{pde:1.1}$ and $\eqref{pde:1.3}$, respectively.
Then one may obtain the following error estimates.
\begin{itemize}
\item If $g\in H^{3/2}(\partial\Omega)$ and $\Omega$ is a bounded $C^{1,1}$ domain, then there holds
\begin{equation}\label{pri:1.2}
\|u_\varepsilon - u_0\|_{L^2(\Omega_{\varepsilon})}
\leq C\varepsilon^{1/2}\Big\{\|F\|_{H^{1/2}(\Omega)}
+\|g\|_{H^{3/2}(\partial\Omega)}
\Big\}.
\end{equation}
\item If $g\in W^{1-1/p,p}(\partial\Omega)$ for some $0<p-2\ll 1$ and $\Omega$ is a bounded Lipschitz domain, then there exists a Meyer's index $\sigma:=1/2-1/p$, such that,
\begin{equation}\label{pri:1.6}
\|u_\varepsilon - u_0\|_{L^2(\Omega_{\varepsilon})}
\leq C\varepsilon^{\sigma}
\Big\{\|F\|_{H^{\sigma}(\Omega)}
+\|g\|_{W^{1-1/p,p}(\partial\Omega)}
\Big\},
\end{equation}
\end{itemize}
where $C$ depends on $\mu_0,\mu_1,\mu_2,\tau,d,r_0,\mathfrak{g}^{\omega}$
and the boundary character of $\omega$ and $\Omega$.
\end{theorem}

We refer the reader to Subsection $\ref{subsec:1.2}$
for the definition of fractional Sobolev-type spaces such as  $H^{1/2}(\Omega)$, $W^{1-1/p,p}(\partial\Omega)$, as well as,
the notation $r_0$ and ``$\ll$''.
If ignoring the influence caused by the boundary conditions
related to $\partial\Omega$, then we can obtain
the following sharp error estimates.

\begin{theorem}[optimal convergence rates]\label{thm:1.5}
Assume $\omega$ and $\mathcal{L}_\varepsilon$ satisfy
the same conditions as in Theorem $\ref{thm:1.1}$,
while we take $\Omega=\mathbb{R}^d$ here.
Let $0<\lambda\leq \mu_0$.
Given $F\in C^1_0(\mathbb{R}^d)$, let $u_{\varepsilon,\lambda}\in H^{1}(\Omega_\varepsilon)$
and $u_{0,\lambda}\in H^1(\mathbb{R}^d)$ be the weak solutions
to
\begin{equation}\label{pde:1.5}
(\emph{i})\left\{\begin{aligned}
\lambda u_{\varepsilon,\lambda}
-\nabla\cdot A(x/\varepsilon,\nabla
u_{\varepsilon,\lambda})
&= F &\quad& \emph{in}\quad \Omega_\varepsilon;\\
\sigma_\varepsilon(u_{\varepsilon,\lambda})
&=0  &\quad& \emph{on}\quad \partial\Omega_\varepsilon,
\end{aligned}\right.
\qquad
(\emph{ii})~
\lambda u_{0,\lambda} - \nabla\cdot \widehat{A}(\nabla u_{0,\lambda})
= F \quad \emph{in} \quad \mathbb{R}^d,
\end{equation}
respectively.  Then there holds optimal error
estimates:
\begin{itemize}
  \item In the case of $d\geq 3$, we have
  \begin{equation}\label{pri:1.7}
  \|u_{\varepsilon,\lambda}
  -u_{0,\lambda}\|_{L^\frac{2d}{d-2}(\Omega_\varepsilon)}
  \leq C\varepsilon\|\nabla F\|_{L^{\frac{2d}{d+2}}(\mathbb{R}^d)}
  \end{equation}
  where the constant $C$ is independent of $\lambda$.
  \item In the case of $d=2$, we acquire
  \begin{equation}\label{pri:1.8}
  \|u_{\varepsilon,\lambda}-
  u_{0,\lambda}\|_{L^p(\Omega_\varepsilon)}
  \leq C\varepsilon\|F\|_{H^{1}(\mathbb{R}^2)}
  \end{equation}
  for $2\leq p<\infty$, where
  the constant $C$ depends on
$\mu_0,\mu_1,\mu_2, \lambda, \tau, p$
and the boundary character of $\omega$.
\end{itemize}
\end{theorem}

\begin{remark}
\emph{Compared to the case of unperforated domains,
the regularity of the source term $F$ has a significant impact on the power of convergence rate
(see also \cite{WXZ2020} for linear systems). In other words,
the norms $\|F\|_{H^{1/2}(\Omega)}$ in $\eqref{pri:1.2}$ and $\|F\|_{H^{\sigma}(\Omega)}$ in $\eqref{pri:1.6}$
can not be weaken in terms of smooth index of the
Sobolev space, otherwise we will loss the power of the rate. However, from the estimate $\eqref{pri:1.7}$,
one may believe that its integral index has a potential chance to be improved. Besides, letting $\lambda\to 0$ in
$\eqref{pri:1.7}$ one may derive that
  \begin{equation*}
  \|u_{\varepsilon}
  -u_{0}\|_{L^\frac{2d}{d-2}(\Omega_\varepsilon)}
  \leq C\varepsilon\|\nabla F\|_{L^{\frac{2d}{d+2}}(\mathbb{R}^d)}
  \end{equation*}
where $u_{\varepsilon}$, $u_0$ (up to a constant) uniquely
solve the equation: $-\nabla\cdot A(x/\varepsilon,\nabla
u_{\varepsilon})= F$ in $\Omega_\varepsilon$ with
$\sigma_\varepsilon(u_{\varepsilon}) = 0$ on $\partial(\varepsilon\omega)$, and the effective one
$- \nabla\cdot \widehat{A}(\nabla u_{0})
= F$ in $\mathbb{R}^d$, respectively. This result is new
even for linear equations with variable coefficients.
To our best knowledge, the same type result was merely acquired for linear elliptic equations with constant coefficients in \cite[Theorem 2.4]{Y} via fundamental solution arguments.}
\end{remark}

\begin{remark}
\emph{Roughly speaking,
the frame work of Theorems $\ref{thm:1.1}$ and $\ref{thm:1.5}$ is based upon energy estimates.
Inevitably, the phenomenon of boundary layer will be generated in the calculations, which means we have to
handle the quantities like $\varepsilon\|N(\cdot/\varepsilon;\xi)\|_{H^{1/2}(\partial\Omega)}$ or $\|\nabla u_0\|_{L^2(O_\varepsilon)}$ (where
$O_\varepsilon$ is the boundary layer set of $\Omega$
defined in Subsection $\ref{subsec:1.2}$), which merely
offer us $O(\varepsilon^{1/2})$ at most.
In general, for linear equations, people may employ a duality argument to accelerate the convergence rate to the sharp one
(see for example \cite{S5,TS1,X3}), which is also known as
the Aubin-Nitsche's approach in numerical fields.
However, successfully applying this idea to homogenization problems on perforated domains involves more advanced techniques in analysis (see \cite{WXZ2020}). Concerning
the nonlinear model $\eqref{pde:1.1}$, a possible way to
get the sharp convergence rate is appealing to maximum
principles for divergence operators (see for example
\cite[Section 10.5]{GT}), while
this approach merely holds for scalar equations. Without
a proof, we claim that under the conditions
$g\in C^{1,1}(\partial\Omega)$ and $F\in L^p(\Omega)\cap H^1(\Omega)$ with $p>d$, there holds
\begin{equation*}
 \|u_\varepsilon - u_0\|_{L^q(\Omega_\varepsilon)}
 \lesssim \varepsilon\Big\{\|F\|_{H^1(\Omega)}
 +\|F\|_{L^p(\Omega)}+\|g\|_{C^{1,1}(\partial\Omega)}\Big\},
\end{equation*}
where $q=2d/d-2$ if $d\geq 3$; $2\leq q<\infty$ if $d=2$.
It is still an interesting question whether there is a method that does not depend on maximum
principles to obtain the best error estimates of the nonlinear model $\eqref{pde:1.1}$ even when $\omega=\mathbb{R}^d$. In other words, our present results
 rely on the so-called De Giorgi-Nash-Moser theory heavily.}
\end{remark}

Then we turn to the regularity estimates of weak solutions.

\begin{theorem}[interior Lipschitz estimates at large-scales]
\label{thm:1.2}
Let $B_2\subset\Omega$.
Suppose that $\mathcal{L}_\varepsilon$ satisfies the same conditions as in Theorem $\ref{thm:1.1}$.
Let $u_\varepsilon\in H^1(B_2^\varepsilon)$ be a weak solution of \begin{equation}\label{pde:7.3}
 \left\{ \begin{aligned}
 \mathcal{L}_{\varepsilon}u_\varepsilon  &= 0 ~&&\emph{in}~~ B^{\varepsilon}_2,\\
 \sigma_{\varepsilon}(u_{\varepsilon}) &= 0 ~&&\emph{on}~~
 \partial B^{\varepsilon}_2|_{B_2},
  \end{aligned}\right.
\end{equation}
where $B_2^\varepsilon:=B_2\cap(\varepsilon\omega)$ and
$\partial B_2^\varepsilon|_{B_2}:=\partial B_2^\varepsilon\cap B_2$ (see Subsection $\ref{subsec:1.2}$).
Then one may derive that
\begin{equation}\label{pri:1.1}
\begin{aligned}
\Big(\dashint_{B^{\varepsilon}_r} |\nabla u_\varepsilon|^2\Big)^{\frac{1}{2}}
\lesssim\Big(\dashint_{B^{\varepsilon}_1} |\nabla u_\varepsilon|^2\Big)^{\frac{1}{2}}
\end{aligned}
\end{equation}
for any $\varepsilon\leq r\leq (1/2)$,
where the up to constant depends on $\mu_0,\mu_1,\mu_2,\tau,d,\mathfrak{g}^{\omega}$
and the boundary character of $\omega$.
\end{theorem}

Here, the notation $\lesssim$ means $\leq$ up to a multiplicative constant
and we usually call it ``the up to constant'' (see Subsection $\ref{subsec:1.2}$).

In terms of mixing boundary problems, there is no pointwise  $C^{1,\alpha}$ estimates near boundaries without any geometry
assumption of the interface, even though we
assume $\varepsilon=1$ and the equations $\eqref{pde:1.1}$
become a linear one with smooth coefficients.
Here, the boundaries of $\Omega_\varepsilon$ near $\partial\Omega$ would be even worse as $\varepsilon$ varies.
However, one may derive the following large-scale estimates
as substitutions.

\begin{theorem}[boundary Lipschitz estimates at large-scales]
\label{thm:1.3}
Let $0<\varepsilon\ll1$ and $\Omega$ be a bounded $C^{1,1}$ domain. Suppose that $\mathcal{L}_\varepsilon$ and $\omega$  satisfy the same conditions as in Theorem $\ref{thm:1.1}$.
Let $u_\varepsilon\in H^1(D_4^\varepsilon)$
be a weak solution of
\begin{equation}\label{pde:7.4}
\left\{\begin{aligned}
\mathcal{L}_\varepsilon u_\varepsilon &= 0&\qquad&\emph{in}~~ D^{\varepsilon}_4,\\
\sigma_{\varepsilon}(u_{\varepsilon})&= 0&\qquad&\emph{on}~~ \partial D_4^\varepsilon|_{D_4},\\
u_{\varepsilon}&= 0&\qquad&\emph{on}~~
\partial D_4^\varepsilon|_{\Delta_4},
\end{aligned}\right.
\end{equation}
where the notation $D_{4}^{\varepsilon},
\partial D_{4}^\varepsilon|_{\{D_4\text{~or~}\Delta_{4}\}}$ are referred to Subsection $\ref{subsec:1.2}$. Then there holds
\begin{equation}
\begin{aligned}\label{pri:1.14}
\Big(\dashint_{D^{\varepsilon}_r} |\nabla u_\varepsilon|^2dx\Big)^{\frac{1}{2}}
\lesssim\Big(\dashint_{D^{\varepsilon}_{1}} |\nabla u_\varepsilon|^2dx\Big)^{\frac{1}{2}}
\end{aligned}
\end{equation}
for any $\varepsilon\leq r\leq 1/2$, where
the up to constant additionally relies on
the boundary character of $\Omega$ compared to
the counterpart in $\eqref{pri:1.1}$.
\end{theorem}

\begin{theorem}[quenched Calder\'{o}n-Zygmund estimates]\label{thm:1.4}
Let $2\leq p<\infty$.
Let $0<\varepsilon\ll1$ and $\Omega$ be a bounded $C^{1,1}$ domain. Assume that $\mathcal{L}_{\varepsilon}$ and $\omega$ satisfy the same hypothesises as in Theorem \ref{thm:1.1}.
For any $f\in L^{p}(\Omega;\mathbb{R}^d)$, suppose  that $u_{\varepsilon}$ is the weak solution to
\begin{equation}\label{pri:1.15}
\left\{\begin{aligned} \mathcal{L}_{\varepsilon}u_{\varepsilon}&=\nabla\cdot f&\qquad& \emph{in}\quad \Omega_{\varepsilon},\\ \sigma_{\varepsilon}(u_{\varepsilon})&=-\vec{n}\cdot f&\qquad& \emph{on}\quad K_{\varepsilon},\\
u_{\varepsilon}&=0&\qquad&\emph{on}\quad\Gamma_{\varepsilon}.
\end{aligned}\right.
\end{equation}
Then there holds
\begin{equation}\label{pri:1.16}
\bigg(\int_{\Omega}\Big(\dashint_{B(x,\varepsilon)
\cap\Omega_{\varepsilon}}|\nabla u_{\varepsilon}|^2\Big)^{\frac{p}{2}}dx\bigg)^{\frac{1}{p}}
\lesssim\bigg(\int_{\Omega}\Big(\dashint_{B(x,\varepsilon)
\cap\Omega_\varepsilon}|f|^2\Big)^{\frac{p}{2}}dx\bigg)^{\frac{1}{p}},
\end{equation}
where the up to constant depends on $\mu_0,\mu_1,\mu_2,\tau, d,r_0, \mathfrak{g}^{\omega},p$ and the characters of $\Omega$ and $\omega$.
\end{theorem}

\begin{remark}\label{remark:1.1}
\emph{Concerned with optimal uniform regularity estimates,
the only possibility is to derive the ``interior'' uniform
Lipschitz estimate for the weak solution to $\eqref{pde:7.3}$, i.e.,
\begin{equation*}
  |\nabla u_\varepsilon(x)|
  \lesssim \Big(\dashint_{B^{\varepsilon}_{r}(x)}
  |u_\varepsilon|^2\Big)^{1/2}
\end{equation*}
for any $x\in B^{\varepsilon}_{1}$ and $0<r<1/2$, since
the boundary of $\omega$ owns a good regularity.}
\end{remark}

\begin{remark}
\emph{One may scale $\eqref{pde:1.1}$ to the case $\varepsilon =1$, and
denote its solution $u_1$ in such the case by $u$.
Let us explain the relationship between small-scale estimates (based on perturbations) and large-scale estimates (appealing to homogenization) as follows.
\begin{equation*}
\begin{aligned}
\text{local~smoothness~of~operators}\quad
&\Longrightarrow^{\text{classical~regularity~theory}} \quad \text{small-scale~regularities~of~} u
& &(0<r<1);\\
\text{regularities~of~} u_0
\quad
&\Longleftrightarrow^{\text{homogenization theory}}
\quad \text{large-scale~regularities~of~} u
& &(1\leq r<\infty).
\end{aligned}
\end{equation*}
This picture tells us that large-scale regularities of
$u$ can be good enough, provided homogenized solution $u_0$
being sufficiently smooth, and have no business with
perturbation arguments at small-scales. That is the main
reason why we can still investigate some regularity estimates for weak solutions (such as Theorems $\ref{thm:1.3}$ and
$\ref{thm:1.4}$) even when the known behaviour of
the solution would be ill-posedness measured by some
little stronger norms at small-scales. Also,
from the relationship,
it is not hard to understand why we use an integral average to replace pointwise function in $L^p$-norms to reformulate
Calder\'{o}n-Zygmund estimates as stated in Theorem $\ref{thm:1.4}$.}
\end{remark}

\begin{remark}
\emph{The condition of $\partial\Omega\in C^{1,1}$ in Theorems
$\ref{thm:1.3}$ and $\ref{thm:1.4}$ seems to be
strange at first glance, and a reasonable one should be
$\partial\Omega\in C^{1,\eta}$ with $0<\eta<1$ since
it is exactly the content of Schauder's theory for both
linear and quasi-linear elliptic operators.
However, it revealed a subtle difference between
the linear and nonlinear homogenization problems, which
is whether the homogenized solution $u_0$ still owns a much better regularity.
Regarding to our nonlinear model $\eqref{pde:1.1}$,
we should remind readers about the difference between $\widehat{A}(\cdot)$ and $A(x_0,\cdot)$.
Merely the Lipschitz continuity of $\widehat{A}$ was known (see Lemma $\ref{lemma:2.2}$ and Remark $\ref{remark:2.4}$), and so we just infer that $\nabla u_0\in C^{1,\alpha^\prime}_{\text{loc}}(\Omega)$ for some $\alpha^\prime\in(0,1)$
via linearization coupled with H\"older estimates.
(see Theorems $\ref{lemma:4.2}$, $\ref{thm:2.19}$).
As a comparison, for each fixed $\varepsilon$, it follows  from Schauder's estimates that $u_\varepsilon\in C_{\text{loc}}^{1,\gamma}
(\Omega^\prime_\varepsilon)$ for
any $0<\gamma\leq \alpha$ with $\Omega^\prime\subset\subset\Omega$. Hence, the higher regularity assumption of $\partial\Omega$ is exactly caused by the operation of linearization, which is not necessary if one may confirm that $\widehat{A}(\cdot)$ and $A(x_0,\cdot)$ own the same level smoothness. }
\end{remark}

\subsection{Related to the geometrical assumptions on $\omega$ and smoothness assumption $\eqref{a:3}$}

Compared to homogenization problems on unperforated domains,
the the difficulties arose from perforated domains are essential.
For example, let $A(y,\xi) = \xi$ and then the related corrector
is given by $N(y,\xi) = \phi\cdot\xi$, while it is not hard to check
that the equation $\eqref{pde:1.2}$ may be reduced to
\begin{equation}\label{pde:1.4}
 -\Delta \phi_k = 0  \quad\text{in}~Y\cap\omega,
 \qquad\text{and}\quad
 \vec{n}\cdot\nabla\phi_k = -n_k
 \quad \text{on}\quad Y\cap\partial\omega,
\end{equation}
where $n_k$ is the $k^{\text{th}}$ component of $\vec{n}$,
and $k=1,\cdots, d$.
%If we impose the condition
%$\dashint_{Y\cap\omega}\phi_k = 0$, then there exists a unique
%nontrivial weak solution to $\eqref{pde:1.4}$ (note that $\phi_k = 0$ if $\omega =\mathbb{R}^d$).
Clearly, one may observe that
$\int_{Y\cap\omega}\nabla \phi_k dy \not= 0$, which
will bring in some new influence to the process of homogenization,
mostly coming from the geometry of $\omega$.
%If going back to macroscopic scales,
%one of the dangerous operations would be eager to employ
%the trace theorem on $K_\varepsilon$, because of its
%dependence of the size of the ``holes''.
%Even worse, as $\varepsilon$ approaches to zero,
%too many ``holes'' also lead to a severe problem
%especially when
%repeating operations around them however fail to have
%a cancelation.
%A fortunate remedy seems to be a careful extension argument, which
%just relies on the character of the boundary $\partial \omega$.
%In general, it is proved to be the fundamental idea
%for homogenization on perforated domains, such as
%the extension theorem developed by
%E. Acervi, V. Piat, G. Maso and D. Percivale
%\cite[Theorem 2.1]{APMP} and
%by O. Oleinik,  A. Shamaev, G. Yosifian \cite[Theorem 4.3]{OSY},
%which ultimately together the new improvements makes our previous framework in \cite{WXZ}
%workable to the present model.
In this connection, we would like to address
some specific difficulties.
\begin{enumerate}
  \item[(i).] The fact that
$\int_{Y\cap\omega}\nabla N(\cdot,\xi) dy \not= 0$
prevents us from simply repeating the proof used in
\cite[Lemma 2.3]{WXZ} or \cite[Lemma 1]{P} to prove
the coercive property of
$\widehat{A}$ (see Lemma $\ref{lemma:2.2}$), while this property
plays a crucial role
in the quantitative homogenization theory as we have
explained in \cite{WXZ,W} with details.
Given its importance, we employ the extension theorem
developed in \cite[Thoerem 4.2]{OSY} (or
 \cite[Theorem 2.1]{APMP} in the case of
 $\partial Y\cap(\mathbb{R}^d\setminus\omega)=\emptyset$) to show
a clear proof for this property,
inspired by a similar result stated in \cite{PR,ZR}.
\emph{This difficulty
can not be observed from the linear models,
such as the example mentioned above.}
  \item[(ii).] For later two-scale expansions, we will impose
  an composite function $N(x/\varepsilon,\varphi)$
  with $\varphi\in H_0^1(\Omega;\mathbb{R}^d)$, which
  may wreck the periodicity of $N(\cdot,\xi)$ for any fixed $\xi$.
  This loss causes that we can not use the so-called periodic
  cancellation (which is quite useful to error estimates), i.e.,
\begin{equation}\label{KEY}
  \|\varpi(\cdot/\varepsilon)f\|_{L^{2}(\Omega)}
  \leq C\|\varpi\|_{L^{2}(Y)}\|f\|_{L^2(\Omega)}
  + o(1),
  \qquad \text{as}~\varepsilon\to 0,
\end{equation}
where $\varpi\in L^2_{\text{per}}(Y)$ and $f\in
C(\bar{\Omega})$.
Because of this, we have to show
$\nabla_\xi N(y,\xi)\in
L^\infty((Y\cap\omega)\times\mathbb{R}^d)$.
Our argument relies on the local boundedness estimate
coupled with the weak Harnack inequality,
originally developed by L. Caffarelli \cite{C}
for unperforated settings.
%Since the behavior of $N(y,\xi)$
%near $\partial\omega$ is also involved, we provide
%the boundary estimates in Lemma $\ref{lemma:7.1}$ for
%the completeness.
%Until now, we do not require any smoothness
%assumption on $A$ with respective to the first variable.
Moreover, the imposed flux corrector $E$
(see Lemma $\ref{lemma:2.55}$) will confront with the same
problem when $E$ and $\varphi$ are composed to be the form of
$E(\cdot/\varepsilon,\varphi)$.
This consequently requires
%a uniform $L^p$-bound
%of the quantity $\frac{\nabla N(\cdot,\xi)-\nabla
%N(\cdot,\xi^\prime)}
%{|\xi-\xi^\prime|}$ on the region $Y\cap\omega$
%with $p\geq 2$ for any
%$\xi,\xi^\prime\in\mathbb{R}^d$ (see Lemma $\ref{lemma:2.55}$),
%and therefore
the smoothness assumptions on $A(\cdot,\xi)$ and $\partial\omega$
(see also Remark $\ref{remark:2.5}$). \emph{In the end, we should warn that
Lipschitz estimates for $u_\varepsilon$ near $\partial\Omega$ at small-scales can not be guaranteed
by the assumption $\eqref{a:3}$ and the geometrical assumptions on $\omega$.}
\end{enumerate}

\begin{remark}
\emph{For the linear case, the smoothness assumption $\eqref{a:3}$
and hypothesises of regular boundaries of $\omega$ are not
necessary (see \cite{WXZ2020}). Thus, there is an interesting question whether we can establish a new theory independent of
these two conditions.}
\end{remark}

\subsection{Related to fractional
Sobolev-type spaces in error estimates}\label{subsec:1.3}

To clearly observe the impact caused by the regularity of $F$, we impose the fractional
Sobolev-type space for a description.
Here we merely expose where it involved and the reason why it looks reasonable.
As in \cite{WXZ}, define the first-order approximating corrector $w_\varepsilon := u_\varepsilon - u_0 - \varepsilon N(y,\varphi)$ with $\varphi\in H^1_0(\Omega;\mathbb{R}^d)$.
Then, figure out the following weak formulation that $w_\varepsilon$ satisfies
\begin{equation}\label{f:1.2}
\begin{aligned}
 \int_{\Omega_{\varepsilon}} \big(A(x/\varepsilon,\nabla u_\varepsilon) - A(x/\varepsilon,\nabla v_\varepsilon)\big)
\cdot \nabla w_{\varepsilon} dx
& = \underbrace{\int_\Omega(l_{\varepsilon}^{+}-
\theta\psi^{'}_{\varepsilon})
F\tilde{w}_{\varepsilon}dx}_{\text{T}}
+ \int_{\Omega} \text{~``traditional terms''~} dx
\end{aligned}
\end{equation}
(see the equality $\eqref{eq:3.1}$ for details),
where $\tilde{w}_{\varepsilon}$ is the extension of
$w_{\varepsilon}$ given by Lemma $\ref{extensiontheory}$
and $\psi_{\varepsilon}^\prime$ is a cut-off function
satisfying $\eqref{cutoff}$.
Since ``traditional terms'' could be handled by
a similar argument as in previous work \cite{WXZ,W},
this part is clearly no business with the fractional
Sobolev-type spaces. In fact,
appealing to the auxiliary equation $\eqref{auxi1}$, the term $\text{T}$ will
produce the term like
\begin{equation*}
\varepsilon\int_{\Omega}\nabla F\cdot\nabla_y\Phi(y)(\tilde{w}_\varepsilon\psi_\varepsilon^\prime)dx
\quad\text{with}\quad
y=x/\varepsilon.
\end{equation*}
By Lemma $\ref{interpolation}$, the fractional Sobolev-type spaces therefore play a role
in judging the optimal power of $\varepsilon$.

%\begin{remark}
%The main ideas on error estimates for
%nonlinear elliptic homogenization problems could be found in \cite{WXZ,W}. Roughly speaking, we also follow the
%same strategies here. However, the subtle difference
%should be emphasised. The weak formulations
%in Lemma $\ref{lemma:3.1}$ and $\ref{lemma:3.2}$ are not
%born from the definition of weak solutions in $H^1(\Omega_\varepsilon,\Gamma_\varepsilon)$
%(see Subsection $\ref{subsec:1.1}$). Instead, we treat
%the equalities in $\eqref{pde:1.1}$ and $\eqref{pde:1.3}$
%in the sense of functional spaces of $C^1_0(\Omega)$, which is in fact much larger than the dual of $H^1(\Omega_\varepsilon,\Gamma_\varepsilon)$.
%\end{remark}

\subsection{Related to previous works $\&$ Innovations of
present jobs}

The large-scale (uniform) Lipschitz regularity was first
obtained by M. Avellaneda, F. Lin \cite{MAFHL} through
three-step compactness methods for periodic homogenization problems.
Recently, S. Armstrong, T. Kuusi, J.-C. Mourrat, Z. Shen
\cite{AKM,AS}
created a new approach in
aperiodic settings. In this regard,
a fair statement should not ignore
A. Gloria, S. Neukamm, F. Otto's work
\cite{GNO1}, although
theirs formally published quite recently.
Due to the remarkable developments above,
frankly speaking, large-scale regularity estimates in homogenization theory have already been understood deeply,
and some potential contributions of the present work
would be fulfilling some technical gap between main ideas and
concrete problems in the new background.

In terms of our model $\eqref{pde:1.1}$ in linear cases, the first notable outcome was
obtained by O. Oleinik, A. Shamaev, and G. Yosifian
\cite[pp.124, Theorem 1.2]{OSY} for linear elasticity systems:
\begin{equation}\label{f:1.3}
  \|u_{\varepsilon}-u_{0}-\varepsilon\chi_{\varepsilon}
  \nabla u_{0}\|_{H^{1}(\Omega_{\varepsilon})}
  \lesssim\varepsilon^{1/2}
  \Big\{\|g\|_{H^{5/2}(\partial\Omega)}
  +\|F\|_{H^{1}(\Omega)}\Big\},
\end{equation}
under regularity assumptions on the
coefficients and the reference region $\omega$.
Recently, B. Russell \cite[Theorem 1.4]{BR} improved the
above estimate by receiving
\begin{equation*}
 \|u_{\varepsilon}-u_{0}-
 \varepsilon\chi_{\varepsilon}S_\varepsilon^2
 (\psi_\varepsilon\nabla u_{0})\|_{H^{1}(\Omega_{\varepsilon})}\lesssim \varepsilon^{1/2}
  \|g\|_{H^{1}(\partial\Omega)}
\end{equation*}
for Lipschitz domains without regularity assumptions on the coefficients. Meanwhile,
an interior large-scale Lipschitz regularity
(see \cite[Theorem 1.1]{BR}) was also established.
From our point of view, their impressive contributions
are summarized below, \emph{from which readers will see the source of innovation of this job}.
\begin{itemize}
  \item The literature \cite{OSY} developed
  some extension theorems on perforated domains.
  \emph{Their core ideas stimulated the creation of Lemma $\ref{extensiontheory}$ to
  shift our analysis from $L^2$ to $L^p$ spaces.
  Also, we found Sobolev-Poincar\'e inequalities on perforated domains (see Lemma $\ref{lemma:2.20}$),
  which opened a door to Meyer's estimates (see Theorem $\ref{thm:7.1}$) and therefore played a fundamental role in the whole work.}

  \item The so-called flux corrector was introduced
  by \cite{OSY,BR} in a different format, and the later one
  first extended the corrector $\chi$ from $\omega\cap Y$ to $Y$ and then defined flux correctors on $Y$.
  \emph{This way seems to be easily extended to nonlinear equations according to our previous study experience in \cite{WXZ} (see Lemma $\ref{lemma:2.5}$). Also, it is efficient to observe the
  relationship between the regularity of $F$ and the convergence rate, by help of Lemma $\ref{interpolation}$, and we have explained it more in Subsection $\ref{subsec:1.3}$.}

  \item With regard to the estimate $\eqref{pri:1.1}$ for
  linear equations, the literature \cite{BR} developed some techniques to make the scheme on large-scale regularities
  in \cite{S5} valid for perforated domains, such as
  \cite[Lemma 4.4]{BR}.
  \emph{Although our approach is based upon Meyer's estimates (due to the new development mentioned above),
  his main idea  provided us with a blueprint to
  Theorem $\ref{thm:1.2}$ and $\ref{thm:1.3}$, making
  us focus on the new challenges caused by boundaries and
  nonlinearity of operators.
  In this regard, we strongly refer the reader to the proofs in Sections $\ref{sec:4}, \ref{sec:5}$ for these new tricks
  which are quite involved and not suitable to be presented here. For example, see the proof of Lemma $\ref{lemma:5.2}$.}
\end{itemize}

Besides, under higher regularity assumption on given data,
A. Belyaev, A. Pyatnitski\u{1} and G. Chechkin's \cite{ABAPGC} developed another approach to derive error estimates, whose method is independent of flux correctors.
However, their scheme seems to be hardly extended to nonlinear cases. Recently, under fast decay of correlations at large-scales and stationary ensemble, J. Fishcher and S. Neukamm obtained  optimal convergence rates for the model $\eqref{pde:1.5}$ in the case of $\omega=\mathbb{R}^d$
(see \cite[Theorems 2,3]{FN}).
For linear models with high-contrast coefficients, Z. Shen
developed a new scheme (intrinsic way) to derive interior Lipschitz estimates at large-scales (see \cite{Shen}).

Concerning the quenched Calder\'{o}n-Zygmund estimates,
it is initially appeared in
A. Gloria, S. Neukamm, F. Otto's work
\cite{GNO1} for a quantitative
stochastic homogenization theory.
For nonlinear elliptic type equations (on unperforated domains), interior quenched Calder\'{o}n-Zygmund estimates  was received by S. Armstrong, J.-P. Daniel \cite{AD}.
Recently, as an intermediate step, it
was developed for elliptic systems with stationary random coefficients of integrable correlations by
M. Duerinckx and F. Otto \cite{DO}.
In their scheme, the notable ingredient was Shen's real arguments (see Lemma $\ref{lemma:6.1}$), which inspired by L. Caffarelli and I. Peral's work \cite{CP},
and we also take this way to establish Theorem $\ref{thm:1.4}$.
As for traditional Calder\'{o}n-Zygmund theory of singular integral operators, one may acquire weighted-type estimates
appealing to Muckenhoupt's weight classes, and then
we can apply it to improve the estimate
$\eqref{f:1.3}$ to a sharp one for linear models,
whose methods were presented in a separated work \cite{WXZ2020}.

Finally, we mention that
%our previous work for the equation
%$\eqref{pde:1.1}$ on unperforated domains \cite{WXZ}
%included some valuable point of views, from which the  current research benefitted a lot.
the quantitative homogenization theory has been received an extensive study,
and without attempting to
exhaustive we refer the readers to
\cite{AKM,ASZ,AS,FN,GNO,J,KFS2,NSX,P,BR,S4,S5,SZ1,
TS1,X3,ZVVPSE1}.

\subsection{Organization of the paper}

In Section $\ref{sec:2}$, we first introduced
some quantitative properties of correctors both in average and pointwise senses (see Lemmas $\ref{lemma:2.1}, \ref{lemma:2.55}$), and then we verified
the growth and coerciveness properties of $\widehat{A}$ in
Lemma $\ref{lemma:2.2}$, as well as, some
estimates for flux correctors in Lemma $\ref{lemma:2.5}$.
In Subsection $\ref{subsec:2.2}$,
we introduced some properties of periodic cancellations
and built the boundedness of the operator defined by periodic multiplier in fractional Sobolev-type spaces
(see Lemma $\ref{interpolation}$).
Subsection $\ref{subsec:2.3}$
was devoted to some extension theories (see Lemma
$\ref{extensiontheory}$), and we also established the Sobolev Poincar\'e's inequality on perforated domains there.
To make Preliminaries concise, we separately showed all the related proofs in Section $\ref{sec:8}$.

In Section $\ref{sec:3}$, the main task was to build
the weak formulation of the first-order approximating correctors. The first subsection was
devoted to show the proof of Theorem $\ref{thm:1.1}$,
while the proof of Theorem $\ref{thm:1.5}$
was signed to the second subsection. Sections $\ref{sec:4}$
and $\ref{sec:5}$ handled the interior and
boundary Lipschitz estimates at large-scales, respectively.
Some new tricks had been shown with full details.
In Section $\ref{sec:6}$, we first introduced Shen's real arguments and primary geometry on integrals, and then
had the proof of Theorem $\ref{thm:1.4}$.
%In Section $\ref{sec:8}$, we show the proofs of lemmas
%stated in Section $\ref{sec:2}$.

In Section $\ref{sec:7}$, some fundamental regularity estimates were imposed. Although these results must be
known by experts, we provided the proofs for the reader's
convenience due to the lack of precise references.

\subsection{Notation}\label{subsec:1.2}
\begin{enumerate}
  \item [(1).] Notation for estimates.
  \begin{enumerate}
    \item  $\lesssim$ stand for $\leq$ up to a multiplicative constant, which may depend on some given parameters imposed in the paper, but never on $\varepsilon$. We write $\sim$ when both $\lesssim$ and
    $\gtrsim$ hold.
    \item We use $\ll$ to indicate that the multiplicative constant is much small than 1.

    \item  $\lesssim^{(*)}, =^{(*)}$ or $\leq^{(*)}$ denotes that the equality or inequality follows from $(*)$.
  \end{enumerate}
  \item [(2).] Geometric notation
  \begin{enumerate}
    \item  $d\geq 2$ is the dimension,
    $r_0$ represents the diameter of $\Omega\subset\mathbb{R}^d$, and
    $\langle x,y\rangle:=\sum_{i=1}^dx_iy_i$ represents
    the inner product in $\mathbb{R}^d$.
    The layer set of $\Omega$ is denoted by $O_{n\varepsilon}:=\{x\in\Omega:\text{dist}(x,\partial\Omega)
        <n\varepsilon\}$, and the co-layer set is defined by $\Sigma_{n\varepsilon}:=\Omega\backslash O_{n\varepsilon}$.
    \item  Let $\vartheta:\mathbb{R}^{d-1}\to\mathbb{R}$ be a Lipschitz function (or $C^{1,\eta}$ function with $0<\eta\leq 1$) such that $\vartheta(0)=0$ and $\|\nabla\vartheta\|_{L^{\infty}(\mathbb{R}^{d-1})}\leq M_{0}$ (or $\|\nabla\vartheta\|_{C^{0,\eta}(\mathbb{R}^{d-1})}\leq M_{0}$). For any $r>0$, let
        \begin{equation*}
        \begin{aligned}
          \Delta_{r}&:=\big\{(x',\vartheta(x'))\in
          \mathbb{R}^d:|x'|<r\big\};\\
          D_r&:=\big\{(x',t)\in\mathbb{R}^d:|x'|<r \quad\text{and}\quad
          \vartheta(x')<t<\vartheta(x')+10(M_0+1)r\big\}.
        \end{aligned}
        \end{equation*}
        Up to a diffeomorphism, one may simply write $D_r=B(z,r)\cap\Omega$; $\Delta_r=B(z,r)\cap\partial\Omega$ with $z\in\partial\Omega$.

  \item $D^\varepsilon_r:=D_r\cap\varepsilon\omega$ (half balls with holes); $\partial D^\varepsilon_r:= (\partial D_r\cap\varepsilon\omega)\cup
      (D_r\cap\partial(\varepsilon\omega))$ (boundaries of $D_r^\varepsilon$); Then one may define the set
      of $\partial D_r^\varepsilon$ intersecting with any  set $U$, denoted by
      $\partial D^{\varepsilon}_r|_{U}:=\partial D^{\varepsilon}_r\cap U$.
      In this regard,
%      the ``interior'' boundary part of $\partial D_r^\varepsilon$ is denoted by
%      $\partial D_r^\varepsilon|_{D_r}$. Therefore, we have
%      $\partial D_r^\varepsilon =
%      \partial D_r^{\varepsilon}|_{D_r}\cup
%      \partial D_r^{\varepsilon}|_{\partial D_r}$.
%      If we define $\Delta^\varepsilon_r:=\Delta_r\cap(\varepsilon\omega)$ as the ``bottom'' boundary part of $\partial D^\varepsilon_r$,
%      and then $\Delta^\varepsilon_r=(\partial\Omega\cap
%      \partial D_r)\cap(\varepsilon\omega)
%      =(\partial\Omega\cap\partial D_r)\cap(\partial D_r\cap\varepsilon\omega)
%      =\partial D_r^\varepsilon|_{\partial\Omega\cap\partial D_r}
%      =\partial D_r^\varepsilon|_{\Delta_r}$. Thus, one
%      may regard $\partial D^{\varepsilon}_r|_{\partial D_r\setminus\Delta_r}$ as the ``upper'' boundary of
%      $\partial D_r^\varepsilon$.
%      Now,
      the boundary set
%      $\partial D_r^\varepsilon$ consists of ``interior'',
%      ``bottom'' and ``upper'' boundaries, i.e.,
      $\partial D_r^\varepsilon
      =\partial D_r^{\varepsilon}|_{D_r}
      \cup\partial D_r^{\varepsilon}|_{\Delta_r}
      \cup\partial D_r^{\varepsilon}|_{\partial D_r\setminus\Delta_r}$. If the half ball $D_r^\varepsilon$ is centered at $x$, then one may  denote it by $D_r^\varepsilon(x)$, whose center would be omitted without confusions in general.

  \item Let $B_r:= B(0,r)$ and
  $nB=nB_r=B(0,nr)$; $B^\varepsilon_r:=B(0,r)\cap\varepsilon\omega$.
  Then the boundary set $\partial B_r^\varepsilon
  = \partial B_r^\varepsilon|_{\partial B_r}
  \cup \partial B_r^\varepsilon|_{B_r}$.
%  It is not hard to see that
%  $\partial B_r^\varepsilon|_{B_r}=B(0,r)\cap\partial(\varepsilon\omega)$,
%  and $\partial B_r^\varepsilon|_{\partial B_r}=\partial B(0,r)\cap (\varepsilon\omega)$.
  The notation
  $B_r^\varepsilon(x)$ is to stress the center point $x$   (otherwise omitted).
  \end{enumerate}
  \item [(3).] Notation for spaces and functions.
  \begin{enumerate}
    \item  $H^1(\Omega_\varepsilon,\Gamma_\varepsilon)$
        denotes the closure in $H^1(\Omega_\varepsilon)$
        of smooth functions vanishing on
        $\Gamma_\varepsilon$ (see \cite[pp.3]{OSY}).
        (Similarly, one may have the notation
        $H^{1}(B^\varepsilon_r,\partial B_r^\varepsilon|_{\partial B_r})$,
        $H^{1}(D^\varepsilon_r,\partial D_r^\varepsilon|_{\partial D_r})$ for $r>0$ and $W^{1,p}(\Omega_{\varepsilon},\Gamma_{\varepsilon})$)
        $H^1_{\text{per}}(Y\cap \omega)$ represents
        the closure in $H^1(Y\cap\omega)$ of the
        set of 1-periodic $C^{\infty}(\bar{\omega})$
        functions (see \cite[pp.5]{OSY}).

    \item
    We impose the fractional
Sobolev-type spaces.
%which are also known as
%Aronszajn, Gagliardo or Slobodeckij spaces.
For
$s\in (0,1)$, we define $H^s(\Omega)$ as follows,
\begin{equation*}
  H^s(\Omega) := \bigg\{u\in L^2(\Omega):
  \frac{|u(x)-u(y)|}{|x-y|^{s+\frac{d}{2}}}\in L^2(\Omega\times\Omega)\bigg\},
\end{equation*}
endowed with the so-called Gagliardo's norm of $u$
\begin{equation*}
 \|u\|_{H^{s}(\Omega)}
 :=\bigg(\int_{\Omega}|u|^2 dx
 + \int_{\Omega}\int_{\Omega}
 \frac{|u(x)-u(y)|^2}{|x-y|^{2s+d}}dxdy\bigg)^{\frac{1}{2}}
\end{equation*}
(see for example \cite[pp.524]{DPV}). Let $H_0^s(\Omega)
:= \overline{C_0^\infty(\Omega)}^{H^{s}}$ for any $s\in(0,1)$,
while we denote its dual space by $H^{-s}(\Omega)$.
If $\Omega=\mathbb{R}^d$, the space $H^{s}(\mathbb{R}^d)$ has an equivalent definition via Fourier's transform
(see \cite[Proposition 3.4]{DPV}). The notation $W^{k-1/p,p}(\partial\Omega)$ (with $1\leq p<\infty$
and $k\geq 1$ being an integer) is known as the Besov space (fractional Sobolev-type space) on $\partial\Omega$, which exactly describes the trace of functions in $W^{k,p}(\Omega)$ (see \cite{JK}). In particular,
we also denote $W^{k-1/2,2}(\partial\Omega)$ by   $H^{k-1/2}(\partial\Omega)$.
  \end{enumerate}
   \item [(4).]Notation for derivatives.
  \begin{enumerate}
    \item
  $\nabla v = (\nabla_1 v, \cdots, \nabla_d v)$ is the gradient of $v$, where
  $\nabla_i v = \partial v /\partial x_i$ denotes the
  $i^{\text{th}}$ derivative of $v$, and
  $\nabla^2 v = (\nabla^2_{ij} v)_{d\times d}$  denotes the Hessian matrix of $v$, where
  $\nabla^2_{ij} v = \frac{\partial^2 v}{\partial x_i\partial x_j}$.
   \item
  $\nabla\cdot v=\sum_{i=1}^d \nabla_i v_i$
  denotes the divergence of $v$, where
  $v = (v_1,\cdots,v_d)$ is a vector-valued function.
  \item $\nabla_y v$ indicates the gradient of $v$ with respective
  to the variable $y$, while $\Delta_{x} v$ denotes
  the Laplace operator with respective to the variable $x$,
  where $\Delta :=\nabla\cdot\nabla$.
  \item $\nabla_{\text{tan}}f$ denotes the tangential derivative of $f$ on the responding boundary.
  \end{enumerate}
\end{enumerate}

Finally, we mention that:
(1)
when we say that the multiplicative constant depends on the character of the domain,
it means that the constant relies on $M_0$;
(2) the Einstein's summation convention for repeated indices is used throughout.

\section{Preliminaries}\label{sec:2}

\subsection{Properties of correctors and flux correctors}

Most of the properties of correctors and flux correctors
associated with perforated domains are similar to
those established in unperforated ones. Roughly speaking,
ideas here mainly inspired by \cite{MD,PR,WXZ,ZR}.

\begin{lemma}\label{lemma:2.1}
Suppose that $A$ satisfies the conditions $\eqref{a:1}$ and
$\eqref{a:2}$.
Let $N(\cdot,\xi)\in H_{\emph{per}}^1(Y\cap \omega)$ be the weak solution to the equation $\eqref{pde:1.2}$, and then
for any $\xi\in\mathbb{R}^d$, we have the following estimates
\begin{equation}\label{pri:2.1}
 \dashint_{Y\cap \omega} |N(\cdot,\xi)|^2 + \dashint_{Y\cap \omega} |\nabla N(\cdot,\xi)|^2 \leq C|\xi|^2
\end{equation}
and
\begin{equation}\label{pri:2.2}
\dashint_{Y\cap \omega} |\nabla_\xi N(\cdot,\xi)|^2 + \dashint_{Y\cap \omega} |\nabla_{\xi}\nabla N(\cdot,\xi)|^2 \leq C,
\end{equation}
where $C$ depends only on $\mu_0,\mu_1,\omega$ and $d$.
Moreover, if $\omega$ satisfies the separated property $\eqref{g}$, then
there holds
\begin{equation}\label{pri:2.15}
|N(y,\xi)-N(y,\xi')|\lesssim |\xi-\xi'| \text{~~~for~~}y\in\omega,~\xi,\xi'\in \mathbb{R}^d,
\end{equation}
i.e., $|\nabla_{\xi}N(y,\xi)|\lesssim 1$
for any $y\in\omega$, and $\xi\in \mathbb{R}^d$.
\end{lemma}

\begin{remark}\label{remark:2.1}
\emph{In view of the estimate $\eqref{pri:2.1}$,
one may conclude that $N(y,0) = 0$ for $y\in Y\cap \omega$.}
\end{remark}

\begin{lemma}\label{lemma:2.55}
Let $\omega$ be a regular domain.
Suppose that $A$ satisfies $\eqref{a:1}$, $\eqref{a:2}$ and
$\eqref{a:3}$.
Assume that $N(y,\xi)$ is the corrector satisfying \eqref{pde:1.2}, then for any $p\geq 2$, there holds
\begin{equation}\label{pri:2.18}
  \bigg(\dashint_{Y\cap \omega}
  \big|\nabla (N(y,\xi)-N(y,\xi'))\big|^pdy
  \bigg)^{1/p}\lesssim
  |\xi-\xi'|
\end{equation}
for any $\xi,\xi'\in\mathbb{R}^d$, where
the up to constant depends on $\mu_0,\mu_1,\mu_2,\tau,d$
and the character of $\omega$.
\end{lemma}

\begin{remark}\label{remark:2.5}
\emph{In fact, the range of $p$ relies on
the regularity of the boundary of
$\omega$. There are at least two types of Lipschitz domains which
may guarantee the range $2\leq p<\infty$. The one is the so-called
Reifenberg-flat domains,
whose boundary is even permitted to be a fractal structure but
merely owns a ``small'' Lipschitz constant. The other one is
a class of Lipschitz domains with convex properties. Boundary
estimates involving non-smooth domains have been extensively
studied in the past decades, and we refer the readers to
\cite{BW,CP,S5} and the references therein for more details.
Besides, assuming the same conditions
as in Lemma $\ref{lemma:2.55}$,
 on account of the Sobolev's embedding theorem,
 the desired estimate $\eqref{pri:2.15}$ may derive
 from $\eqref{pri:2.2}$ and $\eqref{pri:2.18}$
 straightforwardly by setting
 $p>d$. However, this argument inevitably relies on
 the additional smoothness assumption both on $A$ and
 boundary of $\omega$.}
\end{remark}

\begin{lemma}\label{lemma:2.2}
Suppose $\mathcal{L}_\varepsilon$ satisfies the assumptions
$\eqref{a:1}$, $\eqref{a:2}$.
Let $\widehat{A}$ be given in $\eqref{eq:1.1}$.
Then the effective operator $\mathcal{L}_0$
is still strongly monotone, coercive, i.e,
\begin{equation}\label{pri:2.3}
\left\{\begin{aligned}
&\big<\widehat{A}(\xi)-\widehat{A}(\xi^\prime),\xi-\xi^\prime\big>\geq C_{1}|\xi-\xi^\prime|^2;\\
& |\widehat{A}(\xi)-\widehat{A}(\xi^\prime)|\leq C|\xi-\xi^\prime|;\\
& \widehat{A}(0) = 0,
\end{aligned}\right.
\end{equation}
where $C, C_{1}$ depend on $\mu_0,\mu_1,\omega$ and $d$.
\end{lemma}

\begin{remark}\label{remark:2.4}
\emph{Due to the second line of $\eqref{pri:2.3}$,
it is known that $\nabla \widehat{A}(z)$ exists for a.e. $z\in\mathbb{R}^d$. Moreover,
there holds
\begin{equation}\label{a:4}
\sum_{i,j=1}^d\nabla_j\widehat{A}_i(z)\xi_j\xi_i
 = \lim_{t\to 0} \frac{\big<\widehat{A}(z+t\xi)-\widehat{A}(z),\xi\big>}{t}
 \geq C_{1}|\xi|^2
\end{equation}
for any $\xi\in\mathbb{R}^d$ and for a.e. $z\in\mathbb{R}^d$,
and this property will guarantee that
the $H^2$ theory is still valid for the effective operator $\mathcal{L}_0$. However, the present approach
fails to reveal any higher regularity of $\widehat{A}$
beyond the Lipschitz continuity even when we assume $A$ to be
sufficiently smooth on $\mathbb{R}^{d\times d}$.}
\end{remark}

\begin{lemma}[flux correctors]\label{lemma:2.5}
Suppose $A$ satisfies $\eqref{a:1}$ and $\eqref{a:2}$.
Let $b(y,\xi) = \theta \widehat{A}(\xi)-l^{+}(y)A(y,\xi+\nabla N(y,\xi)) $, where $y\in Y$
and $\xi\in\mathbb{R}^d$.
Then we have two properties: \emph{(i)} $\dashint_Y b(\cdot,\xi) = 0$;
\emph{(ii)} $\nabla\cdot b(\cdot,\xi) = 0$ in $Y$. Moreover, there
exists the so-called flux correctors $E_{ji}(\cdot,\xi)\in H^1_{\emph{per}}(Y)$ such that
\begin{equation}\label{eq:2.1}
 b_i(y,\xi) = \frac{\partial}{\partial y_j}\big\{E_{ji}(y,\xi)\big\}
 \qquad \text{and} \quad E_{ji} = -E_{ij},
\end{equation}
and
\begin{equation}\label{pri:2.4}
\dashint_Y |\nabla_\xi E_{ji}(\cdot,\xi)|^2 + \dashint_Y |\nabla_\xi \nabla E_{ji}(\cdot,\xi)|^2 \leq C,
\end{equation}
where $C$ depends only on $\mu_0,\mu_1$ and $d$.
Moreover, if we additional assume $\eqref{a:3}$, then there holds
\begin{equation}\label{pri:2.17}
  |E(y,\xi)-E(y,\xi')|\lesssim |\xi-\xi'|\quad
  \text{~for~any~}y,\xi,\xi'\in\mathbb{R}^d,
\end{equation}
i.e., $|\nabla_{\xi}E(y,\xi)|\leq C$ for any $y,\xi\in\mathbb{R}^d$.
\end{lemma}

\subsection{Smoothing operators and periodic cancellations}
\label{subsec:2.2}

We mention that
%to weaken the assumption on $f$ in the inequality $\eqref{KEY}$,
the Steklov averaging operator was originally introduced
by V. Zhikov, S. Pastukhova \cite{ZVVPSE}, and
the smoothing operator by Z. Shen \cite{S3} (see Definition
\ref{def:2.1}).

\begin{definition}\label{def:2.1}
Fix a nonnegative function $\zeta\in C_0^\infty(B(0,1/2))$, and $\int_{\mathbb{R}^d}\zeta(x)dx = 1$. Define the smoothing operator
\begin{equation}\label{pri:2.161}
S_\varepsilon(f)(x) = f*\zeta_\varepsilon(x) = \int_{\mathbb{R}^d} f(x-y)\zeta_\varepsilon(y) dy,
\end{equation}
where $\zeta_\varepsilon=\varepsilon^{-d}\zeta(x/\varepsilon)$.
%Let $\tilde{B}(0,1/2)\subset \mathbb{R}^{d-1}$ be a ball,
%and $\eta\in C^{\infty}_0(\tilde{B}(0,1/2))$ be a nonnegative function
%such that $\int_{\mathbb{R}^{d-1}}\eta(x)dx = 1$.
%Then one may similarly define
%\begin{equation}\label{def:2.2}
% K_\delta(g)(x) = g*\eta_\delta(x) = \int_{\mathbb{R}^{d-1}} g(x-y)\eta_\delta(y) dy,
%\end{equation}
%where $\eta_\delta=\delta^{1-d}\zeta(x/\delta)$.
\end{definition}

\begin{lemma}\label{lemma:2.4}
Let $f\in L^p(\mathbb{R}^d)$ for some $1\leq p<\infty$. Then for any $\varpi\in L_{per}^p(\mathbb{R}^d)$,
\begin{equation}\label{pri:2.6}
\big\|\varpi(\cdot/\varepsilon)S_\varepsilon(f)\big\|_{L^p(\mathbb{R}^d)}
\leq C\big\|\varpi\big\|_{L^p(Y)}\big\|f\big\|_{L^p(\mathbb{R}^d)},
\end{equation}
where $C$ depends on $d$. Moreover, if $f\in W^{1,p}(\mathbb{R}^d)$ for some $1<p<\infty$, we have
\begin{equation}\label{pri:2.7}
\big\|S_\varepsilon(f)-f\big\|_{L^p(\mathbb{R}^d)}
\leq C\varepsilon\big\|\nabla f\big\|_{L^p(\mathbb{R}^d)},
\end{equation}
where $C$ depends only on $d$.
\end{lemma}

\begin{proof}
See \cite[Lemmas 2.1 and 2.2]{S5}.
\end{proof}

\begin{remark}
\emph{
We denote the neighbourhood of $U\subset\mathbb{R}^d$ by $(U)_\delta:=\cup_{x\in U} B(x,\delta)$.
For any $1\leq p<\infty$, let $f\in L^p((U)_\delta)$ and
$0<\delta\ll 1$. We noticed  the following property of the convolution:
$\text{supp}(\zeta_\delta * f)
\subseteq \text{supp}(\zeta_\delta)*\text{supp}(f)$,
which leads to
\begin{equation}\label{pri:2.26}
\begin{aligned}
&\|S_{\delta}(f)\|_{L^p(U)}
\lesssim\|f\|_{L^p((U)_\delta)};\\
&\|\nabla S_{\delta}(f)\|_{L^p(U)}
=\|S_{\delta}(\nabla f)\|_{L^p(U)}
\lesssim \delta^{-1}\|f\|_{L^{p}((U)_\delta)},
\end{aligned}
\end{equation}
where the up to constant depends on $d$ and $\zeta$.}
\end{remark}

Recalling the definition of fractional Sobolev-type spaces
in Subsection $\ref{subsec:1.2}$, we have
the following results.

\begin{lemma}\label{interpolation}
Let $f\in C^{\infty}_{0}(\Omega)$ and $0\leq s\leq 1$. Then for any $\varpi\in W^{1,p}_{\emph{per}}(Y)$ with $p>d$, there holds
\begin{equation}\label{pri:2.8}
\|\varpi(\cdot/\varepsilon)f\|_{H^{s}(\Omega)}\leq C\varepsilon^{-s}\|\varpi\|_{W^{1,p}(Y)}\|f\|_{H^{s}(\Omega)},
\end{equation}
and
\begin{equation}\label{pri:2.9}
\|\varpi(\cdot/\varepsilon)f\|_{H^{-s}(\Omega)}\leq C\varepsilon^{-s}\|\varpi\|_{W^{1,p}(Y)}\|f\|_{H^{-s}(\Omega)},
\end{equation}
where the constant C depends on $d$ and $\Omega$.
\end{lemma}

\subsection{Extension operators}\label{subsec:2.3}

\begin{lemma}[extension property]\label{extensiontheory}
Let $\Omega$ and $\Omega_{0}$ be a bounded Lipschitz domains with $\bar{\Omega}\subset\Omega_{0}$ and $\emph{dist}(\partial \Omega_{0},\Omega)>1$.
Let $\omega$ satisfy a separated property.
For $0<\varepsilon<1$, there exists a linear extension operator $P_{\varepsilon}:H^{1}(\Omega_{\varepsilon},
\Gamma_{\varepsilon})\rightarrow H^{1}_{0}(\Omega_{0})$ such that
\begin{equation}\label{pri:2.23}
 \left\{ \begin{aligned}
  & \|P_{\varepsilon}u\|_{H^{1}_{0}(\Omega_{0})}\leq C_{1} \|u\|_{H^{1}(\Omega_{\varepsilon})},\\
  & \|\nabla P_{\varepsilon}u\|_{L^{2}(\Omega_{0})}\leq C_{2} \|\nabla u\|_{L^{2}(\Omega_{\varepsilon})}\\
  \end{aligned}\right.
\end{equation}
for some constants $C_{1},C_{2}$ depending on
the boundary character of $\Omega$ and $\omega$.
Moreover, if $u\in W^{1,p}(\Omega_{\varepsilon},\Gamma_{\varepsilon})$ and  $\frac{2d}{d+1}-\epsilon<p<\frac{2d}{d-1}+\epsilon$ with $0<\epsilon\ll1$, then there holds
\begin{equation}\label{pri:2.22}
\left\{\begin{aligned}
 & \|P_{\varepsilon}u\|_{W^{1,p}(\Omega_{0})}\leq C_3\|u\|_{W^{1,p}(\Omega_{\varepsilon})};\\
 &\|\nabla P_{\varepsilon}u\|_{L^{p}(\Omega_{0})}\leq C_4\|\nabla u\|_{L^{p}(\Omega_{\varepsilon})},
\end{aligned}\right.
\end{equation}
in which the constant $C_3, C_4$ additionally depends on $p$ and $d$.
\end{lemma}

\begin{remark}\label{remark:2.2}
\emph{The extension property is very important in our later arguments. Due to this lemma, one may transfer the computations
from the region with holes to the ``usual'' one (without holes),
to avoid the difficulties arising from irregular boundary situations.
The condition $\text{dist}(\partial\Omega_0,\Omega)>1$ here can be improved into
$\text{dist}(\partial\Omega_0,\Omega)\sim 10\varepsilon$ through
some small tricks (see \cite[Lemma 2.10]{WXZ2020}).
If $\Omega =\mathbb{R}^d$,
the estimates $\eqref{pri:2.23},\eqref{pri:2.22}$ are still true
with the integral domain $\Omega_{0}$ replaced by $\mathbb{R}^d$.
Finally, we mention that the range of $p$ in the estimate $\eqref{pri:2.22}$ can hardly be
extended to $[2,\infty)$ due to nonsmoothness assumption on the domains, and the optimal range of $p$ owns itself interests.
}
\end{remark}

\begin{lemma}\label{lemma:2.15}
 For $w\in H^{1}(\Omega_{\varepsilon},\Gamma_{\varepsilon})$,
 let $\tilde{w}$ be the extension of $w$ given by Lemma \ref{extensiontheory}. Then we have
 \begin{equation}\label{pri:2.10}
   \|\tilde{w}\|_{L^{2}(O_{4\varepsilon})}\leq C\varepsilon \|\nabla \tilde{w}\|_{L^{2}(\Omega)},
 \end{equation}
 where C depends on $d,\Omega$ and $\omega$.
\end{lemma}
\begin{proof}
  See \cite[Lemma 3.4]{PR}.
\end{proof}

\begin{lemma}[Sobolev-Poincar\'e's inequality on perforated domains]\label{lemma:2.20}
Let
$\omega$ satisfy a separated property.
Let $w\in W^{1,p}{(\varepsilon\omega)}$ with
$|\frac{1}{p}-\frac{1}{2}|<\frac{1}{2d}+\epsilon$ and
$0<\epsilon\ll 1$. Let $1/q=1/p-1/d$.
Then for any $r>0$
and $x\in\mathbb{R}^d$ there exists a constant $c_r$ such that
\begin{equation}\label{pri:2.20}
\|w-c_r\|_{L^q(B^{\varepsilon}(x,r))}
\lesssim
\|\nabla w\|_{L^p(B^{\varepsilon}(x,3r))},
\end{equation}
where the up to constant is independent of
$\varepsilon, x$ and $r$. Moreover, if $w\in W^{1,p}(\Omega_\varepsilon,\Gamma_\varepsilon)$, then,
for any
$D_{5r}^{\varepsilon}\subset\Omega_\varepsilon$
with $r> 0$,  we have
\begin{equation}\label{pri:2.21}
 \|w\|_{L^q(D_{r}^{\varepsilon})}
 \lesssim \|\nabla w\|_{L^p(D_{3r}^{\varepsilon})},
\end{equation}
whose estimated constant will rely on $d$, $p$ and
the boundary character of $\Omega$ and $\omega$,
but does not depend on $r$ and $\varepsilon$ either.
\end{lemma}

\section{Convergence rates}\label{sec:3}
%In this section, we derive the convergence rates of
%\eqref{pde:1.1} by calculating the integral  $\int_{\Omega_{\varepsilon}}A(x/\varepsilon,\nabla u_{\varepsilon})\nabla w_{\varepsilon}dx$.
As a start, we introduce some cut-off functions.
Let ${\psi}_{\varepsilon}^{\prime},{\psi}_{\varepsilon}\in C_{0}^{\infty}(\Omega)$ satisfy
\begin{equation}\label{cutoff}
\left\{\begin{aligned}
  & 0\leq \psi_{\varepsilon},\psi_{\varepsilon}^{\prime}\leq 1\quad \text{for}~~x\in\Omega ,\\
  & \text{supp}(\psi_{\varepsilon})\subset \Sigma_{3\varepsilon},~~ \text{supp}(\psi_{\varepsilon}^{\prime})\subset \Sigma_{\varepsilon},\\
  & \psi_{\varepsilon}=1\quad \text{in}~~\Sigma_{4\varepsilon},\quad
  \psi_{\varepsilon}^{\prime}=1\quad \text{in}~~\Sigma_{2\varepsilon},\\
  & |\nabla \psi_{\varepsilon}|\lesssim \varepsilon^{-1},|\nabla \psi_{\varepsilon}^{\prime}|\lesssim \varepsilon^{-1}.
  \end{aligned}\right.
\end{equation}
By the definition of ${\psi}_{\varepsilon}^{\prime},{\psi}_{\varepsilon}$, it's known that $\psi_{\varepsilon}(1-\psi_{\varepsilon}^{\prime})=0~\text{in}~ \Omega.$

\subsection{The proof of Theorem $\ref{thm:1.1}$}

\begin{lemma}[energy estimates of weak formulations]
\label{lemma:3.1}
Let $\Omega\subset\mathbb{R}^d$ and $\omega$
be Lipschitz domains.
Assume that $A$ satisfies $\eqref{a:1}$ and $\eqref{a:2}$.
Let $F\in H^{s}(\mathbb{R}^d)$ with $0\leq s\leq 1$.
Suppose that $u_\varepsilon\in H^{1}(\Omega_{\varepsilon})~\text{and}~ u_0\in H^1(\Omega)$
satisfy equations \eqref{pde:1.1} and \eqref{pde:1.3}, respectively.
Let $w_{\varepsilon}=u_{\varepsilon}-v_{\varepsilon},$
$v_{\varepsilon}=u_{0}+\varepsilon N(x/\varepsilon,\varphi)$ in which $\varphi=S_{\varepsilon}(\psi_{\varepsilon}\nabla u_{0})$.
Then we have
\begin{equation}\label{pri:2.5}
\begin{aligned}
\|\nabla w_{\varepsilon}\|_{L^{2}(\Omega_{\varepsilon})}
\lesssim  \Bigg\{\varepsilon \bigg(\|\nabla_{\xi}E(\cdot /\varepsilon ,\varphi)\cdot\nabla\varphi\|_{L^{2}(\Omega)}
&+\|\nabla_{\xi}N(\cdot /\varepsilon ,\varphi)\cdot\nabla\varphi\|_{L^{2}(\Omega_{\varepsilon})}\bigg)\\
&+\varepsilon^{s}\|F\|_{H^{s}(\mathbb{R}^d)}+\|\nabla u_{0}-\varphi\|_{L^{2}(\Omega)}\Bigg\},
\end{aligned}
\end{equation}
in which the up to constant depends only on $\mu_{0},\mu_{1},d$, but independent of $\varepsilon$ and $s$.
\end{lemma}

\begin{proof}
By the definition of $\varphi$, it's known that $\varphi\in H^{1}_0(\Omega).$
In view of
$u_\varepsilon$ and $u_{0}$ are solutions to \eqref{pde:1.1} and \eqref{pde:1.3}, respectively, we have
\begin{equation*}
\begin{aligned}
& \int_{\Omega_{\varepsilon}} A(x/\varepsilon,\nabla u_\varepsilon)\nabla w_{\varepsilon}dx=\int_{\Omega_{\varepsilon}}Fw_{\varepsilon} dx
=\int_{\Omega}l_{\varepsilon}^{+}F\tilde{w}_{\varepsilon} dx,\\
& \int_{\Omega} \widehat{A}(\nabla u_0)\nabla (\tilde{w}_{\varepsilon}\psi_{\varepsilon}^{\prime})dx
=\int_{\Omega}F(\tilde{w}_{\varepsilon}\psi_{\varepsilon}^{\prime})dx,
\end{aligned}
\end{equation*}
where we use the fact that $w_{\varepsilon}\in H^{1}(\Omega_{\varepsilon}, \Gamma_{\varepsilon})$ and $\tilde{w}_{\varepsilon}$ is the extension of $w_{\varepsilon}$ given by Lemma \ref{extensiontheory}.
In fact, because of  $\varphi=S_{\varepsilon}(\psi_{\varepsilon}\nabla u_{0}),$ we have $\varphi=0$ on $O_{2\varepsilon}$ and in view of Remark $\ref{remark:2.1}$, we see that
$N(x/\varepsilon,\varphi) = 0, \nabla_{y}N(x/\varepsilon,\varphi)=0$
for any $x\in O_{2\varepsilon}\cap \Omega_{\varepsilon}$.  This coupled with
$u_\varepsilon = u_0$ on $\Gamma_{\varepsilon}$ leads to the fact $w_\varepsilon\in H^1(\Omega_{\varepsilon},\Gamma_{\varepsilon})$.

It follows from the above two equalities that
\begin{equation}\label{eq:3.1}
\begin{aligned}
& \int_{\Omega_{\varepsilon}} \big(A(x/\varepsilon,\nabla u_\varepsilon) - A(x/\varepsilon,\nabla v_\varepsilon)\big)
\cdot \nabla w_{\varepsilon} dx\\
& =\int_{\Omega}l_{\varepsilon}^{+}F \tilde{w}_{\varepsilon}
-\theta F\tilde{w}_{\varepsilon}\psi_{\varepsilon}^{\prime}dx
+\theta\int_{\Omega}\widehat{A}(\nabla u_{0})\nabla (\tilde{w}_{\varepsilon}\psi_{\varepsilon}^{\prime})dx
-\int_{\Omega}l_{\varepsilon}^{+}A(x/\varepsilon,\nabla v_{\varepsilon})\nabla \tilde{w}_{\varepsilon}dx\\
& =\int_{\Omega}l_{\varepsilon}^{+}F \tilde{w}_{\varepsilon}
-\theta F\tilde{w}_{\varepsilon}\psi_{\varepsilon}^{\prime}dx
+\theta\int_{\Omega}[\widehat{A}(\nabla u_{0})-\widehat{A}(\varphi)]\nabla (\tilde{w}_{\varepsilon}\psi_{\varepsilon}^{\prime})dx\\
&\quad +\int_{\Omega}[\theta\widehat{A}( \varphi)-l_{\varepsilon}^{+}A(x/\varepsilon,\varphi+\nabla_{y}
N(x/\varepsilon,\varphi))]\nabla(\tilde{w}_{\varepsilon}
\psi_{\varepsilon}^{\prime})dx\\
& \quad+\int_{\Omega}l_{\varepsilon}^{+}A(x/\varepsilon,
\varphi+\nabla_{y}N(x/\varepsilon,\varphi))\nabla(\tilde{w}_{\varepsilon}
\psi_{\varepsilon}^{\prime})-l_{\varepsilon}^{+}A(x/\varepsilon,\nabla v_{\varepsilon})\nabla \tilde{w}_{\varepsilon}dx\\
& :=I_{1}+I_{2}+I_{3}+I_{4}.
\end{aligned}
\end{equation}
So, to obtain the desired
result $\eqref{pri:2.5}$ is reduced to estimate every
term $I_{i}$ with $i=1,2,3,4$.

With respect to $I_{1}$, there holds the following
decomposition
\begin{equation*}
  I_{1}=\int_{\Omega}(l_{\varepsilon}^{+}-\theta)
  F\tilde{w}_{\varepsilon}\psi_{\varepsilon}^{\prime}dx
  +\int_{\Omega}(1-\psi_{\varepsilon}^{\prime})l_{\varepsilon}^{+}
  F\tilde{w}_{\varepsilon}dx:=I_{11}+I_{12}.
\end{equation*}
Since $\text{supp}(1-\psi_{\varepsilon}^{\prime})
=O_{2\varepsilon}$ and Lemma \ref{lemma:2.15},
we have
\begin{equation}\label{f:3.6}
  \begin{aligned}
  |I_{12}|& \leq \int_{O_{2\varepsilon}}|F\tilde{w}_{\varepsilon}|dx
  \lesssim \|F\|_{L^{2}(O_{2\varepsilon})}
  \|\tilde{w}_{\varepsilon}\|_{L^{2}(O_{2\varepsilon})}
  \lesssim^{\eqref{pri:2.10}} \varepsilon \|F\|_{L^{2}({\Omega})}
  \|\nabla \tilde{w}_{\varepsilon}\|_{L^{2}(\Omega)}.
  \end{aligned}
\end{equation}
To deal with the first term $I_{11}$, we consider the
auxiliary equation
\begin{equation}\label{auxi1}
  \left\{\begin{aligned}
  & -\Delta \Psi(y)=l^{+}(y)-\theta~~\text{in}~ Y,\\
  & \dashint_{Y}\Psi dy=0,~\Psi\in H^{1}_{\text{per}}(Y).
  \end{aligned}\right.
\end{equation}
According to $\int_{Y}l^{+}(y)-\theta dy=0,$ it's known that \eqref{auxi1} has a solution $\Psi\in H^{1}_{per}(Y)$. Moreover,
let $B:=B(0,1/4)$, and from interior $W^{2,p}$ estimates it follows that
\begin{equation*}
\begin{aligned}
\|\nabla\Psi\|_{W^{1,p}(B)}
\lesssim \|\nabla\Psi\|_{L^{2}(2B)} + \|l^{+}-\theta\|_{L^p(2B)}
\lesssim 1
\end{aligned}
\end{equation*}
for $2\leq p< \infty$. Therefore, a covering argument leads to
\begin{equation}\label{f:3.2}
\|\nabla\Psi\|_{W^{1,p}(Y)} \lesssim 1.
\end{equation}
Now one may proceed to address the term $I_{11}$. Inserting the first line of
the equation $\eqref{auxi1}$ into $I_{11}$ and,
\begin{equation*}
  \begin{aligned}
  I_{11}
  & =-\int_{\Omega}\Delta_{y} \Psi(F\tilde{w}_{\varepsilon}\psi_{\varepsilon}^{\prime})dx
  =-\varepsilon
  \int_{\Omega}{\nabla_{x}}\cdot\big(\nabla_{y} \Psi\big)(F\tilde{w}_{\varepsilon}\psi_{\varepsilon}^{\prime})dx\\
  & =\varepsilon\int_{\Omega}\nabla_{y}
  \Psi\cdot\nabla(F\tilde{w}_{\varepsilon}
  \psi_{\varepsilon}^{\prime})dx=\varepsilon\int_{\Omega}\nabla_{y}
  \Psi\cdot\nabla F(\tilde{w}_{\varepsilon}
  \psi_{\varepsilon}^{\prime})dx
  +\varepsilon\int_{\Omega}\nabla_{y}\Psi\cdot\nabla (\tilde{w}_{\varepsilon}
  \psi_{\varepsilon}^{\prime})Fdx\\
  &:=\varepsilon I_{11}^a+\varepsilon I_{11}^b,
  \end{aligned}
\end{equation*}
where $y=x/\varepsilon$. The easier term is
\begin{equation}\label{f:3.4}
|I_{11}^{b}|\lesssim\|F\|_{L^{2}(\Omega)}(\|\nabla \tilde{w}_{\varepsilon}\|_{L^{2}(\Omega)}+
\varepsilon^{-1}\|\tilde{w}_{\varepsilon}\|_{L^{2}(O_{2\varepsilon})})
\lesssim^{\eqref{pri:2.10}}\|F\|_{L^{2}(\Omega)}\|\nabla \tilde{w}_{\varepsilon}\|_{L^{2}(\Omega)}.
\end{equation}
Then we deal with the other term $I_{11}^{a}$ as follows:
\begin{equation}\label{f:3.5}
\begin{aligned}
|I_{11}^{a}|& \leq \|\nabla F\|_{H^{s-1}(\mathbb{R}^{d})}\|
\nabla\Psi(\cdot/\varepsilon)\tilde{w}_{\varepsilon}
\psi^{'}_{\varepsilon}\|_{H^{1-s}(\mathbb{R}^d)}\\
& \lesssim^{\eqref{f:3.3}} \|F\|_{H^{s}(\mathbb{R}^{d})}\|
\nabla\Psi(\cdot/\varepsilon)\tilde{w}_{\varepsilon}
\psi^{'}_{\varepsilon}\|_{H^{1-s}(\Omega)}\\
& \lesssim^{\eqref{pri:2.8}}
\varepsilon^{s-1}\| F\|_{H^{s}(\mathbb{R}^{d})}
\|\nabla\Psi\|_{W^{1,p}(Y)}
\| \tilde{w}_{\varepsilon}\psi^{'}_{\varepsilon}
\|_{H^{1-s}(\Omega)}\\
& \lesssim^{\eqref{f:3.2}}
\varepsilon^{s-1}\| F\|_{H^{s}(\mathbb{R}^{d})}
\| \tilde{w}_{\varepsilon}\psi^{'}_{\varepsilon}\|_{H^{1}(\Omega)},
\end{aligned}
\end{equation}
where $p>d$ and $s\in[0,1]$. Here we adopt $\|f\|_{H^{s}(\mathbb{R}^d)}:=
\|(1+|\xi|^2)^{s/2}\widehat{f}\|_{L^2(\mathbb{R}^d)}$
(with $s\in\mathbb{R}$ and $\widehat{f}$
represents Fourier transform of $f$)
as the definition of the norms of
the fractional Sobolev functions. Thus, it is not hard to observe that
\begin{equation}\label{f:3.3}
\|\nabla F\|_{H^{s-1}(\mathbb{R}^{d})}
\lesssim
\| F\|_{H^{s}(\mathbb{R}^{d})}
\quad\text{and}\quad
\|F\|_{L^{2}(\mathbb{R}^d)}\leq \|F\|_{H^{s}(\mathbb{R}^d)}.
\end{equation}
%Also, we remark that the definition of $H^{s}(\mathbb{R}^d)$
%via Fourier transform is equivalent to that given by
%Gagliardo norm (see for example \cite[Proposition 3.4]{DPV}).
In fact, we employ zero-extension (see \cite[Lemma 5.1]{DPV}) in the
second inequality of $\eqref{f:3.5}$, and \cite[Proposition 2.2]{DPV} in the last one.
%Finally, we mention another fact
%that $H^{s}(\Omega) = [L^2(\Omega),H^1(\Omega)]_{s}$ with $s\in(0,1)$,
%whose left-hand side is understood in the sense of
%Gagliardo's norms (see for example
%\cite[Proposition 2.17]{JK}).
Hence,  we have
\begin{equation*}
\begin{aligned}
|I_{1}|
&\lesssim^{\eqref{f:3.6},\eqref{f:3.4},
\eqref{f:3.5},\eqref{f:3.3}} \varepsilon^{s}\| F\|_{H^{s}(\mathbb{R}^{d})}
\Big\{\|\nabla \tilde{w}_{\varepsilon}\|_{L^{2}(\Omega)}
+ \| \tilde{w}_{\varepsilon}\psi^{'}_{\varepsilon}\|_{H^{1}(\Omega)}\Big\}\\
&\lesssim^{\eqref{pri:2.10}}
\varepsilon^{s}\| F\|_{H^{s}(\mathbb{R}^{d})}
\|\nabla \tilde{w}_{\varepsilon}\|_{L^{2}(\Omega)}.
\end{aligned}
\end{equation*}

By the properties of $\widehat{A}(\xi)$, we have
\begin{equation*}
  \begin{aligned}
  |I_{2}|
  & =\Big|\theta\int_{\Omega} (\widehat{A}(\nabla u_{0})-\widehat{A}(\varphi))\nabla(\tilde{w}_{\varepsilon}
  \psi_{\varepsilon}^{\prime})dx\Big|\\
  & \lesssim^{\eqref{pri:2.3}}
  \int_{\Omega} |\nabla u_{0}-\varphi|\cdot|\nabla(\tilde{w}_{\varepsilon}
  \psi_{\varepsilon}^{\prime})|dx
  \lesssim^{\eqref{pri:2.10}}\|\nabla u_{0}-\varphi\|_{L^{2}(\Omega)}
  \|\nabla\tilde{w}_{\varepsilon}\|_{L^{2}(\Omega)},
  \end{aligned}
\end{equation*}
where we also use H\"{o}lder's inequality in the last inequality.

Recalling that $b(y,\xi) = \theta \widehat{A}(\xi)-l^{+}(y)A(y,\xi+\nabla N(y,\xi)),$ it follows from Lemmas \ref{lemma:2.5} and \ref{lemma:2.15} that
\begin{equation*}
  \begin{aligned}
  |I_{3}|&=
   \Big|\int_{\Omega}b(x/\varepsilon,\varphi)
  \nabla(\tilde{w}_{\varepsilon}\psi_{\varepsilon}^{\prime})dx\Big|\\
  & =^{\eqref{eq:2.1}}\Big|\varepsilon\int_{\Omega}\frac{\partial}{\partial x_{j}}\{E_{ji}(x/\varepsilon,\varphi)\}\frac{\partial}
  {\partial x_{i}}(\tilde{w}_{\varepsilon}\psi_{\varepsilon}^{\prime})dx
  -\varepsilon\int_{\Omega}\frac{\partial}{\partial \xi_{k}}\{E_{ji}(x/\varepsilon,\varphi)\}\frac{\partial\varphi_{k}}{\partial x_{j}}\frac{\partial}
  {\partial x_{i}}(\tilde{w}_{\varepsilon}\psi_{\varepsilon}^{\prime})dx\Big|\\
  & =\Big|-\varepsilon\int_{\Omega}E_{ji}(x/\varepsilon,\varphi)
  \frac{\partial^{2}}
  {\partial x_{j}\partial x_{i}}(\tilde{w}_{\varepsilon}\psi_{\varepsilon}^{\prime})dx
  -\varepsilon\int_{\Omega}\frac{\partial}{\partial \xi_{k}}\{E_{ji}(x/\varepsilon,\varphi)\}\frac{\partial\varphi_{k}}{\partial x_{j}}\frac{\partial}
  {\partial x_{i}}(\tilde{w}_{\varepsilon}\psi_{\varepsilon}^{\prime})dx\Big|\\
  & =^{\eqref{eq:2.1}}\Big|\varepsilon\int_{\Omega}\frac{\partial}{\partial \xi_{k}}\{E_{ji}(x/\varepsilon,\varphi)\}
  \frac{\partial\varphi_{k}}{\partial x_{j}}\frac{\partial}
  {\partial x_{i}}(\tilde{w}_{\varepsilon}\psi_{\varepsilon}^{\prime})
  dx\Big|\\
  &\lesssim^{\eqref{pri:2.10}} \varepsilon\|\nabla_{\xi} E_{ji}(\cdot/\varepsilon,\varphi)\nabla_{j}
  \varphi\|_{L^{2}(\Omega)}\|\nabla_i\tilde{w}_{\varepsilon}
  \|_{L^{2}(\Omega)},
  \end{aligned}
\end{equation*}
where we use the fact that $E_{ji}(x/\varepsilon,\varphi)=0$ on $\partial\Omega$ according to $\varphi\in H^{1}_0(\Omega)$ in the third equality.

For the last term $I_{4}$, one may have the following
decomposition,
\begin{equation*}
  \begin{aligned}
  I_{4}
  & =\int_{\Omega}l_{\varepsilon}^{+}
[A(x/\varepsilon,\varphi+\nabla_{y}N(y,\varphi))-A(x/\varepsilon,\nabla v_{\varepsilon})]\nabla \tilde{w}_{\varepsilon}dx+\int_{\Omega}l_{\varepsilon}^{+}A(x/\varepsilon,
\varphi+\nabla_{y}N(y,\varphi))\nabla(\tilde{w}_{\varepsilon}
\psi_{\varepsilon}^{\prime}-\tilde{w}_{\varepsilon})dx\\
& :=I_{41}+I_{42}.
  \end{aligned}
\end{equation*}
Then we have
\begin{equation*}
  \begin{aligned}
  |I_{41}|
  & \lesssim^{\eqref{a:1}} \int_{\Omega}l_{\varepsilon}^{+}|\varphi-\nabla u_{0}-\varepsilon \nabla_{\xi}N(x/\varepsilon,\varphi)\nabla\varphi|\cdot
  |\nabla\tilde{w}_{\varepsilon}|dx\\
  & \lesssim  \|\varphi-\nabla u_{0}\|_{L^{2}(\Omega)}
  \|\nabla \tilde{w}_{\varepsilon}\|_{L^{2}(\Omega)}
  +\varepsilon\|\nabla_{\xi}N(\cdot/\varepsilon,\varphi)\nabla\varphi
  \|_{L^{2}(\Omega_{\varepsilon})}
  \|\nabla \tilde{w}_{\varepsilon}\|_{L^{2}(\Omega)};\\
|I_{42}|
& \lesssim^{\eqref{a:1}} \int_{\Omega}l^{+}_{\varepsilon}|
\varphi+\nabla_{y}N(y,\varphi)|\cdot|\nabla[\tilde{w}_{\varepsilon}
(1-\psi_{\varepsilon}^{\prime})]|dx.
\end{aligned}
\end{equation*}
 According to $\text{supp}(1-\psi_{\varepsilon}^{\prime})
 =O_{2\varepsilon}$ and
 $\nabla_{y}N(x/\varepsilon,\varphi)=0$
for any $x\in O_{2\varepsilon}\cap \Omega_{\varepsilon}$, we see that
$I_{42}=0$. Thus, plugging the estimates of $I_{41}$ back into $I_4$, there holds
\begin{equation*}
  \begin{aligned}
  |I_{4}|\lesssim \bigg( \|\varphi-\nabla u_{0}\|_{L^{2}(\Omega)}
  +\varepsilon\|\nabla_{\xi}N(\cdot/\varepsilon,\varphi)\nabla\varphi
  \|_{L^{2}(\Omega_{\varepsilon})}\bigg)
  \|\nabla \tilde{w}_{\varepsilon}\|_{L^{2}(\Omega)}.
  \end{aligned}
\end{equation*}

Consequently,
combining the above estimates of $I_{i}$
with $i=1,2,3,4$ and the assumption \eqref{a:1}, we arrive at the stated estimate \eqref{pri:2.5} appealing
to Lemma $\ref{extensiontheory}$. This ends
the proof.
\end{proof}

\begin{lemma}\label{lemma:3.2}
Assume the same conditions as in Theorem $\ref{thm:1.1}$,
while we set $F\in H^s(\mathbb{R}^d)$ with
$0\leq s\leq 1$ in the
present lemma.
Then we have the following estimate
\begin{equation}\label{pri:3.1}
\|\nabla w_\varepsilon\|_{L^2(\Omega_{\varepsilon})}
\lesssim
\bigg\{\|\nabla u_0\|_{L^2(O_{4\varepsilon})}
+ \varepsilon\|\nabla^2 u_0
\|_{L^2(\Sigma_{3\varepsilon})}
+\varepsilon^{s}\|F\|_{H^{s}(\mathbb{R}^d)}\bigg\}
\end{equation}
and
\begin{equation}\label{pri:3.2}
\|u_\varepsilon - u_0\|_{L^2(\Omega_{\varepsilon})}
\lesssim \bigg\{\|\nabla u_0\|_{L^2(O_{4\varepsilon})}
+ \varepsilon\|\nabla^2 u_0
\|_{L^2(\Sigma_{3\varepsilon})}
+\varepsilon^{s}\|F\|_{H^{s}(\mathbb{R}^d)}\bigg\},
\end{equation}
where the layer set $O_{4\varepsilon}$ and co-layer set $\Sigma_{3\varepsilon}$ are defined in Subsection $\ref{subsec:1.2}$, and
the up to constant depends on $\mu_0,\mu_1,\mu_2,\tau,d,r_0$ and
the boundary character of $\omega$, but never relies on $s$ and $\varepsilon$.
\end{lemma}

\begin{proof}
According to Lemma \ref{lemma:3.1},
to show the estimate $\eqref{pri:3.1}$, it suffices to estimate
$\|\nabla_{\xi}E(\cdot/\varepsilon,\varphi)\cdot\nabla\varphi\|_{L^{2}(\Omega)},$\qquad $\|\nabla_{\xi}N(\cdot /\varepsilon ,\varphi)\cdot\nabla\varphi\|_{L^{2}(\Omega_{\varepsilon})}$ and
$\|\nabla u_{0}-\varphi\|_{L^{2}(\Omega)}.$
On account of Lemmas \ref{lemma:2.1} and \ref{lemma:2.5}, we can derive that
\begin{equation*}
  \begin{aligned}
 \|\nabla_{\xi}E(\cdot/\varepsilon,\varphi)\cdot\nabla\varphi
  \|_{L^{2}(\Omega)}
&+ \|\nabla_{\xi}N(\cdot /\varepsilon ,\varphi)\cdot\nabla\varphi\|_{L^{2}(\Omega_{\varepsilon})}\\
& \lesssim^{\eqref{pri:2.15},\eqref{pri:2.17}} \|\nabla \varphi\|_{L^{2}(\Omega)}
\lesssim^{\eqref{pri:2.6}} \|\nabla^{2}  u_{0}\|_{L^{2}(\Sigma_{3\varepsilon})}+\varepsilon^{-1}\|\nabla u_{0}\|_{L^{2}(O_{4\varepsilon})},
  \end{aligned}
\end{equation*}
where we recall $\varphi=S_{\varepsilon}(\psi_{\varepsilon}\nabla u_{0})$,
and use the properties of $\psi_{\varepsilon}$ (see $\eqref{cutoff}$) in the last inequality. Also,
we have
\begin{equation*}
  \begin{aligned}
 \|\nabla u_{0}-\varphi\|_{L^{2}(\Omega)}
 & \leq \|\psi_{\varepsilon}\nabla u_{0}-S_{\varepsilon}(\psi_{\varepsilon}\nabla u_{0})\|_{L^{2}(\Omega)}+\|(1-\psi_{\varepsilon})\nabla u_{0}\|_{L^{2}(\Omega)}\\
&\lesssim^{\eqref{pri:2.7}}\varepsilon\|\nabla(\psi_{\varepsilon}\nabla u_{0})\|_{L^{2}(\Omega)}+\|(1-\psi_{\varepsilon})\nabla u_{0}\|_{L^{2}(\Omega)}
\lesssim\varepsilon\|\nabla^{2}
u_{0}\|_{L^{2}(\Sigma_{3\varepsilon})}
+\|\nabla  u_{0}\|_{L^{2}(O_{4\varepsilon})}.
  \end{aligned}
\end{equation*}
Consequently, plugging the above two estimates back into
the estimate $\eqref{pri:2.5}$ leads to the desired estimate
\eqref{pri:3.1}.

We proceed to show the estimate \eqref{pri:3.2}, and
\begin{equation*}
\begin{aligned}
\|u_{\varepsilon}-u_{0}\|_{L^{2}(\Omega_{\varepsilon})}
&\leq \|u_{\varepsilon}-u_{0}-\varepsilon N(\cdot/\varepsilon,\varphi)\|_{L^{2}(\Omega_{\varepsilon})}
+\varepsilon \|N(\cdot/\varepsilon,\varphi)\|_{L^{2}(\Omega_{\varepsilon})}\\
&\lesssim \|\nabla w_{\varepsilon}\|_{L^{2}(\Omega_{\varepsilon})}+
\varepsilon \|N(\cdot/\varepsilon,\varphi)\|_{L^{2}(\Omega_{\varepsilon})},
\end{aligned}
\end{equation*}
in which we employ Lemma \ref{extensiontheory} and Poincar\'{e}'s inequality in the second inequality.
Next, we will show
\begin{equation}\label{f:3.1}
\int_{\Omega_{\varepsilon}}|N(x/\varepsilon,\varphi)|^{2}dx\lesssim \int_{\Omega}|\varphi|^{2}dx.
\end{equation}
To do so, we collect a family of small cubes denoted by $Y_\varepsilon^i = \varepsilon(i+Y)$ for
$i\in\mathbb{Z}^d$ with an index set $I_\varepsilon$, such that $\Omega_{\varepsilon}\backslash O_{2\varepsilon}\subset
\cup_{i\in I_\varepsilon} Y_\varepsilon^i \subset \Omega$
and $Y_\varepsilon^i\cap Y_\varepsilon^j = \emptyset$ if $i\not=j$. Thus
\begin{equation}\label{f:3.15}
\begin{aligned}
 \int_{\Omega_{\varepsilon}} |N(x/\varepsilon,\varphi)|^2 dx
& \leq \sum_{i\in I_\varepsilon}\int_{Y^i_\varepsilon\cap \varepsilon\omega} |N(x/\varepsilon,\varphi)|^2 dx
 +  \int_{O_{2\varepsilon}\cap\Omega_{\varepsilon}}
 |N(x/\varepsilon,\varphi)|^2 dx \\
& \lesssim \sum_{i\in I_\varepsilon} |Y_\varepsilon^i||\varphi^i|^2
\lesssim \int_{\Omega}|\varphi|^2 dx,
\end{aligned}
\end{equation}
where we employ the estimate $\eqref{pri:2.1}$ and the fact that
\begin{equation*}
N(x/\varepsilon,\varphi)  = 0  \qquad \forall x\in O_{2\varepsilon}\cap\Omega_{\varepsilon}.
\end{equation*}
One may take $\varphi^i = \inf_{x\in Y_\varepsilon^i}
|S_\varepsilon(\psi_{\varepsilon}\nabla u_0)(x)|$,
and the second step in $\eqref{f:3.15}$ is due to the fact that $N(y,\xi)$ is continuous with respective to the second variable (see Lemma $\ref{lemma:2.1}$),
 while the last step in $\eqref{f:3.15}$ comes from
Chebyshev's inequality. Therefore, the estimate $\eqref{pri:3.2}$ consequently follows from \eqref{pri:3.1},\eqref{f:3.1} and the following inequality
\begin{equation*}
\|S_\varepsilon(\psi_{\varepsilon}\nabla u_0)\|_{L^2(\Omega)}
\lesssim^{\eqref{pri:2.6}} \|\psi_{\varepsilon}\nabla u_0\|_{L^2(\Omega)}
\lesssim \Big\{\varepsilon^{-1}\|\nabla u_0\|_{L^2(O_{4\varepsilon})}
+ \|\nabla^2 u_0\|_{L^2(\Sigma_{3\varepsilon})}\Big\},
\end{equation*}
where we employ Poincar\'e's inequality in the second step,
and this ends the proof.
\end{proof}

\noindent\textbf{Proof of Theorem $\ref{thm:1.1}$}.
For any $0\leq s\leq 1$,
let $\tilde{F}$ be the $H^{s}$-extension of
$F$ such that $\tilde{F} = F$ on $\Omega$ and
\begin{equation}\label{f:3.16}
\|\tilde{F}\|_{H^{s}(\mathbb{R}^d)}
\lesssim \|F\|_{H^{s}(\Omega)}
\end{equation}
(see \cite[Theorem 5.4]{DPV} for the case of $0<s<1$,
and for the case $s=0$ we take zero-extension while
we adopt common extension theorem in the Sobolev space
$W^{1,2}(\Omega)$ for $s=1$). Although we have extended
the given data $F$, there is no change in the equations
$\eqref{pde:1.1}$ and $\eqref{pde:1.3}$, and therefore
this operation has no influence on the related solutions.
Based upon Lemma $\ref{lemma:3.2}$, the desired results
are reduced to address the layer type quantity
$\|\nabla u_0\|_{L^2(O_{4\varepsilon})}$ and co-layer type one $\|\nabla^2 u_0\|_{L^2(\Sigma_{3\varepsilon})}$. Obviously, the related estimates will rely on the regularity of $\partial\Omega$.
We first hand the estimate $\eqref{pri:1.2}$.
It follows from the estimates $\eqref{pri:3.2}$ and $\eqref{pri:2.11}$ that
\begin{equation}\label{f:2.8}
\begin{aligned}
\|u_\varepsilon-u_0\|_{L^2(\Omega_{\varepsilon})}
&\lesssim \|\nabla u_0\|_{L^2(O_{4\varepsilon})}
+ \varepsilon
\Big\{\|F\|_{L^2(\Omega)}
+\|g\|_{H^{3/2}(\partial\Omega)}\Big\}
+\varepsilon^{s}\|\tilde{F}\|_{H^{s}(\mathbb{R}^d)}\\
&\lesssim^{\eqref{f:3.3}}\|\nabla u_0\|_{L^2(O_{4\varepsilon})}
+\varepsilon\|g\|_{H^{3/2}(\partial\Omega)}+
\varepsilon^{s}\|\tilde{F}\|_{H^{s}(\mathbb{R}^d)},
\end{aligned}
\end{equation}
where one may choose $s\in[1/2,1]$ and notice that
$F=\tilde{F}$ in $\Omega$.
Regarding to $\|\nabla u_0\|_{L^2(O_{4\varepsilon})}$, by
co-area formula, we carry out the following computations:
\begin{equation}\label{pri:3.13}
\begin{aligned}
\|\nabla u_0\|^2_{L^2(O_{4\varepsilon})}&=\int_{0}^{4\varepsilon}
\int_{\partial \Sigma_{t}}|\nabla u_{0}|^2dS_tdt
\lesssim\varepsilon\sup_{0<t<4\varepsilon}\int_{\partial \Sigma_{t}}|\nabla u_{0}|^2dS_t\\
&\lesssim \varepsilon\Big(\int_{\Omega}|\nabla u_{0}|^2dx+\int_{\Omega}|\nabla^2 u_{0}|^2dx\Big)\\
&\lesssim^{\eqref{pri:2.12},\eqref{pri:2.11}}
\varepsilon
\Big(\|F\|^2_{L^{2}(\Omega)}
+\|g\|^2_{H^{3/2}(\partial\Omega)}\Big),
\end{aligned}
\end{equation}
where we employ the trace theorem near the boundary
in the third step, i.e.,
\begin{equation*}
\int_{\partial \Sigma_t}|\nabla u_0|^2dS\lesssim\int_{\Omega}|\nabla u_0|^2dx+\int_{\Omega}|\nabla^2 u_0|^2dx
\end{equation*}
uniformly holds for $0<t\ll1$. By inserting \eqref{pri:3.13} into \eqref{f:2.8},
we arrive at
\begin{equation*}
\begin{aligned}
\|u_\varepsilon-u_0\|_{L^2(\Omega_{\varepsilon})}
&\lesssim
\varepsilon^{1/2}\Big\{\|F\|_{L^{2}(\Omega)}
+\|g\|_{H^{3/2}(\partial\Omega)}\Big\}
+\varepsilon\|g\|_{H^{3/2}(\partial\Omega)}+
\varepsilon^{s}\|\tilde{F}\|_{H^{s}(\mathbb{R}^d)}\\
&\lesssim
\varepsilon^{1/2}
\Big\{\|g\|_{H^{3/2}(\partial\Omega)}
+\|\tilde{F}\|_{H^{1/2}(\mathbb{R}^d)}\Big\}
\lesssim^{\eqref{f:3.16}} \varepsilon^{1/2}
\Big\{\|g\|_{H^{3/2}(\partial\Omega)}
+\|F\|_{H^{1/2}(\Omega)}\Big\}
\end{aligned}
\end{equation*}
This gives the stated estimate \eqref{pri:1.2}.

Then we proceed to show the estimate \eqref{pri:1.6}, and first claim that
\begin{equation}\label{pri:3.11}
  \|\nabla^2 u_{0}\|_{L^{2}(\Sigma_{3\varepsilon})}
  \lesssim\varepsilon^{-\frac{1}{2}-\frac{1}{p}}\|\nabla u_{0}\|_{L^{p}(\Omega)},
\end{equation}
and the details of the proof of \eqref{pri:3.11} can be found in \cite[Lemma 3.9]{W} (originally from \cite[Lemma 6.1.5]{S4}).
In view of Lemma \ref{lemma:3.2}, we have
\begin{equation}\label{f:3.17}
\begin{aligned}
\|u_{\varepsilon}-u_{0}\|_{L^2(\Omega_{\varepsilon})}
& \lesssim^{\eqref{pri:3.11}} \|\nabla u_{0}\|_{L^{2}(O_{4\varepsilon})}+\varepsilon^{\frac{1}{2}-\frac{1}{p}}
\|\nabla u_{0}\|_{L^{p}(\Omega)}
+\varepsilon^{s}\|\tilde{F}\|_{H^{s}(\mathbb{R}^d)}\\
 &\lesssim \varepsilon^{\frac{1}{2}-\frac{1}{p}}
\|\nabla u_{0}\|_{L^{p}(\Omega)}
+\varepsilon^{s}\|\tilde{F}\|_{H^{s}(\mathbb{R}^d)}
\end{aligned}
\end{equation}
for some $p>2$,
in which the second step follows from H\"{o}lder's inequality.
Let $\sigma=1/2-1/p$ and one may choose $p>2$ such that
$0<p-2\ll 1$.  Then it follows from
Theorem \ref{thm:2.20} that
\begin{equation}\label{f:3.18}
\begin{aligned}
\|\nabla u_{0}\|_{L^{p}(\Omega)}
&\lesssim^{\eqref{pri:2.13}} \Big\{\|\nabla (-\Delta)^{-1}\tilde{F}^{0}\|_{L^{p}(\Omega)}
+\|g\|_{W^{1-1/p,p}(\partial\Omega)}\Big\}\\
&\lesssim \Big\{\|F\|_{L^{\frac{pd}{d+p}}(\Omega)}
+\|g\|_{W^{1-1/p,p}(\partial\Omega)}\Big\}
\lesssim \Big\{\|F\|_{L^{\frac{2d}{d-2\sigma}}(\Omega)}
+\|g\|_{W^{1-1/p,p}(\partial\Omega)}\Big\}
\\
&\lesssim \Big\{\|\tilde{F}\|_{H^{\sigma}(\mathbb{R}^d)}
+\|g\|_{W^{1-1/p,p}(\partial\Omega)}\Big\},
\end{aligned}
\end{equation}
where the operator $\nabla(\Delta)^{-1}$ defines the Riesz potential of order $1$, and $\tilde{F}^0$ is zero-extension of $F$ to $\mathbb{R}^d$.
In the second step, we use fractional integral estimates
(see \cite[Theorem 7.25]{MGLM}). In the third one,
we merely employ H\"older's inequality by noting that
$\frac{pd}{d+p}<2<\frac{2d}{d-2\sigma}$, while
we employ fractional Sobolev inequality (see for example \cite[Theorem 6.5]{DPV}) in the last step.
Thus, plugging the estimate $\eqref{f:3.18}$ back into
$\eqref{f:3.17}$ and setting $s=\sigma$ in
$\eqref{f:3.17}$, we finally arrive at
\begin{equation*}
\|u_{\varepsilon}-u_{0}\|_{L^2(\Omega_{\varepsilon})}
\lesssim \varepsilon^\sigma\Big\{ \|g\|_{W^{1-1/p,p}(\partial\Omega)}
+\|\tilde{F}\|_{H^{\sigma}(\mathbb{R}^d)}\Big\}
\lesssim^{\eqref{f:3.16}} \varepsilon^\sigma\Big\{ \|g\|_{W^{1-1/p,p}(\partial\Omega)}
+\|F\|_{H^{\sigma}(\Omega)}\Big\},
\end{equation*}
and this closes the whole proof.
\qed

\subsection{The proof of Theorem $\ref{thm:1.5}$}

In this subsection,
we omit the subscript $\lambda$
of $u_{\varepsilon,\lambda}$ and $u_{0,\lambda}$ in
Theorem $\ref{thm:1.5}$ for the ease of the statement.

\begin{lemma}[weak formulation]
Let $\Omega=\mathbb{R}^d$ and $\omega$
be a Lipschitz domain.
Assume that $A$ satisfies $\eqref{a:1}$ and $\eqref{a:2}$.
Suppose that $u_\varepsilon\in H^{1}(\Omega_{\varepsilon})~\text{and}~ u_0\in H^1(\mathbb{R}^d)$
satisfy equations $\emph{(i)}$ and $\emph{(ii)}$ in
$\eqref{pde:1.5}$, respectively.
Let $w_{\varepsilon}=u_{\varepsilon}-v_{\varepsilon}$ and
$v_{\varepsilon}=u_{0}+\varepsilon N(x/\varepsilon,\varphi)$
with $\varphi\in H^1(\mathbb{R}^d;\mathbb{R}^d)$.
Then we have
\begin{equation}\label{eq:3.2}
\begin{aligned}
\lambda\int_{\Omega_\varepsilon}|w_\varepsilon|^2 dx
&+ \int_{\Omega_\varepsilon} \big[A(y,\nabla u_\varepsilon)
-A(y,\nabla v_\varepsilon)\big]\cdot\nabla w_\varepsilon dx\\
&= \int_{\mathbb{R}^d}(l_\varepsilon^{+}-\theta)(F-\lambda u_0)
\tilde{w}_\varepsilon dx
 -\varepsilon\lambda
\int_{\mathbb{R}^d}l_\varepsilon^{+}N(y,\varphi)\tilde{w}_\varepsilon dx\\
&\qquad\qquad + \int_{\mathbb{R}^d}
\Big\{\theta\widehat{A}(\nabla u_0)
-l_\varepsilon^{+}A\big(y,\varphi
+\nabla_y N(y,\varphi)\big)\Big\}\nabla\tilde{w}_\varepsilon dx\\
&\qquad\qquad\qquad\qquad + \int_{\mathbb{R}^d}
l_\varepsilon^{+}\Big\{A\big(y,\varphi
+\nabla_y N(y,\varphi)\big)
-A(y,\nabla v_\varepsilon)\Big\}\nabla\tilde{w}_\varepsilon dx
\end{aligned}
\end{equation}
in which $y=x/\varepsilon$, and $\tilde{w}_\varepsilon$ is the extension of
$w_\varepsilon$ in the sense of Lemma $\ref{extensiontheory}$ and Remark $\ref{remark:2.2}$.
\end{lemma}

\begin{proof}
The computation is similar to that given in $\eqref{eq:3.1}$, and start from
\begin{equation*}
\begin{aligned}
\lambda \int_{\Omega_\varepsilon} u_\varepsilon w_\varepsilon dx
+ \int_{\Omega_\varepsilon}A(x/\varepsilon,\nabla u_\varepsilon)
\nabla w_\varepsilon dx
&= \int_{\Omega_\varepsilon} F w_\varepsilon dx
= \int_{\mathbb{R}^d} l_\varepsilon^{+}F\tilde{w}_\varepsilon dx;\\
\lambda \int_{\mathbb{R}^d} u_0 \tilde{w}_\varepsilon dx
+ \int_{\mathbb{R}^d}\widehat{A}(\nabla u_0)
\nabla \tilde{w}_\varepsilon dx
&= \int_{\mathbb{R}^d} F \tilde{w}_\varepsilon dx.
\end{aligned}
\end{equation*}
and note that $
\int_{\Omega_\varepsilon} u_\varepsilon w_\varepsilon dx
=\int_{\mathbb{R}^d}l_\varepsilon^{+} \tilde{u}_\varepsilon \tilde{w}_\varepsilon dx$, where $\tilde{u}_\varepsilon$
and $\tilde{w}_\varepsilon$ are the extension functions of
$u_\varepsilon, w_\varepsilon$, respectively.
The rest of the calculations is standard and we
do not reproduce it here.
\end{proof}

\noindent\textbf{Proof of Theorem $\ref{thm:1.5}$}.
We first take $\varphi = \nabla u_0$ in the weak formulation
$\eqref{eq:3.2}$.
On account of the assumption $\eqref{a:1}$, we have
\begin{equation}\label{f:3.14}
\text{L.H.S.~of~}\eqref{eq:3.2}
\geq
\mu_0 \int_{\Omega_\varepsilon}|\nabla w_\varepsilon|^2 dx
+ \lambda\int_{\Omega_\varepsilon}|w_\varepsilon|^2 dx.
\end{equation}
We denote the right-hand side of $\eqref{eq:3.2}$ by
$I_1-\varepsilon\lambda I_2 + I_3 + I_4$ in order.
To make the estimated constant independent of
$\lambda$,
we split the proof into two cases: (1) $d\geq 3$;
(2) $d=2$.

We first show the proof under the assumption $d\geq 3$.
Thanks to the auxiliary equation $\eqref{auxi1}$, we have
\begin{equation*}
\begin{aligned}
I_1 := \int_{\mathbb{R}^d}(l_\varepsilon^{+}-\theta)(F-\lambda u_0)
\tilde{w}_\varepsilon dx
&=-\int_{\mathbb{R}^d}\Delta_y\Psi(y)(F-\lambda u_0)
\tilde{w}_\varepsilon dx \\
&=\varepsilon\int_{\mathbb{R}^d}\nabla_y\Psi(y)
\nabla(F-\lambda u_0)
\tilde{w}_\varepsilon dx
+\varepsilon\int_{\mathbb{R}^d}\nabla_y\Psi(y)
(F-\lambda u_0)
\nabla\tilde{w}_\varepsilon dx,
\end{aligned}
\end{equation*}
and there holds
\begin{equation*}
\begin{aligned}
|I_1|
&\lesssim^{\eqref{f:3.2}} \varepsilon
\bigg\{
\Big(\|\nabla F\|_{L^{\frac{2d}{d+2}}(\mathbb{R}^d)}
+\|F-\lambda u_0\|_{L^{2}(\mathbb{R}^d)}
\Big)\|\nabla \tilde{w}_\varepsilon\|_{L^{2}(\mathbb{R}^d)}
+\lambda\|\nabla u_0\|_{L^{2}(\mathbb{R}^d)}
\|\tilde{w}_\varepsilon\|_{L^{2}(\mathbb{R}^d)}\bigg\}\\
&\lesssim^{\eqref{pri:2.23}} \varepsilon
\bigg\{
\Big(\|\nabla F\|_{L^{\frac{2d}{d+2}}(\mathbb{R}^d)}
+\|F-\lambda u_0\|_{L^{2}(\mathbb{R}^d)}
\Big)\|\nabla w_\varepsilon\|_{L^{2}(\Omega_\varepsilon)}
+\lambda\|\nabla u_0\|_{L^{2}(\mathbb{R}^d)}
\|w_\varepsilon\|_{H^{1}(\Omega_\varepsilon)}\bigg\}\\
&\lesssim^{\eqref{f:3.9},\eqref{f:3.10}} \varepsilon
\|\nabla F\|_{L^{\frac{2d}{d+2}}(\mathbb{R}^d)}
\bigg\{\sqrt{\mu_0}\|\nabla w_\varepsilon\|_{L^{2}(\Omega_\varepsilon)}
+\sqrt{\lambda}
\|w_\varepsilon\|_{H^{1}(\Omega_\varepsilon)}\bigg\}\\
&\lesssim  \varepsilon
\|\nabla F\|_{L^{\frac{2d}{d+2}}(\mathbb{R}^d)}
\bigg\{\sqrt{\mu_0}
\|\nabla w_\varepsilon\|_{L^{2}(\Omega_\varepsilon)}
+\sqrt{\lambda}
\|w_\varepsilon\|_{L^{2}(\Omega_\varepsilon)}\bigg\}.
\end{aligned}
\end{equation*}
where we also employ Sobolev's inequality in the first
step, and the fact that
$0<\lambda\leq \mu_0$ in the last one. In the
above computations we mention that
\begin{equation}\label{f:3.9}
\begin{aligned}
\|F-\lambda u_0\|_{L^{2}(\mathbb{R}^d)}
&\lesssim^{(\text{ii})~\text{of~}\eqref{pde:1.5}} \|F\|_{L^{2}(\mathbb{R}^d)}
+ \|F+\nabla\cdot\widehat{A}(\nabla u_0)\|_{L^{2}(\mathbb{R}^d)}\\
&\lesssim^{\eqref{pri:2.3}} \|F\|_{L^{2}(\mathbb{R}^d)}
+ \|\nabla^2u_0\|_{L^2(\mathbb{R}^d)}
\lesssim^{\eqref{pri:7.3}} \|F\|_{L^{2}(\mathbb{R}^d)}
\lesssim\|\nabla F\|_{L^{\frac{2d}{d+2}}(\mathbb{R}^d)},
\end{aligned}
\end{equation}
and
\begin{equation}\label{f:3.10}
\begin{aligned}
\sqrt{\lambda}\|\nabla u_0\|_{L^{2}(\mathbb{R}^d)}
\lesssim^{\eqref{pri:7.3}} \|F\|_{L^{2}(\mathbb{R}^d)}
\lesssim\|\nabla F\|_{L^{\frac{2d}{d+2}}(\mathbb{R}^d)}.
\end{aligned}
\end{equation}
Hence, it follows from Young's inequality that
\begin{equation}\label{f:3.11}
|I_1|\lesssim \varepsilon^2
\|\nabla F\|_{L^{\frac{2d}{d+2}}(\mathbb{R}^d)}^2
+ \delta\lambda\|w_\varepsilon\|_{L^{2}(\Omega_\varepsilon)}^2
+ \delta\mu_0\|\nabla w_\varepsilon\|_{L^{2}(\Omega_\varepsilon)}^2,
\end{equation}
where we set $0<\delta<1/10$ throughout the present proof.

Then we turn to $I_2$, and
\begin{equation*}
I_2 := \int_{\mathbb{R}^d}l_\varepsilon^{+}N(y,\nabla u_0)\tilde{w}_\varepsilon dx
\lesssim
\|N(y,\nabla u_0)\|_{L^2(\Omega_\varepsilon)}
\|w_\varepsilon\|_{L^2(\Omega_\varepsilon)}
\lesssim^{\eqref{f:3.1}}
\|\nabla u_0\|_{L^2(\mathbb{R}^d)}
\|w_\varepsilon\|_{L^2(\Omega_\varepsilon)}.
\end{equation*}
This implies that
\begin{equation}\label{f:3.12}
\begin{aligned}
\big|\varepsilon\lambda I_2\big|
\lesssim^{\eqref{f:3.10}} \varepsilon\sqrt{\lambda}\|\nabla F\|_{L^{\frac{2d}{d+2}}(\mathbb{R}^d)}
\|w_\varepsilon\|_{L^2(\Omega_\varepsilon)}
\lesssim \varepsilon^2
\|\nabla F\|_{L^{\frac{2d}{d+2}}(\mathbb{R}^d)}^2
+ \delta\lambda\|w_\varepsilon\|_{L^{2}(\Omega_\varepsilon)}^2.
\end{aligned}
\end{equation}

Now, we proceed to handle the term $I_3$, which appeals to flux correctors.
\begin{equation*}
\begin{aligned}
I_3:& =
\int_{\mathbb{R}^d}
\Big\{\theta\widehat{A}(\nabla u_0)
-l_\varepsilon^{+}A\big(y,\nabla u_0
+\nabla_y N(y,\nabla u_0)\big)\Big\}\nabla\tilde{w}_\varepsilon dx\\
& = \int_{\mathbb{R}^d}b(y,\nabla u_0)
  \nabla\tilde{w}_{\varepsilon}dx
  =^{\eqref{eq:2.1}}-\varepsilon\int_{\Omega}\frac{\partial}{\partial \xi_{k}}\{E_{ji}(y,\nabla u_0)\}
  \nabla_{kj}^2u_0
  \nabla_i \tilde{w}_{\varepsilon}dx\\
  &\lesssim^{\eqref{pri:2.17}} \varepsilon\|\nabla^2 u_0\|_{L^2(\mathbb{R}^d)}
  \|\nabla\tilde{w}_\varepsilon\|_{L^2(\mathbb{R}^d)}
\end{aligned}
\end{equation*}

Concerning the last term $I_4$,  we have
\begin{equation*}
\begin{aligned}
I_4:&=\int_{\mathbb{R}^d}
l_\varepsilon^{+}\Big\{A\big(y,\nabla u_0
+\nabla_y N(y,\nabla u_0)\big)
-A(y,\nabla v_\varepsilon)\Big\}\nabla\tilde{w}_\varepsilon dx\\
&\lesssim^{\eqref{a:1}} \varepsilon\|\nabla_\xi N(y,\nabla u_0)\nabla^2 u_0\|_{L^2(\mathbb{R}^d)}
\|\nabla\tilde{w}_\varepsilon\|_{L^2(\mathbb{R}^d)}
\lesssim^{\eqref{pri:2.15}} \varepsilon\|\nabla^2 u_0\|_{L^2(\mathbb{R}^d)}
\|\nabla\tilde{w}_\varepsilon\|_{L^2(\mathbb{R}^d)}.
\end{aligned}
\end{equation*}
Thus, combining the estimates of $I_3$ and $I_4$ we
arrive at
\begin{equation}\label{f:3.13}
 I_3 + I_4
 \lesssim^{\eqref{pri:7.3}} \varepsilon\|F\|_{L^2(\mathbb{R}^d)}\|\nabla w_\varepsilon\|_{L^2(\Omega_\varepsilon)}
 \lesssim \varepsilon^2\|\nabla F\|_{L^{\frac{2d}{d+2}}(\mathbb{R}^d)}^2
 +\delta\mu_0\|\nabla w_\varepsilon\|_{L^2(\Omega_\varepsilon)}^2,
\end{equation}
where we use Sobolev's inequality and Young's inequality in the last step. Collecting the estimates
$\eqref{f:3.11},\eqref{f:3.12},\eqref{f:3.13}$ one may obtain
\begin{equation*}
\text{R.H.S.~of~}\eqref{eq:3.2}
\lesssim \varepsilon^2
\|\nabla F\|_{L^{\frac{2d}{d+2}}(\mathbb{R}^d)}^2
+ \delta\lambda\|w_\varepsilon\|_{L^{2}(\Omega_\varepsilon)}^2
+ \delta\mu_0\|\nabla w_\varepsilon\|_{L^{2}(\Omega_\varepsilon)}^2
\end{equation*}
and this together with $\eqref{f:3.14}$ leads to
\begin{equation*}
 \sqrt{\lambda/\mu_0}\| w_\varepsilon\|_{L^{2}(\Omega_\varepsilon)}
 + \|\nabla w_\varepsilon\|_{L^{2}(\Omega_\varepsilon)}
 \lesssim \varepsilon
 \|\nabla F\|_{L^{\frac{2d}{d+2}}(\mathbb{R}^d)},
\end{equation*}
where the up to constant never depends on $\lambda$. So,
it follows from Sobolev's inequality that
\begin{equation*}
\|w_\varepsilon\|_{L^{\frac{2d}{d-2}}(\Omega_\varepsilon)}
\leq \| \tilde{w}_\varepsilon\|_{L^{\frac{2d}{d-2}}(\mathbb{R}^d)}
\lesssim   \| \nabla \tilde{w}_\varepsilon\|_{L^{2}(\mathbb{R}^d)}
\lesssim^{\eqref{pri:2.23}}
   \|\nabla w_\varepsilon\|_{L^{2}(\Omega_\varepsilon)}
   \lesssim \varepsilon\|\nabla F\|_{L^{\frac{2d}{d+2}}(\mathbb{R}^d)},
\end{equation*}
which finally gives the desired estimate $\eqref{pri:1.7}$
(the proof is similar to the estimate $\eqref{pri:3.2}$ derived from $\eqref{pri:3.1}$).

In the case of $d=2$, we do not seek for the estimated constant independent of $\lambda$, which makes the
proof to be simple. We almost merely employ H\"older's inequality
to handle the right-hand side of $\eqref{eq:3.2}$, and without showing the details we present that
\begin{equation*}
\text{R.H.S.~of~}\eqref{eq:3.2}
\lesssim \varepsilon\|F\|_{H^1(\mathbb{R}^2)}
\|w_\varepsilon\|_{H^{1}(\Omega_\varepsilon)}
\end{equation*}
and this coupled with the estimate $\eqref{f:3.14}$ implies
$\|w_\varepsilon\|_{H^1(\Omega_\varepsilon)}
 \lesssim \varepsilon \|F\|_{H^1(\mathbb{R}^2)}
$. Hence, by the Sobolev embedding theorem and extension results $\eqref{pri:2.23}$ we consequently reach the
stated estimate $\eqref{pri:1.8}$ and we have completed the whole proof.
\qed

\section{Interior estimates}\label{sec:4}
\begin{lemma}[approximating lemma I]\label{lemma:4.1}
Let $\varepsilon\leq r<(1/2)$.
Assume the same conditions as in Theorem $\ref{thm:1.2}$.
Let $u_\varepsilon\in H^1(B^{\varepsilon}_{2r})$ be a weak solution of
\begin{equation}\label{pri:4.011}
 \left\{ \begin{aligned}
 \mathcal{L}_{\varepsilon}u_\varepsilon  &= 0
 &~&\emph{in}~~ B^{\varepsilon}_{2r},\\
 \sigma_{\varepsilon}(u_{\varepsilon})&= 0
 &~&\emph{on}~~ \partial B^{\varepsilon}_{2r}|_{B_{2r}}.
  \end{aligned}\right.
\end{equation}
 Then there exists
$w\in H^1(B_r)$ such that
$\mathcal{L}_0 w  = 0$
and,
\begin{equation}\label{pri:4.1}
\begin{aligned}
 \Big(\dashint_{B^{\varepsilon}_r} |u_\varepsilon - w|^2  \Big)^{1/2}
 \lesssim\left(\frac{\varepsilon}{r}\right)^{\sigma}
 &\Big(\dashint_{B^{\varepsilon}_{2r}}|u_\varepsilon|^2 \Big)^{1/2},
\end{aligned}
\end{equation}
where $\sigma=\frac{1}{2}-\frac{1}{p}$ and $0<p-2\ll1$ is the same $p$ as in \eqref{pri:1.6} and Theorem \ref{thm:2.20}.
\end{lemma}

\begin{proof}
The main idea may be found in \cite[Lemma 11.2]{SZ1}.
By rescaling argument one may assume $r=1$.
Before proceeding our proof, we need to extend
$u_{\varepsilon}\in H^{1}(B^{\varepsilon}_{3/2})$
to $H^{1}(B_3)$. Since $u_{\varepsilon}$
does not vanish on
$\partial B^\varepsilon_{3/2}|_{\partial B_{3/2}}$,
we can not apply the extension operator
in Lemma \ref{extensiontheory} directly.
The idea is to multiply a cut-off function at first. Suppose that
\begin{equation*}
  \rho\in C^{1}_{0}(B_{7/4}),~ \rho(x)=1~\text{on}~B_{3/2},\quad\text{and}\quad|\nabla \rho|\lesssim  1.
\end{equation*}
Then we have $\rho u_{\varepsilon}\in H^1(B^{\varepsilon}_{7/4},
\partial B^\varepsilon_{7/4}|_{\partial B_{7/4}})$. By Lemma \ref{extensiontheory}, we know that the extension function $\tilde{u}_{\varepsilon}$ satisfying that  $\tilde{u}_{\varepsilon}(x)=\rho(x)u_{\varepsilon}(x)
=u_{\varepsilon}(x)\text{~for~}x\in B^{\varepsilon}_{3/2}$ and $\tilde{u}_{\varepsilon}\in H^{1}_0(B_{3})$. Moreover,
there holds
\begin{equation*}%\label{extension1}
  \begin{aligned}
  \|\tilde{u}_{\varepsilon}\|_{H^{1}_0(B_3)}
 \lesssim^{\eqref{pri:2.23}} \|\rho u_{\varepsilon}\|_{H^1(B^{\varepsilon}_{7/4})}
& \lesssim \|u_{\varepsilon}\|_{L^{2}(B^{\varepsilon}_{2})}
+\|\nabla u_{\varepsilon}\|_{L^{2}(B^{\varepsilon}_{7/4})}
\lesssim^{\eqref{pri:2.14}} \|u_{\varepsilon}\|_{L^{2}(B^{\varepsilon}_{2})}.
\end{aligned}
\end{equation*}

%On account of \eqref{extension1} and co-area formula, it is true that there exists $\bar{r}\in[1,3/2]$ such that
%\begin{equation}\label{f:4.2}
%\|\tilde{u}_\varepsilon\|_{H^{1}(\partial B(0,\bar{r}))}
%\leq \|u_{\varepsilon}\|_{L^{2}(B_{\varepsilon}(0,2))}.
%\end{equation}

Now, for $\bar{r}\in[1,3/2]$, we consider
%\begin{equation*}
%  \left\{\begin{aligned}
%  \mathcal{L}_\varepsilon v_\varepsilon &= 0~&\text{in}&~B_{\varepsilon}(0,\bar{r}),\\
%  \sigma_{\varepsilon}(v_{\varepsilon})& =0~&\text{on}&~
%  S_{\varepsilon}(\bar{r}),\\
% v_{\varepsilon}&=(\tilde{u}_{\varepsilon})_{\delta}~&\text{on}&~
%  \Gamma_{\varepsilon}(\bar{r}),
%  \end{aligned}\right.
%\end{equation*}
%and
\begin{equation*}
  \left\{\begin{aligned}
 \mathcal{L}_0 w &= 0
 & &\text{in}~~B_{\bar{r}},\\
 w&=\tilde{u}_{\varepsilon}
 & &\text{on}~~
  \partial B_{\bar{r}}.
  \end{aligned}\right.
\end{equation*}
It follows from the estimate \eqref{pri:1.6} that
\begin{equation}\label{pri:4.3}
\begin{aligned}
  \|u_{\varepsilon}-w\|_{L^{2}(B^{\varepsilon}_{\bar{r}})}
  &\lesssim\varepsilon^{\frac{1}{2}-\frac{1}{p}}
  \|\tilde{u}_{\varepsilon}\|_{W^{1-1/p,p}(\partial B_{\bar{r}})}\\
  &\lesssim
  \varepsilon^{\frac{1}{2}-\frac{1}{p}}
  \|\tilde{u}_{\varepsilon}\|_{W^{1,p}(B_{\bar{r}})}
  \lesssim
  \varepsilon^{\frac{1}{2}-\frac{1}{p}}
  \|\nabla \tilde{u}_{\varepsilon}\|_{L^{p}(B_{3})}\\
  &\lesssim^{\eqref{pri:2.22}}\varepsilon^{\sigma}
  \Big(\|\nabla u_{\varepsilon}\|_{L^{p}(B^{\varepsilon}_{7/4})}
  +\|u_{\varepsilon}\|_{L^{p}(B^{\varepsilon}_{7/4})}\Big),
  \end{aligned}
\end{equation}
in which $\sigma=1/2-1/p$, $p$ is the same in Theorem \ref{thm:2.20}, the second inequality follows from
the trace theorem, and the third one from
Poincar\'e's inequality.
Due to Meyer's estimates (also known as self-improvement properties),
we have
\begin{equation}\label{pri:4.6}
 \|\nabla u_{\varepsilon}\|_{L^{p}(B^{\varepsilon}_{7/4})}
\lesssim^{\eqref{pri:7.6}} \|\nabla u_{\varepsilon}\|_{L^2(B^{\varepsilon}_{15/8})}
\lesssim^{\eqref{pri:2.14}}
\| u_{\varepsilon}\|_{L^2(B^{\varepsilon}_{2})}.
\end{equation}
Appealing to Lemma \ref{lemma:2.20}, it follows that
\begin{equation}\label{pri:4.7}
\begin{aligned}
\|u_{\varepsilon}\|_{L^{p}(B^{\varepsilon}_{7/4})}
\lesssim\|u_{\varepsilon}-c_{r}\|_{L^{p}(B^{\varepsilon}_{7/4})}+c_r
&\lesssim^{\eqref{pri:2.20}}
\|\nabla u_{\varepsilon}\|_{L^{2}(B^{\varepsilon}_{15/8})}+c_r\\
&\lesssim\|\nabla u_{\varepsilon}\|_{L^{2}(B^{\varepsilon}_{15/8})}
+\|u_{\varepsilon}\|_{L^{2}(B^{\varepsilon}_{2})},
\end{aligned}
\end{equation}
where the last step follows from the fact that constant $c_r$ in \eqref{pri:2.20} is the average of $u_{\varepsilon}$ in a domain smaller than $B_2^{\varepsilon}$.
By inserting \eqref{pri:4.6} and \eqref{pri:4.7} into \eqref{pri:4.3},
we arrive at
\begin{equation}\label{f:4.10}
\|u_{\varepsilon}-w\|_{L^{2}(B^{\varepsilon}_1)}\lesssim
\varepsilon^{\sigma}
\| u_{\varepsilon}\|_{L^2(B^{\varepsilon}_{2})}.
\end{equation}

To complete the whole argument, we
appeal to rescaling arguments. Let $u(x) = u(ry)$ and
$x=ry$ with $r\geq \varepsilon$, where
$x\in B^\varepsilon_{2r}$ and
$y\in B^{\varepsilon^\prime}_2$
with $\varepsilon^\prime = \varepsilon/r$. Let
$u_r(y):= \frac{1}{r}u(ry)$, and then $u(x)=r u_r(x/r)$.
Then it is not hard to verify that
\begin{equation*}
0 = \nabla_x\cdot A(x/\varepsilon,\nabla_x u)
%=r^{-1}\nabla_y\cdot
%A(ry/\varepsilon,\nabla_y u_r)
\quad\text{in}\quad B^{\varepsilon}_{2r};
\quad\Rightarrow\quad
\nabla_y\cdot
A(y/\varepsilon^\prime,\nabla_y u_r) = 0
\quad\text{in}\quad B^{\varepsilon^\prime}_{2},
\end{equation*}
and $A$ is exactly the same one by absorbing
the scale $r$ into the parameter $\varepsilon$.
Similarly, we can assume $v_r(y):=\frac{1}{r}v(ry)$, and
infer that $\nabla_y \widehat{A}(\nabla_y v_r) = 0$ in
$B_1$ (clearly $v_r$ and $u_r$ admit the same scale).
Hence, it follows from $\eqref{f:4.10}$ that
\begin{equation*}
\|u_{r}-v_{r}\|_{L^{2}(B^{\varepsilon^\prime}_1)}
\lesssim (\varepsilon^\prime)^{\sigma}
\|u_{r}\|_{L^2(B^{\varepsilon^\prime}_2)}.
\end{equation*}
Scaling back the above estimate leads to the desired estimate $\eqref{pri:4.1}$,
and we have completed the whole proof.
\end{proof}

Before we proceed further,
we recall the definition of $G(r,v)$ and
$G_{\varepsilon}(r,v)$ as follows.
%as the following:
\begin{equation*}
\begin{aligned}
G(r,v) &= \frac{1}{r}\inf_{M\in\mathbb{R}^{d}\atop c\in\mathbb{R}}
\Big(\dashint_{B(0,r)}|v-Mx-c|^2\Big)^{\frac{1}{2}};\\
G_{\varepsilon}(r,v)
&= \frac{1}{r}\inf_{M\in\mathbb{R}^{d}\atop c\in\mathbb{R}}
\Big(\dashint_{B^{\varepsilon}(0,r)}|v-Mx-c|^2\Big)^{\frac{1}{2}}.
\end{aligned}
\end{equation*}
%\begin{equation}
%\begin{aligned}
%
%\end{aligned}
%\end{equation}
\begin{lemma}[interior comparing at large-scales]\label{lemma:2.6}
  Suppose that $\mathcal{L}_{0}(v)=0 ~\text{in}~ B_{2r}$, $r\geq \varepsilon$, it holds that
  \begin{equation}\label{g(r,v)}
    G(r,v)\lesssim  G_{\varepsilon}(2r,v),
  \end{equation}
in which the up to constant depending on $\mu_{0},\mu_{1}, d, \mathfrak{g}^{\omega}$ and $\omega$.
\end{lemma}
\begin{proof}
The main idea is similar to that in Lemma $\ref{lemma:2.20}$
(also see in \cite{BR}).
First, we decompose the domain $B(0,r)$.
$T_{\varepsilon}:=\{z\in \mathbb{Z}^{d}:\varepsilon(Y+z)\cap B(0,r)\neq\emptyset\}$.
Fix $z\in T_{\varepsilon}$, and we denote the bounded,
connected components of $\mathbb{R}^{d}
\backslash\omega$ by $\{H_{k}\}_{k=1}^N$
with $H_{k}\cap(Y+z)\neq\emptyset.$
Define cut-off function $\varphi_{k}\in C^{\infty}_0(Y^{*}(z))$ as
\begin{equation*}
  \left\{\begin{aligned}
 & \varphi_k(x)=1,~\text{if}~x\in H_k,\\
 & \varphi_k(x)=0,~\text{if~dist}(x,H_k)>\frac{1}{4}\mathfrak{g}^w,\\
 & |\nabla \varphi_k|\leq C,
  \end{aligned}\right.
\end{equation*}
where $\mathfrak{g}^{\omega}$ is
defined in \eqref{g}, C depends on $\omega$, and
\begin{equation*}
  Y^*(z):=\bigcup_{j=1}^{3^d}(Y+z_j),z_j\in\mathbb{Z}^d~
  \text{and}~|z-z_j|\leq\sqrt{d}.
\end{equation*}
(In the absence of confusion, we also write it as $Y^*$.)
Set $\varphi=\sum_{k=1}^N\varphi_k\in C^{\infty}_0(Y^*)$, We note that
\begin{equation*}
  \varphi(1-\varphi)=0~\text{in}~Y^*\backslash\omega,~\text{hence}~\nabla \varphi=0~\text{in}~Y^*\backslash\omega.
\end{equation*}
In the case of $\mathcal{L}_{0}(v)=0$ in $Y^*$,
for any $M\in\mathbb{R}^d,c\in\mathbb{R}$,
set $\tilde{v}(x)=v(x)-Mx-c$ and then we claim that
\begin{equation}\label{claim1}
  \int_{Y+z}|\tilde{v}|^{2}dx
  \lesssim \int_{Y^{*}\cap\omega}|\tilde{v}|^{2}dx,
\end{equation}
where the up to constant depends only on $\mu_{0},\mu_{1},d,\omega$ and is independent of $z$. By employing Poincar\'{e}'s inequality, it follows that
\begin{equation}\label{pri:4.16}
  \int_{(Y+z)\backslash\omega}|\tilde{v}(x)|^{2}dx\leq \sum_{k=1}^{N}\int_{H_{k}}|\varphi(x)\tilde{v}(x)|^{2}dx
  \lesssim \int_{Y^{*}}|\nabla(\tilde{v}\varphi)|^{2}dx.
\end{equation}
After a routine calculation, one may have
\begin{equation}\label{pri:4.17}
\int_{Y^{*}}|\nabla(\tilde{v}\varphi)|^{2}dx
\lesssim \int_{Y^{*}}|\nabla\tilde{v}|^{2}|\varphi|^{2}dx
+\int_{Y^{*}}|\nabla\varphi|^{2}|\tilde{v}|^{2}dx.
\end{equation}
The second term in the right-hand side
of the above inequality is good, and we
just need to deal with the first term.
According to $\mathcal{L}_{0}(v)=0$ in $Y^*$,
there holds
\begin{equation*}
\begin{aligned}
\int_{Y^{*}}|\nabla\tilde{v}|^{2}|\varphi|^{2}dx
& =\int_{Y^{*}}|\nabla v-M|^{2}|\varphi|^{2}dx
\leq^{\eqref{pri:2.3}} C \int_{Y^{*}}|\varphi|^{2}[\widehat{A}(\nabla v)-\widehat{A}(M)]\nabla\tilde{v}dx\\
& =-2C\int_{Y^{*}}\varphi\tilde{v}[\widehat{A}(\nabla v)-\widehat{A}(M)]\nabla\varphi dx-C\int_{Y^{*}}\text{div}[\widehat{A}(\nabla v)-\widehat{A}(M)]|\varphi|^{2}\tilde{v}dx,\\
& =-2C\int_{Y^{*}}\varphi\tilde{v}[\widehat{A}(\nabla v)-\widehat{A}(M)]\nabla\varphi dx
\end{aligned}
\end{equation*}
in which we employ divergence theorem in the third step. It follows from \eqref{pri:2.3} and Young's inequality that
\begin{equation*}
\begin{aligned}
\int_{Y^{*}}|\nabla\tilde{v}|^{2}|\varphi|^{2}dx
\lesssim\int_{Y^{*}}|\varphi\nabla\tilde{v}|\cdot|\tilde{v}
\nabla\varphi|dx
 \lesssim \delta\int_{Y^{*}}|\varphi\nabla\tilde{v}|^{2}dx
+C_{\delta}\int_{Y^{*}}|\tilde{v}\nabla\varphi|^{2}dx.
\end{aligned}
\end{equation*}
By choosing a suitable $\delta$, one may derive that
\begin{equation*}
 \int_{Y^{*}}|\nabla\tilde{v}|^{2}|\varphi|^{2}dx
 \lesssim\int_{Y^{*}}|\tilde{v}\nabla\varphi|^{2}dx.
\end{equation*}
Combining the above inequality with \eqref{pri:4.17},
we have proved that
\begin{equation*}
\int_{Y^{*}}|\nabla(\tilde{v}\varphi)|^{2}dx
\lesssim\int_{Y^{*}}|\tilde{v}|^{2}|\nabla\varphi|^{2}dx.
\end{equation*}
Therefore, on account of \eqref{pri:4.16} it is not hard to see that
\begin{equation*}
\begin{aligned}
\int_{(Y+z)\backslash\omega}|\tilde{v}|^{2}dx
\lesssim\int_{Y^{*}}|\tilde{v}|^{2}|\nabla\varphi|^{2}dx
\lesssim \int_{Y^{*}\cap\omega}|\tilde{v}|^{2}dx,
\end{aligned}
\end{equation*}
where the last inequality follows from the
fact that $\nabla \varphi=0$ on $Y^{*}\backslash\omega$.
Hence, we have the desired claim \eqref{claim1}.

In the following, we proceed to show \eqref{g(r,v)}.
Recalling that $\mathcal{L}_{0}(v)=0$ in
$\varepsilon Y^{*}$ and $\mathcal{L}_0 = -\nabla_x\cdot \widehat{A}(\nabla_x)$,
%we note that
%$-\frac{\partial}{\partial x_{j}}\widehat{A}^{j}(\nabla v(x))=F(x)$
%for $x\in\varepsilon Y^{*}$.
set $x=\varepsilon y$ with $y\in Y^{*}$ and
$\overline{v}(y)=\frac{1}{\varepsilon}v(\varepsilon y)
=\frac{1}{\varepsilon}v(x)$,
%$\overline{F}(y)=\varepsilon F(\varepsilon y)
%=\varepsilon F(x)$,
and then $\overline{v}$ satisfies
the equation
\begin{equation}\label{pri:4.19}
\mathcal{L}_0(\overline{v})
=0 \quad \text{in} ~ Y^{*}.
\end{equation}
Due to the claim \eqref{claim1}, one may obtain
\begin{equation*}
  \int_{Y+z}|\overline{v}(y)-My-c|^{2}dy
  \lesssim\int_{Y^{*}\cap\omega}\big|\overline{v}(y)-My-c\big|^{2}dy
\end{equation*}
for any $M\in \mathbb{R}^d$ and $c\in\mathbb{R}$, which is equivalent to
\begin{equation*}
\int_{\varepsilon(Y+z)}
\Big|\frac{1}{\varepsilon} v(x)-\frac{M}{\varepsilon}x-c\Big|^{2}
dx\lesssim
\int_{\varepsilon(Y^{*}\cap\omega)}
\Big|\frac{1}{\varepsilon} v(x)-\frac{M}{\varepsilon}x-c\Big|^{2}dx.
\end{equation*}
This further gives that
\begin{equation*}
\int_{\varepsilon(Y+z)}| v(x)-Mx-c\varepsilon
|^{2}dx\lesssim\int_{\varepsilon(Y^{*}\cap\omega)}
| v(x)-Mx-c\varepsilon|^{2}dx.
\end{equation*}
Because of the arbitrariness of $c$, we may derive that
\begin{equation*}
  \|v-Mx-c\|_{L^{2}(\varepsilon (Y+z))}
  \lesssim \|v-Mx-c\|_{L^{2}(\varepsilon (Y^{*}\cap\omega))}.
\end{equation*}
According to the fact that there is a constant $N'<\infty$
depending only on $d$ such that $Y^{*}(z_{1})\cap Y^{*}(z_{2})
\neq\emptyset$ for at most $N'$ coordinates if $z_{1}\neq z_{2}$,
and then summing over all $z\in T_{\varepsilon}$ gives
\begin{equation}\label{pri:4.20}
\|v-Mx-c\|_{L^{2}(B_r)}
\leq C \|v-Mx-c\|_{L^{2}(B^{\varepsilon}_{2r})},
\end{equation}
in which we use the fact $r\geq\varepsilon$
in the above inequality.
By recalling the definition of $G(r,v)$ and $G_{\varepsilon}(2r,v)$,
we have completed the whole proof.
\end{proof}

For the ease of the statement, we impose the notation
\begin{equation*}
\begin{aligned}
\Phi(r) := \frac{1}{r}\inf_{c\in\mathbb{R}}
\Big(\dashint_{B^{\varepsilon}(0,r)}|u_\varepsilon - c|^2 \Big)^{1/2}.
\end{aligned}
\end{equation*}

\begin{lemma}[iteration's inequality I]\label{lemma:4.4}
Let $\sigma=1/2-1/p$ be given as in
Lemma $\ref{lemma:4.1}$.
Assume the same conditions as in Theorem $\ref{thm:1.2}$.
Let $u_\varepsilon$ be the solution of
$\mathcal{L}_\varepsilon(u_\varepsilon) = 0$ in $B^{\varepsilon}(0,2r)$ with $\sigma_{\varepsilon}(u_{\varepsilon})=0$ on $\partial B^{\varepsilon}_{2r}|_{B_{2r}}$.
Then there exists $\theta\in(0,1/4)$ such that
\begin{equation}
 G_{\varepsilon}(\theta r, u_\varepsilon) \leq \frac{1}{2}G_{\varepsilon}(r,u_\varepsilon)
 + C\left(\frac{\varepsilon}{r}
 \right)^{\sigma}\Phi(2r)
\end{equation}
for any $\varepsilon\leq r<1/4$.
\end{lemma}

\begin{proof}
Fixed $r\in[\varepsilon,1/4)$, let $w$ be a solution to
$\mathcal{L}_0 w = 0$ in $B(0,r)$ as in Lemma $\ref{lemma:4.1}$.
For any $\theta\in(0,\frac{1}{4})$ (which will be fixed later), we have
\begin{equation*}
\begin{aligned}
G_{\varepsilon}(\theta r,u_\varepsilon)
& =\frac{1}{\theta r}\inf_{M\in \mathbb{R}^{d}\atop c\in \mathbb{R}}
\bigg(\dashint_{B^{\varepsilon}_{\theta r}} |u_{\varepsilon}-Mx-c|^2\bigg)^{\frac{1}{2}}\\
& \leq \frac{1}{\theta r}\inf_{M\in \mathbb{R}^{d}\atop c\in \mathbb{R}}
\bigg(\dashint_{B^{\varepsilon}_{\theta r}} |w-Mx-c|^2\bigg)^{\frac{1}{2}}+\frac{1}{\theta r}
\bigg(\dashint_{B^{\varepsilon}_{\theta r}} |u_{\varepsilon}-w|^2\bigg)^{\frac{1}{2}}\\
& \leq \frac{1}{\theta r}[\nabla w]_{C^{0,\alpha}(B^{\varepsilon}_{\theta r})}(\theta r)^{1+\alpha}+\frac{1}{\theta r}
\bigg(\dashint_{B^{\varepsilon}_{\theta r}} |u_{\varepsilon}-w|^2\bigg)^{\frac{1}{2}},
\end{aligned}
\end{equation*}
where we take $M=\nabla w(0),c=w(0)$ and employ
the mean value theorem for $w(x)$ in the last step.
It's easy to see that the right-hand side above is less than
\begin{equation*}
  \theta^{\alpha}r^{\alpha}[\nabla w]_{C^{0,\alpha}
  (B_{\theta r})}
  +r^{-1}\theta^{-1-\frac{d}{2}}
  \bigg(\dashint_{B^{\varepsilon}_{r}} |u_{\varepsilon}-w|^2\bigg)^{\frac{1}{2}}.
\end{equation*}
Let $\tilde{w}(x)=w(x)-Mx-c$, and
we take $\tilde{a}(x)=(\tilde{a}_{ij}(x))
=\nabla_{\xi_{j}}\widehat{A}^{i}(\nabla w)$
as we did in Theorem \ref{lemma:4.2}, for $1\leq k\leq d$,
$\tilde{w}(x)$
satisfies
\begin{equation*}
  -\nabla\cdot \tilde{a}(x)\nabla(\nabla_{k}\tilde{w})=0~
  \text{~in~} B_{r}.
\end{equation*}
By \cite[Theorem 8.13]{MGLM} we have
\begin{equation*}
  [\nabla \tilde{w}]_{C^{0,\alpha}
  (B_{r/4})}\leq
  Cr^{-\alpha-1}\Big(\dashint_{B_{r/2}}|\tilde{w}|^2\Big)^{1/2}.
\end{equation*}
According to the fact that $\nabla \tilde{w}=\nabla w-M$, one may have
\begin{equation*}
%\begin{aligned}
[\nabla w]_{C^{0,\alpha}(B_{r/4})}
  =[\nabla \tilde{w}]_{C^{0,\alpha}(B_{r/4})}
  \leq Cr^{-\alpha-1}\Big(\dashint_{B_{r/2}}|w-Mx-c|^2\Big)^{1/2}.
%\end{aligned}
\end{equation*}
Then we have
\begin{equation*}
\begin{aligned}
G_{\varepsilon}(\theta r,u_{\varepsilon})
& \lesssim \theta^{\alpha}r^{-1}\bigg(\dashint_{B_{r/2}} |w-Mx-c|^2\bigg)^{\frac{1}{2}} +r^{-1}\theta^{-1-\frac{d}{2}}
\bigg(\dashint_{B^{\varepsilon}_r} |u_{\varepsilon}-w|^2\bigg)^{\frac{1}{2}}\\
& \lesssim \theta^{\alpha}G(\frac{r}{2},w)+r^{-1}\theta^{-1-\frac{d}{2}}
\bigg(\dashint_{B^{\varepsilon}_r}
|u_{\varepsilon}-w|^2\bigg)^{\frac{1}{2}},
\end{aligned}
\end{equation*}
And then, by Lemma \ref{lemma:2.6}, the right hand side above is less than
\begin{equation*}
   C \theta^{\alpha}G_{\varepsilon}(r,w)
   +r^{-1}\theta^{-1-\frac{d}{2}}
   \bigg(\dashint_{B^{\varepsilon}_r} |u_{\varepsilon}-w|^2\bigg)^{\frac{1}{2}}.
\end{equation*}
By the definition of $G_{\varepsilon}(r,w)$, it follows that
\begin{equation*}
\begin{aligned}
  G_{\varepsilon}(\theta r,u_{\varepsilon})
  &\leq C \theta^{\alpha}G_{\varepsilon}(r,u_{\varepsilon})
  +Cr^{-1}\theta^{-1-\frac{d}{2}}\bigg(
  \dashint_{B^{\varepsilon}_r} |u_{\varepsilon}-w|^2\bigg)^{\frac{1}{2}}\\
  & \leq \frac{1}{2} G_{\varepsilon}(r,u_{\varepsilon})+Cr^{-1}
  \bigg(\dashint_{B^{\varepsilon}_r} |u_{\varepsilon}-w|^2\bigg)^{\frac{1}{2}},
  \end{aligned}
\end{equation*}
where we choose $\theta$ small enough such that $C\theta^{\alpha}=\frac{1}{2}$.
By Lemma \ref{lemma:4.1}, we arrive at
\begin{equation*}
G_{\varepsilon}(\theta r,u_\varepsilon)\leq \frac{1}{2}G_{\varepsilon}(r,u_\varepsilon)
+C\Big(\frac{\varepsilon}{r}\Big)^{\sigma}\frac{1}{r}
\bigg(\dashint_{B^{\varepsilon}_{2r}} |u_{\varepsilon}|^2\bigg)^{\frac{1}{2}}.
\end{equation*}
Note that for any $c\in\mathbb{R}$,
$u_\varepsilon - c$ is still a solution of
$\mathcal{L}_\varepsilon u_\varepsilon = 0$ in $B^{\varepsilon}_{2r}$,
and the proof is complete.
\end{proof}

\begin{lemma}[iteration lemma]\label{lemma:4.3}
Let $\Psi(r)$ and $\psi(r)$ be two nonnegative continuous functions on the integral $(0,1]$.
Let $0<\varepsilon\ll1$. Suppose that there exists a constant $C_0$ such that
\begin{equation}\label{pri:4.4}
\left\{\begin{aligned}
  &\max_{r\leq t\leq 2r} \Psi(t) \leq C_0 \Psi(2r),\\
  &\max_{r\leq s,t\leq 2r} |\psi(t)-\psi(s)|\leq C_0 \Psi(2r).
  \end{aligned}\right.
\end{equation}
We further assume that
\begin{equation}\label{pri:4.5}
\Psi(\theta r)\leq \frac{1}{2}\Psi(r) + C_0w(\varepsilon/r)
\Big\{\Psi(2r)+\psi(2r)\Big\}
\end{equation}
holds for any $\varepsilon\leq r <(1/4)$, where $\theta\in(0,1/4)$ and $w$ is a nonnegative
increasing function in $[0,1]$ such that $w(0)=0$ and
\begin{equation*}
 \int_0^1 \frac{w(t)}{t} dt <\infty.
\end{equation*}
Then, we have
\begin{equation}
\max_{\varepsilon\leq r\leq 1}\Big\{\Psi(r)+ \psi(r)\Big\}
\leq C\Big\{\Psi(1)+\psi(1)\Big\},
\end{equation}
where $C$ depends only on $C_0, \theta$ and $\omega$.
\end{lemma}

\begin{proof}
The proof may be found in \cite[Lemma 8.5]{S5}.
\end{proof}

\noindent\textbf{Proof of Theorem $\ref{thm:1.2}$}.
It is fine to assume $0<\varepsilon<1/4$, otherwise it follows from the classical theory.
In view of Lemma $\ref{lemma:4.3}$,
we set $\Psi(r) = G_{\varepsilon}(r,u_\varepsilon)$, $w(t) =t^{\frac{1}{2}-\frac{1}{p}}$ with
$0<p-2\ll 1$.
To prove the desired estimate
$\eqref{pri:1.1}$, it is sufficient to verify $\eqref{pri:4.4}$
and $\eqref{pri:4.5}$.
Let $\psi(r) = |M_r|$, where $M_r$ is the matrix associated with $\Psi(r)$ such that
\begin{equation*}
\begin{aligned}
\Psi(r) &= \frac{1}{r}
\inf_{c\in\mathbb{R}}\Big(\dashint_{B^{\varepsilon}_r}|u_\varepsilon-M_r x - c|^2\Big)^{\frac{1}{2}}.
\end{aligned}
\end{equation*}
Then it follows that
\begin{equation}\label{f:4.7}
\left\{\begin{aligned}
& \max_{r\leq t\leq 2r} \Psi(t) \lesssim \Psi(2r),\\
& \Phi(2r) \lesssim \Big\{\Psi(2r) + \psi(2r)\Big\},\\
& \psi(r)\leq \Psi(r)+\frac{1}{r}\inf_{c\in
\mathbb{R}}\bigg(\dashint_{B^{\varepsilon}_r}
|u_{\varepsilon}-c|^2dx\bigg)^{\frac{1}{2}}.
\end{aligned}\right.
\end{equation}
According to Lemma $\ref{lemma:4.4}$, we have
\begin{equation*}
\Psi(\theta r)\leq \frac{1}{2}\Psi(r) + C_0 w(\varepsilon/r)\Big\{\Psi(2r)+\psi(2r)\Big\}
\end{equation*}
for $\varepsilon\leq r<1/4$,
so condition $\eqref{pri:4.5}$ in Lemma $\ref{lemma:4.3}$ holds.
Let $t,s\in [r,2r]$, and $v(x)=(M_t-M_s)x$. It is clear to see $v$ is harmonic in $\mathbb{R}^d$ and $\mathcal{L}_{0}(v)=0$ in $\mathbb{R}^d$, it's easy to see that
\begin{equation}\label{f:4.9}
\begin{aligned}
|M_t-M_s|&\lesssim  \frac{1}{r}\Big(\dashint_{B_{r/2}}|(M_t-M_s)x-c|^2\Big)^{\frac{1}{2}}
\lesssim^{\eqref{pri:4.20}}  \frac{1}{r}\Big(\dashint_{B^{\varepsilon}_r}|(M_t-M_s)x-c|^2\Big)^{\frac{1}{2}}\\
&\lesssim \frac{1}{t}
\Big(\dashint_{B^{\varepsilon}_t}|u_\varepsilon - M_tx-c|^2\Big)^{\frac{1}{2}}
+ \frac{1}{s}\Big(\dashint_{B^{\varepsilon}_s}|u_\varepsilon - M_sx-c|^2\Big)^{\frac{1}{2}}
\end{aligned}
\end{equation}
where the last step is based on the fact that $s,t\in[r,2r]$. By taking infimum about $c\in\mathbb{R}$ on the both sides of \eqref{f:4.9}, it follows that
\begin{equation}\label{f:4.8}
|M_t-M_s|\lesssim \Big\{\Psi(t)+\Psi(s)\Big\}
\lesssim \Psi(2r).
\end{equation}
Due to estimate $\eqref{f:4.8}$ , it's known that $\psi(r)$ satisfies the second condition in
$\eqref{pri:4.4}$.
%Before we get the estimate \eqref{pri:1.1}, we claim that for $r\in(\varepsilon,1]$,
%\begin{equation}\label{claim2}
% \|u_{\varepsilon}-c\|_{L^{2}(B_{\varepsilon}(0,r))}\leq \|u_{\varepsilon}-c\|_{L^{2}(\Sigma)}\leq C\|\nabla u_{\varepsilon}\|_{L^{2}(B_{\varepsilon}(0,\frac{3}{2}r))},
%\end{equation}
%in which $c=\dashint_{\Sigma}u_{\varepsilon}dx$,
%and $\Sigma$ is a ``good'' domain such that
%$B_{\varepsilon}(0,r)\subset\Sigma
%\subset B_{\varepsilon}(0,\frac{3}{2}r)$
%and Poincar\'{e}'s inequality holds on $\Sigma$.
%We can get this $\Sigma$ by avoiding the cusps on
%the boundary of $B_{\varepsilon}(0,\frac{5}{4}r)$,
%and then the estimate \eqref{claim2} follows from
%Poincar\'{e}'s inequality.

Hence, according to Lemma $\ref{lemma:4.3}$,
for any $r\in(\varepsilon,1]$,
we have the following estimate
\begin{equation}
\begin{aligned}
\frac{1}{r}\inf_{c\in\mathbb{R}}
\Big(\dashint_{B^{\varepsilon}_r}|u_\varepsilon - c|^2 \Big)^{\frac{1}{2}}
&\leq \Big\{\Psi(r) + \psi(r)\Big\}
\lesssim\Big\{\Psi(1) + \psi(1)\Big\}\\
&\lesssim \Big\{G_{\varepsilon}(1,u_{\varepsilon})+\psi(1)\Big\}
\lesssim^{\eqref{f:4.7}} G_{\varepsilon}(1,u_{\varepsilon})+
\inf_{c\in\mathbb{R}}\bigg(\dashint_{B^{\varepsilon}_1} |u_{\varepsilon}-c|^2dx\bigg)^{\frac{1}{2}}.
\end{aligned}
\end{equation}
If we take $M=0$ and the constant
$c$ as in Lemma $\ref{lemma:2.20}$, then
it follows that
\begin{equation*}
\begin{aligned}
\frac{1}{r}\inf_{c\in\mathbb{R}}
\Big(\dashint_{B^{\varepsilon}_r}|u_\varepsilon - c|^2 \Big)^{\frac{1}{2}}
  &\lesssim \Big(\dashint_{B^{\varepsilon}_1}| u_\varepsilon-c|^2\Big)^{1/2}\\
&\lesssim^{\eqref{pri:2.20}} \Big(\dashint_{B^{\varepsilon}_3}|\nabla u_\varepsilon|^2\Big)^{1/2}.
\end{aligned}
\end{equation*}
Therefore, the desired estimate $\eqref{pri:1.1}$
is consequently obtained by
the above estimate coupled with
Caccioppoli's inequality $\eqref{pri:2.14}$,
and we have completed the whole proof.
\qed

\section{Boundary estimates}\label{sec:5}
\begin{lemma}[approximating lemma II]\label{lemma:5.2}
Let $\varepsilon\leq r\leq1$.
Let $\Omega$ be a bounded Lipschitz domain and $\omega$ be a regular reference domain. Suppose that $\mathcal{L}_\varepsilon$
satisfies $\eqref{a:1}$-$\eqref{a:3}$.
Let $u_\varepsilon$ be a weak solution of
\begin{equation*}
 \left\{ \begin{aligned}
 \mathcal{L}_{\varepsilon}u_\varepsilon  &= 0& ~\emph{in}&~ D^{\varepsilon}_{4r},\\
 \sigma_{\varepsilon}(u_{\varepsilon})&= 0& ~\emph{on}&~
 \partial D^{\varepsilon}_{4r}|_{D_{4r}},\\
 u_{\varepsilon}&= 0&~\emph{on}&~\partial D^{\varepsilon}_{4r}|_{\Delta_{4r}}.
  \end{aligned}\right.
\end{equation*}
Then there exists
$v\in H^1(D_{r})$ such that
$\mathcal{L}_0 v  = 0$ in $D_{r}$ with $v=0$ on $\Delta_{r}$,
and
\begin{equation}\label{pri:5.2}
 \Big(\dashint_{D^{\varepsilon}_{r/12}} |u_\varepsilon - v|^2  \Big)^{1/2}
 \lesssim \left(\frac{\varepsilon}{r}\right)^{\frac{\sigma}{2}}
 \Big(\dashint_{D^{\varepsilon}_{3r}}|u_\varepsilon|^2 \Big)^{1/2},
\end{equation}
where $\sigma=1/2-1/p$ and $0<p-2\ll1$ is the same $p$ as in \eqref{pri:1.6} and Theorem \ref{thm:2.20}.
\end{lemma}
\begin{proof}
Although the main idea is similar to that given for
Lemma $\ref{lemma:4.1}$, it is proved to be still complicated
task to close the whole arguments due to the worse boundary conditions. As a normal way, people usually
build $v$ by solving a Dirichlet problem with the boundary
condition $v=u_\varepsilon$ on $\partial D_r$. There are
two notable problematic issues: (1). $u_\varepsilon$ has no definition on $\partial D_r\setminus(\varepsilon\omega)$;
(2). The constructed approximating function $v$ is required to vanish on $\Delta_r$. One may employ the extension of $u_\varepsilon$, denoted by $\tilde{u}_\varepsilon$, to make
the boundary equality well-defined. However, $\tilde{u}_\varepsilon \not=0$ on $\Delta_r\setminus (\varepsilon\omega)$, which means that
it breaks the requirement in (2).

Let $0<\delta\ll 1$ (it will be chosen later on), and $\bar{t}\in[1/4,1/2]$ be arbitrary but fixed. By rescaling one may assume $r=1$. The proof consists of three parts: \textbf{(A)}. Outline
the main ideas; \textbf{(B)}. Present some auxiliary estimates; \textbf{(C)}. Carry out computations and complete the proof.

\textbf{Part (A)}.
To overcome the stated difficulty, we divided the approximating process into
three ingredients.

(1). Find a regularization part of
$u_\varepsilon$ (denoted by $v_\varepsilon$), and define a function to measure their difference (denoted by $z_\varepsilon$), as follows:
\begin{equation*}
(\text{i})\left\{
\begin{aligned}
\mathcal{L}_{\varepsilon}(v_{\varepsilon})&=0&\quad&
\text{in}\quad D^{\varepsilon}_{\bar{t}},\\
\sigma_{\varepsilon}(v_{\varepsilon})&=0&\quad&\text{on}\quad \partial D^{\varepsilon}_{\bar{t}}|_{D_{\bar{t}}},\\
v_{\varepsilon}&=S_{\delta}(\tilde{u}_{\varepsilon})&\quad&\text{on}\quad \partial D_{\bar{t}}^\varepsilon|_{\partial D_{\bar{t}}};
\end{aligned}
 \right.
 \qquad
(\text{ii})
\left\{\begin{aligned}
z_\varepsilon-\Delta z_{\varepsilon}&=0&\quad&\text{in}\quad D_{\bar{t}},\\
z_{\varepsilon}&=\tilde{u}_{\varepsilon}-S_{\delta}(\tilde{u}_{\varepsilon})
&\quad&\text{on}\quad\partial D_{\bar{t}},
\end{aligned}\right.
\end{equation*}
in which $S_{\delta}(\tilde{u}_{\varepsilon})$ is defined in \eqref{pri:2.161}, and the extension of $\tilde{u}_\varepsilon$ is explained in Part (B).

(2). Approximate the regularization part by homogenization. Thus, we construct
$v_h$ satisfying
\begin{equation*}
(\text{iii})
\left\{
\begin{aligned}
\mathcal{L}_{0}(v_{h})&=0&\qquad&\text{in}\quad D_{\bar{t}},\\
v_{h}&=S_{\delta}(\tilde{u}_{\varepsilon})&\qquad&\text{on}\quad \partial D_{\bar{t}}.
 \end{aligned}\right.
\end{equation*}

(3). Define the desired approximating function in the following way, and estimate their difference to make the
whole arguments close,
\begin{equation*}
(\text{iv})\left\{\begin{aligned}
\mathcal{L}_{0}(v)&=0&\qquad&\text{in}\quad D_{\bar{t}},\\
v&=S_{\delta}(\tilde{u}_{\varepsilon})&\qquad&\text{on}\quad \partial D_{t}\backslash\Delta_{\bar{t}},\\
v&=0&\qquad&\text{on}\quad\Delta_{\bar{t}}.
 \end{aligned}\right.
\end{equation*}

\textbf{Part (B).}
First, as in Lemma \ref{lemma:4.1}, we want to extend $u_{\varepsilon}\in H^{1}(D^{\varepsilon}_{1/4})$ to $\tilde{u}_{\varepsilon}\in H^{1}_0(D^{0})$, in which $D_{1}\subset D^0$. Suppose that
 $$\varphi\in C^{\infty}_0(\mathbb{R}^d),~\varphi=1~\text{on}~D_{1/4},~
 \varphi=0~\text{on}~\Omega\backslash D_{1/2},~\text{and}~|\nabla \varphi|\lesssim 1.$$
Then we have $\varphi u_{\varepsilon}\in H^{1}\big(D^{\varepsilon}_{\frac{1}{2}},\partial D_{\frac{1}{2}}^\varepsilon|_{\partial D_{\frac{1}{2}}}\big)$ (the same to
the definition of $H^1(\Omega_\varepsilon,\Gamma_\varepsilon)$). By Lemma \ref{extensiontheory}, one may denote the extension function
of $\varphi u_\varepsilon$ by $\tilde{u}_{\varepsilon}$,  satisfying that $\tilde{u}_{\varepsilon}=\varphi u_{\varepsilon}=u_{\varepsilon}$
on $D_{1/4}^{\varepsilon}$ with $\tilde{u}_{\varepsilon}\in H^{1}_{0}(D^{0})$, and
\begin{equation}\label{pri:5.3}
 \|\tilde{u}_{\varepsilon}\|_{H^{1}_{0}(D^{0})}\lesssim \|\varphi u_{\varepsilon}\|_{H^{1}(D^{\varepsilon}_{1/2})}
 \lesssim^{\eqref{pri:5.1}}
 \|u_{\varepsilon}\|_{L^{2}(D^{\varepsilon}_{3})}.
\end{equation}

Now, we claim that for any $t\in (0,3/2]$,  one may have
\begin{equation}\label{pri:5.4}
\|\tilde{u}_{\varepsilon}\|_{L^{2}(\Delta_{t}\backslash \varepsilon\omega)}\lesssim\varepsilon^{1/2}
\|u_{\varepsilon}\|_{L^{2}(D_{3}^{\varepsilon})}.
\end{equation}
Note that $\Delta_{t}\backslash \varepsilon\omega$ represents the holes intersected with the boundary and the diameters of the holes are around the $\varepsilon$-scale. By the trace theorem near the boundary, it follows that
\begin{equation*}
\begin{aligned}
\int_{\Delta_{t}\backslash\varepsilon\omega}|\tilde{u}_{\varepsilon}|^2dS
&\lesssim\frac{1}{\varepsilon}\int_{O_{\varepsilon}\cap D_{t}}|\tilde{u}_{\varepsilon}|^2dx+\varepsilon\int_{O_{\varepsilon}\cap D_{t}}|\nabla\tilde{u}_{\varepsilon}|^2dx\\
&\lesssim\varepsilon\int_{O_{\varepsilon}\cap D_{t}}|\nabla\tilde{u}_{\varepsilon}|^2dx
\lesssim\varepsilon\int_{D_{3/2}}|\nabla\tilde{u}_{\varepsilon}|^2dx
\lesssim^{\eqref{pri:5.3}}\varepsilon\int_{D_{3}^{\varepsilon}}|u_{\varepsilon}|^2dx,
\end{aligned}
\end{equation*}
in which we employ Poincar\'{e}'s inequality in the second step, and we derive \eqref{pri:5.4}.
%and
%\begin{equation}\label{pri:5.5}
%\|\nabla_{\text{tan}}\tilde{u}_{\varepsilon}\|_{L^{2}(\Delta_{t}\backslash \Gamma_{\varepsilon})}\lesssim\varepsilon^{-1/p}
%\|u_{\varepsilon}\|_{L^{2}(D_{3}^{\varepsilon})}
%\end{equation}
%for some $2<p\ll3$.

Also, for any $t\in (0,3/2]$, there holds
\begin{equation}\label{pri:5.8}
\|u_{\varepsilon}\|_{W^{1,p}(D_{t/3}^{\varepsilon})}
\lesssim^{\eqref{pri:2.21}} \|\nabla u_{\varepsilon}\|_{L^{p}(D_{t}^{\varepsilon})}
\lesssim^{\eqref{pri:7.7},\eqref{pri:5.1}}
\|u_{\varepsilon}\|_{L^{2}(D_{3}^{\varepsilon})},
\end{equation}
where $0<p-2\ll 1$.
As a result, we have
\begin{equation}\label{pri:5.5}
\|\tilde{u}_{\varepsilon}\|_{W^{1,p}(D^0)}\lesssim^{\eqref{pri:2.22}}
\|\varphi u_{\varepsilon}\|_{W^{1,p}(D_{1/2}^{\varepsilon})}
\lesssim^{\eqref{pri:5.8}} \|u_{\varepsilon}\|_{L^2(D_{3}^{\varepsilon})}.
\end{equation}

\textbf{Part (C).}
According to the description in Part (A), we will study the following estimates in this part. Recalling that
$\bar{t}\in[1/4,1/2]$,
\begin{equation}\label{f:5.4}
\begin{aligned}
\|u_{\varepsilon}-v\|_{L^{2}(D_{1/12}^{\varepsilon})}&
\leq\|u_{\varepsilon}-v_{\varepsilon}-z_{\varepsilon}
\|_{L^{2}(D_{\bar{t}/3}^{\varepsilon})}
+\|v_{\varepsilon}-v_{h}\|_{L^{2}(D_{\bar{t}}^{\varepsilon})}
+\|v_{h}-v\|_{L^{2}(D_{\bar{t}}^{\varepsilon})}
+\|z_{\varepsilon}\|_{L^{2}(D_{\bar{t}}^{\varepsilon})}\\
&:= I_1 + I_2 + I_3 + I_4.
\end{aligned}
\end{equation}

Let $\sigma=1/2-1/p$.
For $v_{\varepsilon}$ and $v_h$, it follows from estimates \eqref{pri:1.6} and trace theorem that
\begin{equation}\label{pri:5.6}
\begin{aligned}
I_2 = \|v_{\varepsilon}-v_{h}\|_{L^{2}(D_{\bar{t}}^{\varepsilon})}
&\lesssim\varepsilon^{\sigma}
\|S_{\delta}(\tilde{u}_{\varepsilon})\|_{W^{1-1/p,p}(\partial D_{\bar{t}})}
\lesssim \varepsilon^{\sigma}
\|S_{\delta}(\tilde{u}_{\varepsilon})\|_{W^{1,p}(D_{\bar{t}})}\\
&\lesssim^{\eqref{pri:2.26}} \varepsilon^{\sigma}
\|\tilde{u}_{\varepsilon}\|_{W^{1,p}((D_{\bar{t}})_{\delta})}\\
&\lesssim^{\eqref{pri:5.5}}
\varepsilon^{\sigma}
\|u_{\varepsilon}\|_{L^2(D_{3}^{\varepsilon})}.
\end{aligned}
\end{equation}

%According to the estimate \eqref{pri:5.5}, one may have
%\begin{equation}\label{pri:5.71}
%\|v_{\varepsilon}-v_{h}\|_{L^{2}(D_{t}^{\varepsilon})}
%\lesssim
%\varepsilon^{\sigma}
%\|u_{\varepsilon}\|_{L^2(D_{3}^{\varepsilon})}.
%\end{equation}

Next, we claim that
\begin{equation}\label{pri:5.121}
I_1 = \|u_{\varepsilon}-v_{\varepsilon}-z_{\varepsilon}\|_
{L^{2}(D_{\bar{t}/3}^{\varepsilon})}
\lesssim^{\eqref{pri:2.21}}
\|\nabla u_{\varepsilon}-\nabla v_{\varepsilon}\|_{L^{2}(D^{\varepsilon}_{\bar{t}})}
+\|\nabla z_{\varepsilon}\|_{L^{2}(D_{\bar{t}})}
\lesssim^{\eqref{f:5.3}}\|\nabla z_{\varepsilon}\|_{L^{2}(D_{\bar{t}})},
\end{equation}
and the last inequality is due to the following energy estimates.
If setting  $\phi:=u_{\varepsilon}-v_{\varepsilon}-z_{\varepsilon}$, then we observe that $\phi=0$ on $\partial D_{\bar{t}}^{\varepsilon}|_{\partial D_{\bar{t}}}$ by the definition of $u_{\varepsilon},v_{\varepsilon}$ and $z_{\varepsilon}$, and
so $\phi$ belongs to $H^{1}(D_{\bar{t}}^{\varepsilon},\partial D_{\bar{t}}^{\varepsilon}|_{\partial D_{\bar{t}}})$. Integration by parts, there holds
\begin{equation*}
  \int_{D_{\bar{t}}^{\varepsilon}}\big[A(x/\varepsilon,\nabla u_{\varepsilon})-
  A(x/\varepsilon,\nabla v_{\varepsilon})\big]\cdot\nabla \phi dx=0.
\end{equation*}
Thus, it follows from the assumption $\eqref{a:1}$ that
\begin{equation*}
 \mu_0\|\nabla u_\varepsilon - \nabla v_\varepsilon\|_{L^2(D_{\bar{t}}^\varepsilon)}^2
 \leq \mu_1\int_{D_{\bar{t}}^\varepsilon}|\nabla u_\varepsilon - \nabla v_\varepsilon||\nabla z_\varepsilon|dx
 \leq \frac{\mu_0}{2}
 \|\nabla u_\varepsilon - \nabla v_\varepsilon\|_{L^2(D_{\bar{t}}^\varepsilon)}^2
 + C_{\mu_0}\|\nabla z_\varepsilon\|_{L^2(D_{\bar{t}}^\varepsilon)}^2,
\end{equation*}
where we use Young's inequality in the second step, and this implies
\begin{equation}\label{f:5.3}
\|\nabla u_\varepsilon - \nabla v_\varepsilon\|_{L^2(D_{\bar{t}}^\varepsilon)}
\lesssim \|\nabla z_\varepsilon\|_{L^2(D_{\bar{t}}^\varepsilon)}.
\end{equation}

Again, in view of energy estimates for $z_{\varepsilon}$, one may have
\begin{equation}\label{pri:5.131}
\begin{aligned}
I_1+I_4 \lesssim^{\eqref{pri:5.121}} \|z_{\varepsilon}\|_{H^{1}(D_{\bar{t}})}\lesssim
 \|\tilde{u}_{\varepsilon}-S_{\delta}(\tilde{u}_{\varepsilon})\|
 _{H^{1/2}(\partial D_{\bar{t}})}
 &\lesssim
 \bigg(\|\tilde{u}_{\varepsilon}-S_{\delta}(\tilde{u}_{\varepsilon})
 \|_{H^{1}(D_{\bar{t}})}
 \|\tilde{u}_{\varepsilon}-S_{\delta}(\tilde{u}_{\varepsilon})\|
 _{L^{2}(D_{\bar{t}})} \bigg)^{1/2}\\
 &\lesssim^{\eqref{pri:2.7},\eqref{pri:2.26}} \bigg(\delta\|\tilde{u}_{\varepsilon} \|_{H^{1}((D_{\bar{t}})_{\delta})}
 \|\nabla\tilde{u}_{\varepsilon}\| _{L^{2}((D_{\bar{t}})_{\delta})} \bigg)^{1/2}\\
 &\lesssim^{\eqref{pri:2.23},\eqref{pri:5.1}}
 \delta^{1/2}\|u_{\varepsilon}
 \|_{L^{2}(D_{3}^{\varepsilon})},
 \end{aligned}
\end{equation}
where the second step comes
from the trace theorem in Besov spaces.

%Recalling that $z_{\varepsilon}$ satisfies the auxiliary equation (iv), from H\"{o}lder's inequality, it follows that
%\begin{equation}\label{pri:5.141}
%\begin{aligned}
%\|z_{\varepsilon}\|_{L^2(D_t)}
%&\lesssim\|z_{\varepsilon}\|_{L^{\frac{2d}{d-1}}(D_t)}
%\lesssim\|(z_{\varepsilon})^*\|_{L^{2}(\partial D_t)}\\
%&\lesssim\|z_{\varepsilon}\|_{L^{2}(\partial D_t)}
%\lesssim\|z_{\varepsilon}\|_{H^{1/2}(\partial D_t)}
%\lesssim^{\eqref{pri:5.131}}
%\delta^{1/2}\|u_{\varepsilon}\|_{L^{2}(D_{3}^{\varepsilon})},
%\end{aligned}
%\end{equation}
%in which the notation $(z_\varepsilon)^*$ represents the nontangential maximal
%function of $z_\varepsilon$ (see for example \cite[Definition 2.19]{X3}).
%Here the second inequality follows from \cite[Remark 9.3]{KFS1},
%the third one is the so-called nontangential maximal function estimate
%(see for example \cite[Theorem 7.5.14]{S4}).

The reminder of the proof is to calculate $I_3$.
We note that $v_{h}-v=0$ on $\partial D_{t}\backslash \Delta_{t}$ and $v_{h}-v=S_{\delta}(\tilde{u}_{\varepsilon})$ on $\Delta_{t}$, coupled with \eqref{pri:2.3}, one may have
\begin{equation}\label{pri:5.7}
  \begin{aligned}
 I^2_3\leq \|v_{h}-v\|^2_{L^{2}(D_{\bar{t}})}
  &\lesssim\|\nabla v_{h}-\nabla v\|^2_{L^{2}(D_{\bar{t}})}\lesssim
  \int_{D_{\bar{t}}}[\hat{A}(\nabla v_{h})-\hat{A}(\nabla v)][\nabla v_{h}-\nabla v]\\
  &=-\int_{D_{\bar{t}}}\nabla\cdot[\hat{A}(\nabla v_{h})-\hat{A}(\nabla v)](v_{h}-v)
  +\int_{\partial D_{\bar{t}}}\vec{n}\cdot[\hat{A}(\nabla v_{h})-\hat{A}(\nabla v)](v_{h}-v)dS\\
  &\lesssim \int_{\Delta_{\bar{t}}}
  |\nabla S_{\delta}(\tilde{u}_{\varepsilon})|
  |S_{\delta}(\tilde{u}_{\varepsilon})|dS\lesssim \|\nabla S_{\delta}(\tilde{u}_{\varepsilon})\|_{L^{2}(\Delta_{\bar{t}})}
\|S_{\delta}(\tilde{u}_{\varepsilon})\|_{L^{2}(\Delta_{\bar{t}})}.
  \end{aligned}
\end{equation}
Appealing to the trace theorem, one may have
\begin{equation}\label{pri:5.91}
\begin{aligned}
\|\nabla S_{\delta}(\tilde{u}_{\varepsilon})\|_{L^{2}(\Delta_{\bar{t}})}
&\lesssim
\delta^{-\frac{1}{2}}\|\nabla S_{\delta}(\tilde{u}_{\varepsilon})\|_{L^{2}
(D_{\bar{t}}\cap O_\delta)}+
\delta^{\frac{1}{2}}\|\nabla^2 S_{\delta}(\tilde{u}_{\varepsilon})\|_{L^{2}(D_{\bar{t}}\cap O_\delta)}\\
&\lesssim^{\eqref{pri:2.26}}
\delta^{-\frac{1}{2}}\| S_{\delta}(\nabla\tilde{u}_{\varepsilon})\|_{L^{2}
(D_{\bar{t}}\cap O_\delta)}\\
&\lesssim
\delta^{-\frac{1}{p}}\| S_{\delta}(\nabla\tilde{u}_{\varepsilon})\|_{L^{p}
(D_{\bar{t}}\cap O_\delta)}\\
&\lesssim^{\eqref{pri:2.26}}\delta^{-\frac{1}{p}}\|\nabla \tilde{u}_{\varepsilon}\|_{L^{p}((D_{\bar{t}})_{\delta})}\\
&\lesssim^{\eqref{pri:2.22}} \delta^{-\frac{1}{p}}\| u_{\varepsilon}\|_{W^{1,p}(D_{5/2}^\varepsilon)}
\lesssim^{\eqref{pri:5.8}}\delta^{-\frac{1}{p}}\|u_{\varepsilon}
\|_{L^{2}(D^{\varepsilon}_{3})}
\end{aligned}
\end{equation}
where the third step follows from H\"{o}lder's inequality.
By the same token, it follows that
\begin{equation*}
\begin{aligned}
\int_{\Delta_{\bar{t}}}|S_{\delta}(\tilde{u}_{\varepsilon})|^2dS
&\lesssim\frac{1}{\varepsilon}\int_{O_{\varepsilon}\cap D_{\bar{t}}}
|S_{\delta}(\tilde{u}_{\varepsilon})|^2dx+
\varepsilon\int_{O_{\varepsilon}\cap D_{\bar{t}}}
|\nabla S_{\delta}(\tilde{u}_{\varepsilon})|^2dx\\
&\lesssim^{\eqref{pri:2.26}}
\frac{1}{\varepsilon}\int_{
(O_{\varepsilon}\cap D_{\bar{t}})_\delta}
|\tilde{u}_{\varepsilon}|^2dx+
\varepsilon\int_{(O_{\varepsilon}\cap D_{\bar{t}})_\delta}
|\nabla\tilde{u}_{\varepsilon}|^2dx\\
&\lesssim^{\eqref{pri:2.10}}
\Big\{\frac{\delta^2}{\varepsilon}+\varepsilon
\Big\}\int_{(O_{\varepsilon}\cap D_{\bar{t}})_\delta}
|\nabla\tilde{u}_{\varepsilon}|^2dx\\
&\lesssim^{\eqref{pri:5.3}}
\Big\{\frac{\delta^2}{\varepsilon}+\varepsilon
\Big\}
\int_{D_3^\varepsilon}
|u_{\varepsilon}|^2dx.
\end{aligned}
\end{equation*}
Then, we have
\begin{equation}\label{pri:5.101}
\bigg(
\int_{\Delta_{t}}|S_{\delta}(\tilde{u}_{\varepsilon})|^2dS
\bigg)^{1/2}\lesssim
\Big\{\frac{\delta}{\varepsilon^{1/2}}
+\varepsilon^{1/2}\Big\}
\|u_{\varepsilon}\|_{L^{2}(D_{3}^{\varepsilon})}.
\end{equation}
Combing the estimates \eqref{pri:5.7}-\eqref{pri:5.101}, we have
\begin{equation}\label{pri:5.111}
\begin{aligned}
I_3^2
&\lesssim
\Big\{\delta^{1-\frac{1}{p}}
\varepsilon^{-\frac{1}{2}}
+\varepsilon^{\frac{1}{2}}\delta^{-\frac{1}{p}}
\Big\}
\|u_{\varepsilon}\|_{L^{2}(D_{3}^{\varepsilon})}^2.
%\lesssim
%\varepsilon^{\frac{1}{2}-\frac{1}{p}}
%\|u_{\varepsilon}\|^2_{L^{2}(D_{3}^{\varepsilon})},
\end{aligned}
\end{equation}

Consequently, plugging the estimates $\eqref{pri:5.6}$, $\eqref{pri:5.131}$ and $\eqref{pri:5.111}$ back into
$\eqref{f:5.4}$ leads to
\begin{equation}\label{f:5.2}
\begin{aligned}
\|u_{\varepsilon}-v\|_{L^{2}(D_{1/12}^{\varepsilon})}&
\leq
\|v_{\varepsilon}-v_{h}\|_{L^{2}(D_{\bar{t}}^{\varepsilon})}
+\|u_{\varepsilon}-v_{\varepsilon}-z_{\varepsilon}
\|_{L^{2}(D_{\bar{t}/3}^{\varepsilon})}
+\|z_{\varepsilon}\|_{L^{2}(D_{\bar{t}}^{\varepsilon})}
+\|v_{h}-v\|_{L^{2}(D_{\bar{t}}^{\varepsilon})}\\
&\lesssim\bigg\{\varepsilon^{\sigma}
+\delta^{1/2}
+\Big(\delta^{1-\frac{1}{p}}
\varepsilon^{-\frac{1}{2}}
+\varepsilon^{\frac{1}{2}}\delta^{-\frac{1}{p}}
\Big)^{\frac{1}{2}}
\bigg\}
\|u_{\varepsilon}\|_{L^{2}(D_{3}^{\varepsilon})}\\
&\lesssim\varepsilon^{\frac{\sigma}{2}}
\|u_{\varepsilon}\|_{L^{2}(D_{3}^{\varepsilon})},
\end{aligned}
\end{equation}
where one may choose $\delta = \varepsilon$ for the last  inequality by noting that $0<p-2\ll1$ and
$\sigma=1/2-1/p$.
By appealing to rescaling arguments as in the proof of
Lemma $\ref{lemma:4.1}$, we can derive the stated estimate
$\eqref{pri:5.2}$ from $\eqref{f:5.2}$ and we do not reproduce the details here. We have completed the whole proof.
\end{proof}

For the ease of statement, we impose the following notations:
Let $g_M(x)= M\cdot x$ be a linear function with a direction vector $M\in\mathbb{R}^d$, and
\begin{equation}\label{def:5.1}
\begin{aligned}
& J_{\varepsilon}(r,v) =
\frac{1}{r}\inf_{M\in\mathbb{R}^{d}}
\Bigg\{\Big(\dashint_{D^{\varepsilon}_r}
|v-g_M|^2dx\Big)^{\frac{1}{2}}
+r\|\nabla_{\text{tan}}g_M\|_{L^{\infty}(\Delta_{r})}+
\|g_M\|_{L^{\infty}(\Delta_{r})}\Bigg\};\\
& J(r,v)=\frac{1}{r}\inf_{M\in\mathbb{R}^{d}}
\Bigg\{\Big(\dashint_{D_r}
|v-g_M|^2dx\Big)^{\frac{1}{2}}
+r\|\nabla_{\text{tan}}g_M\|_{L^{\infty}(\Delta_{r})}+\|g_M\|_{L^{\infty}(\Delta_{r})}\Bigg\},\\
\end{aligned}
\end{equation}
where we recall the tangential derivative
$\nabla _{\text{tan}}$ in Subsection $\ref{subsec:1.2}$.
\begin{lemma}[comparing at large-scales near boundaries]\label{lemma:5.3}
 Let $\varepsilon\leq r\leq 1$. Suppose that $\omega$ is
 a regular reference domain and $\partial \Omega\in C^{1,1}$.  Assume that $v\in H^{1}(D_{4r})$ is a solution to $\mathcal{L}_{0}(v)=0$ in $D_{4r}$ with $v=0$ on $\Delta_{4r}$. Then one may derive that
 \begin{equation}\label{pri:5.17}
   J(r,v)\lesssim J_{\varepsilon}(2r,v),
 \end{equation}
where the up to constant depends on $\mu_{0},\mu_{1},d,\omega$.
\end{lemma}
\begin{proof}
For any $M\in \mathbb{R}^{d}$, let $\tilde{v}=v-g_M$, and we have $\tilde{v}(x)=-g_M$ on $\Delta_{4r}$. The proof is reduced to show that
\begin{equation}\label{pri:5.15}
  \Big(\dashint_{D_{r}}|\tilde{v}|^{2}\Big)^{1/2}\lesssim
  \Big(\dashint_{D_{2r}^{\varepsilon}}|\tilde{v}|^{2}\Big)^{1/2}
  +r\|\nabla_{\text{tan}}g_M\|_{
  L^{\infty}(\Delta_{2r})}+\|g_M\|_{L^{\infty}(\Delta_{2r})},
\end{equation}
where the up to constant is independent of $\varepsilon$ and $r$.

First, we introduce a cutoff function $\varphi^{\varepsilon}$ as we did in Lemma \ref{lemma:2.6}. Recalling that $\mathbb{R}^{d}\backslash \omega=\bigcup^{\infty}_{k=1}H_{k}$ and
$0<\mathfrak{g}^{\omega}\leq\inf_{i\neq j}\bigg\{\text{dist}(H_{i},H_j)\bigg\}$. Let
$0<c<\mathfrak{g}^{\omega}/10$, $\varphi^{\varepsilon}(x)=1$ if $x\in\mathbb{R}^{d}\backslash \varepsilon\omega$, $\varphi^{\varepsilon}(x)=0$ if
$\text{dist}(x,\mathbb{R}^{d}\backslash\varepsilon\omega)\geq c\varepsilon$ and $|\nabla \varphi^{\varepsilon}|\lesssim 1/{\varepsilon}$. So, the support of $\varphi^{\varepsilon}$ is around the holes $\varepsilon H_{k}$. Obviously, it holds that
\begin{equation}\label{pri:5.14}
  \dashint_{D_{r}}|\tilde{v}|^2\lesssim
  \dashint_{D_{r}^{\varepsilon}}|\tilde{v}|^2+
  \dashint_{D_{r}\backslash\varepsilon\omega}|\tilde{v}|^2,
\end{equation}
we only need to deal with the second term of the right hand side above. One may choose a domain $\tilde{D}_{r}$ satisfying that
$D_{r}\subseteq \tilde{D}_{r}\subseteq D_{2r}$ and $\varphi^{\varepsilon}=0$ on $\partial \tilde{D}_{r}\backslash \Delta_{2r}$.
Due to the construction of $\varphi^{\varepsilon}$ and $\tilde{D}_{r}$, one may derive
\begin{equation}\label{pri:5.9}
\begin{aligned}
\dashint_{D_{r}\backslash\varepsilon\omega}|\tilde{v}|^2\leq
\dashint_{D_{r}}
|\varphi^{\varepsilon}\tilde{v}|^2
&\lesssim\varepsilon^{2}\dashint_{D_{r}}
|\nabla(\varphi^{\varepsilon}\tilde{v})|^2\\
&\lesssim\dashint_{D_{2r}^{\varepsilon}}|\tilde{v}|^2+\varepsilon^{2}
\dashint_{\tilde{D}_{r}}
|\varphi^{\varepsilon}\nabla \tilde{v}|^2
\end{aligned}
\end{equation}
Let $\tilde{g}$ be the classic linear extension of $g_M$ such that $\tilde{g}\in H^{1}_{0}(\Delta_{3r})$, $\tilde{g}=g_M$ on $\Delta_{2r}$ and $\|\tilde{g}\|_{H^{1}(\Delta_{3r})}\lesssim \|g_M\|_{H^{1}(\Delta_{2r})}$. Next, we consider $G$ satisfying:
\begin{equation*}
\left\{
\begin{aligned}
\mathcal{L}_{0}(G)&=0 &\qquad&\text{in}\quad D_{3r},\\
G&=\tilde{g}&\qquad&\text{on}\quad\partial D_{3r}.
\end{aligned}\right.
\end{equation*}
Appealing to the rescaling arguments,
we set $G_r(y):=\frac{1}{r}G(ry)$;
$\tilde{g}_r(y):=\frac{1}{r}\tilde{g}(ry)$
and $g_{M,r}:=\frac{1}{r}g_M(ry)$, where
$y\in D_{3}$ and $x=ry\in D_{3r}$. Thus, it is clear to see that $G_r$ is associated with $\tilde{g}_r$ by the same type
equations: $\mathcal{L}_0(G_r) = 0$ in $D_{3}$ with
$G_r = \tilde{g}_r$ on $\partial D_3$. So, it follows from energy estimates that
\begin{equation*}
\|\nabla G_r\|_{L^{2}(D_{3})}\lesssim
\|\tilde{g}_r\|_{H^{1/2}(\partial D_{3})}\lesssim
\|\tilde{g}_r\|_{H^{1}(\partial D_{3})}\lesssim\|g_{M,r}\|_{H^{1}(\Delta_{2})},
\end{equation*}
where we use the facts that $G_r=\tilde{g}_r=0$ on $\partial D_{3}\backslash\Delta_{3}$ and
$\|\tilde{g}_r\|_{H^{1}(\Delta_{3})}\lesssim \|g_{M,r}\|_{H^{1}(\Delta_{2})}$ in the last step.
Therefore, scaling back one may obtain the estimate
\begin{equation}\label{pri:5.12}
\Big(\dashint_{D_{3r}} |\nabla G|^2\Big)^{1/2}
\lesssim
\Big(\dashint_{\Delta_{2r}} |\nabla_{\text{tan}} g_M|^2\Big)^{1/2}
+ \frac{1}{r}\Big(\dashint_{\Delta_{2r}}
|g_M|^2\Big)^{1/2}
\end{equation}
Let $w=\tilde{v}-G$, then $w=0$ on $\Delta_{2r}$ and
$\nabla w=\nabla v-(M+\nabla G)$.
According to \eqref{pri:2.3} and $\varphi^{\varepsilon}w=0$ on $\partial \tilde{D}_r$, it follows that
\begin{equation*}
\begin{aligned}
\int_{\tilde{D}_{r}}|\nabla w|^2(\varphi^{\varepsilon})^2&\lesssim
\int_{\tilde{D}_{r}}\Big[\widehat{A}(\nabla v)-\widehat{A}(M+\nabla G)\Big]\nabla w(\varphi^\varepsilon)^2\\
&=-2\int_{\tilde{D}_{r}}\Big[\widehat{A}(\nabla v)-\widehat{A}(M+\nabla G)\Big]\nabla \varphi^\varepsilon
w\varphi^\varepsilon-\int_{\tilde{D}_{r}}\nabla\cdot
\Big[\widehat{A}(\nabla v)-\widehat{A}(M+\nabla G)\Big](\varphi^\varepsilon)^2w\\
&\lesssim \int_{\tilde{D}_{r}}|\nabla w||\nabla \varphi^{\varepsilon}||w\varphi^{\varepsilon}|-\int_{\tilde{D}_{r}}
\nabla\cdot\Big[\widehat{A}(\nabla v)-\widehat{A}(M+\nabla G)\Big](\varphi^\varepsilon)^2w.
\end{aligned}
\end{equation*}
By noting that $\mathcal{L}_{0}(v)=0$ and $\mathcal{L}_{0}(g_M)=0$, it follows from Young's inequality that
\begin{equation}\label{pri:5.10}
 \int_{\tilde{D}_{r}}|\varphi^{\varepsilon}\nabla w|^2
 \lesssim \delta \int_{\tilde{D}_{r}}|\varphi^{\varepsilon}\nabla w|^2+C_{\delta}\int_{\tilde{D}_{r}}|w\nabla \varphi^{\varepsilon}|^2-\int_{\tilde{D}_{r}}
 \nabla\cdot\Big[\widehat{A}(M)-\widehat{A}(M+\nabla G)\Big](\varphi^\varepsilon)^2w.
\end{equation}
On account of the divergence theorem for the last term in \eqref{pri:5.10}, one may have
\begin{equation}\label{pri:5.11}
\begin{aligned}
-\int_{\tilde{D}_{r}}
\nabla\cdot[\hat{A}(M)-\hat{A}(M+\nabla G)](\varphi^\varepsilon)^2w
&=\int_{\tilde{D}_{r}}
[\hat{A}(M)-\hat{A}(M+\nabla G)](\nabla w |\varphi^\varepsilon|^2+2\nabla\varphi^{\varepsilon}
w\varphi^{\varepsilon})\\
&\lesssim\int_{\tilde{D}_{r}}|\nabla G|(|\nabla w||\varphi^\varepsilon|^2+|\nabla\varphi^{\varepsilon}
||w\varphi^{\varepsilon}|)\\
&\lesssim \delta\int_{\tilde{D}_{r}}|\varphi^{\varepsilon}\nabla w|^2+C_{\delta}\int_{\tilde{D}_{r}}(|\varphi^{\varepsilon}\nabla G|^2+|w\nabla \varphi^{\varepsilon}|^2).
 \end{aligned}
\end{equation}
Plugging  \eqref{pri:5.11} back into \eqref{pri:5.10} and choosing $\delta$ small enough, we derive that
\begin{equation*}
\begin{aligned}
\int_{\tilde{D}_{r}}|\varphi^{\varepsilon}\nabla w|^2
\lesssim \int_{\tilde{D}_{r}}|w\nabla\varphi^{\varepsilon}|^2
+\int_{\tilde{D}_{r}}|\varphi^{\varepsilon}\nabla G|^2
\lesssim\varepsilon^{-2} \int_{D_{2r}^{\varepsilon}}w^2
+\int_{\tilde{D}_{r}}|\varphi^{\varepsilon}\nabla G|^2.
\end{aligned}
\end{equation*}
Reviewing the relationship $\tilde{v} = w + G$, the above estimate implies
\begin{equation}\label{pri:5.13}
\begin{aligned}
\varepsilon^{2}\dashint_{\tilde{D}_{r}}
|\varphi^{\varepsilon}\nabla \tilde{v}|^2
&\lesssim\varepsilon^{2}\dashint_{\tilde{D}_{r}}
|\varphi^{\varepsilon}\nabla w|^2+\varepsilon^{2}\dashint_{\tilde{D}_{r}}
|\varphi^{\varepsilon}\nabla G|^2\\
&\lesssim\dashint_{D_{2r}^{\varepsilon}}w^2+\varepsilon^2\dashint_{D_{3r}}
|\nabla G|^2\lesssim^{\eqref{pri:5.12}}
\dashint_{D_{2r}^{\varepsilon}}\tilde{v}^2+
\dashint_{\Delta_{2r}}
|g_M|^2+r^2\dashint_{\Delta_{2r}}
|\nabla_{\text{tan}}g_M|^2,
\end{aligned}
\end{equation}
in which we also employ the fact $\varepsilon\leq r$ in the last inequality.

Consequently,
collecting the estimates \eqref{pri:5.14}, \eqref{pri:5.9}, \eqref{pri:5.13} gives the desired estimate \eqref{pri:5.15}.
We have completed the whole proof.
\end{proof}

\begin{lemma}\label{lemma:5.4}
Let $\varepsilon<r<1$ and $\Omega$ be a bounded $C^{1,1}$ domain. Suppose that $v$ is the solution to $\mathcal{L}_{0}v=0$ in $D_{4r}$ with $v=0$ on $\Delta_{4r}$. Then for any $\theta\in(0,1/4)$ there holds
\begin{equation}\label{pri:5.16}
  J_{\varepsilon}(\theta r, v)\lesssim \theta^{\alpha}J_{\varepsilon}(r,v),
\end{equation}
where $\alpha\in (0,1)$ and the up to constant depends on $\mu_{0},\mu_{1},d,\mathfrak{g}^\omega$ and the characters of $\omega$ and $\Omega$.
\end{lemma}
\begin{proof}
  By the definition of $J_{\varepsilon}(\theta r,v)$ with  $\theta\in(0,\frac{1}{4})$, we have
\begin{equation*}
J_{\varepsilon}(\theta r,v)=\frac{1}{\theta r}\inf_{M\in \mathbb{R}^d}\bigg\{\Big(\dashint_{D_{\theta r}^{\varepsilon}}|v-g_M|^2\Big)^{1/2}
+\theta r\|\nabla_{\text{tan}}g_M\|_{L^{\infty}(\Delta_{\theta r})}+\|g_M\|_{L^{\infty}(\Delta_{\theta r})}
\bigg\}.
\end{equation*}
We may choose $M_0=\nabla v(x_{0})$ for some $x_{0}\in\Delta_{\theta r}$ here, and it follows from mean value theorem that
\begin{equation*}
\begin{aligned}
J_{\varepsilon}(\theta r,v)\lesssim J(\theta r,v)
&\lesssim
\frac{1}{\theta r}\bigg\{\Big(\dashint_{D_{\theta r}}|v-g_{M_0}|^2\Big)^{1/2}+\|v-g_{M_0}\|_{L^{\infty}(\Delta_{\theta r})}\bigg\}+\|g_{M_0}\|_{L^{\infty}(\Delta_{\theta r})}\\
&\lesssim (\theta r)^{\alpha}[\nabla v]_{C^{0,\alpha}(D_{\theta r}\cup \Delta_{\theta r})}.
\end{aligned}
\end{equation*}
By noting that $[\nabla v]_{C^{0,\alpha}(D_{\theta r}\cup \Delta_{\theta r})}=[\nabla \tilde{v}]_{C^{0,\alpha}(D_{\theta r}\cup \Delta_{\theta r})}$ if $\tilde{v}(x)=v(x)-g_M$ for any $M\in\mathbb{R}^d$, one may have
\begin{equation*}
  r^{\alpha}[\nabla v]_{C^{0,\alpha}(D_{\theta r}\cup \Delta_{\theta r})}\lesssim^{\eqref{pri:7.5}} \frac{1}{r}\bigg\{\Big(\dashint_{D_{\frac{r}{2}}}|\tilde{v}|^2\Big)^{1/2}
  +\|g_M\|_{L^{\infty}(\Delta_{\frac{r}{2}})}\bigg\}
  +\|\nabla_{\text{tan}}g_M\|_{L^{\infty}(\Delta_{\frac{r}{2}})},
\end{equation*}
For $M\in \mathbb{R}^d$ is arbitrary, the desired result
$\eqref{pri:5.16}$  finally follows from the estimate
$\eqref{pri:5.17}$. We have completed the whole proof.
\end{proof}

\begin{lemma}[iteration's inequality II]
Let $\sigma$ be given in Lemma $\ref{lemma:5.2}$.
Assume the same conditions as in Theorem $\ref{thm:1.3}$.
Let $u_\varepsilon$ be a weak solution of
$\eqref{pde:7.4}$.
%\begin{equation*}
% \left\{ \begin{aligned}
% \mathcal{L}_{\varepsilon}u_\varepsilon  &= 0& ~\emph{in}&~ D^{\varepsilon}_{4},\\
% \sigma_{\varepsilon}(u_{\varepsilon})&= 0& ~\emph{on}&~
% D_{4}\cap\partial(\varepsilon\omega),\\
% u_{\varepsilon}&= 0&~\emph{on}&~\Delta_{4}\cap\varepsilon\omega.
%  \end{aligned}\right.
%\end{equation*}
Then there exists $\theta\in(0,1/4)$ such that
\begin{equation}\label{pri:5.22}
 J_{\varepsilon}(\theta r, u_\varepsilon) \leq \frac{1}{2}J_{\varepsilon}(r,u_\varepsilon)
 + C\left(\frac{\varepsilon}{r}\right)^{\frac{\sigma}{2}}\Phi(2r)
\end{equation}
for any $\varepsilon\leq r<1$, and we impose the new notation
$\Phi(r):=\frac{1}{r}\Big(\dashint_{D_r^{\varepsilon}}
|u_{\varepsilon}|^2\Big)^{\frac{1}{2}}$.
\end{lemma}
\begin{proof}
The proof directly follows from Lemmas \ref{lemma:5.2} and \ref{lemma:5.4} and we omit the details.
\end{proof}
\noindent\textbf{The proof of Theorem \ref{thm:1.3}}
The desired estimate \eqref{pri:1.14} mainly follows from \eqref{pri:5.22} coupled with Lemma $\ref{lemma:4.3}$ (see \cite[Lemma 8.5]{S5}) and Caccioppoli's inequality \eqref{pri:5.1}. For the details on interior Lipschitz estimates have already been fully shown in Section $\ref{sec:4}$, we leave it to the reader.
\qed

\section{Quenched Caldenr\'{o}n-Zygmund estimates}\label{sec:6}
\begin{lemma}[Shen's lemma]\label{lemma:6.1}
Suppose that $q>2$ and $\Omega$ be a bounded Lipschitz domain. Let $F\in L^2(\Omega)$ and $f\in L^{p}(\Omega)$ for some $2<p<q$. Suppose that for each ball with the property that $|B|\leq c_{0}|\Omega|$ and either $4B\subset\Omega$ or $B$ is centered on $\partial\Omega$,
there exist two measurable functions $F_{B}$ and $R_B$ on $\Omega\cap2B$, such that $|F|\leq|F_{B}|+|R_B|$ on $\Omega\cap2B$,
\begin{equation}\label{pri:6.1}
\begin{aligned}
\bigg(\dashint_{2B\cap\Omega}|R_B|^q\bigg)^{\frac{1}{q}}
&\leq N_{1}\bigg\{\Big(\dashint_{4B\cap\Omega}|F|^2\Big)^{\frac{1}{2}}
+\sup_{4B_0\supseteq B'\supseteq B}\Big(\dashint_{B'\cap\Omega}|f|^2\Big)^{\frac{1}{2}}\bigg\},\\
\bigg(\dashint_{2B\cap\Omega}|F_B|^2\bigg)^{\frac{1}{2}}
&\leq N_2\sup_{4B_0\supseteq B'\supseteq B}\bigg(\dashint_{B'\cap\Omega}|f|^2\bigg)^{\frac{1}{2}},
\end{aligned}
\end{equation}
where $N_{1},N_2>0$ and $0<c_0<1$. Then $F\in L^{p}(\Omega)$ and
\begin{equation}
\bigg(\int_{\Omega}|F|^p\bigg)^{\frac{1}{p}}\leq C\bigg\{\Big(
\int_{\Omega}|F|^2\Big)^{\frac{1}{2}}+\Big(
\int_{\Omega}|f|^p\Big)^{\frac{1}{p}}\bigg\},
\end{equation}
in which $C$ depends at most on $N_{1},N_2,c_0,p,q$ and the Lipschitz character of $\Omega$.
\end{lemma}
\begin{proof}
See \cite[Theorem 4.2.6]{S4} or \cite[Theorem 4.13]{S3}.
\end{proof}

\begin{lemma}[primary geometry on integrals]
Let $f\in L^1_{\text{loc}}(\mathbb{R}^d)$, and $\Omega\subset\mathbb{R}^d$ be a bounded Lipschitz domain. Then there hold the following inequalities:
\begin{enumerate}
\item[(1).] If $0<r<\varepsilon$ and $D_{r}(x_0)$ is given, then for any $x\in D_{r}(x_0)$ we have
\begin{equation}\label{pri:6.3}
\dashint_{B_{\varepsilon}(x)\cap\Omega}|f|\lesssim\dashint_{D_{4r}(x_0)}
\dashint_{B_{\varepsilon}(x)\cap\Omega}|f|dx;
\end{equation}
\item[(2).] $r\geq\varepsilon$ and $D_{r}(x_0)$ is given, then there holds
\begin{equation}\label{pri:6.4}
\dashint_{D_r(x_0)}|f|\lesssim\dashint_{D_{2r}(x_0)}
\dashint_{B_{\varepsilon}(x)\cap\Omega}|f|\lesssim\dashint_{D_{6r}(x_0)}|f|,
\end{equation}
\end{enumerate}
where the up to constant depends only on $d$.
\end{lemma}
\begin{proof}
  See \cite[Lemma 6.3]{WXZ2020}, while in the case of $\Omega=\mathbb{R}^d$, we refer the readers to \cite[Lemma 6.5]{DO}
\end{proof}
\noindent\textbf{The proof of Theorem \ref{thm:1.4}}
The main idea follows from \cite[Theorem 1.5]{WXZ2020} and the main tool of the proof is Shen's real methods \cite{S3}.
Let $B:=B(x_0,r)$ be any ball with $0<r<\frac{r_0}{10}$ such that $x_{0}\in\partial\Omega$ or $4B\subset \Omega$. We define the following quantities for the ease of statement:
\begin{equation*}
\begin{aligned}
  U(x)&:=\Big(\dashint_{B(x,\varepsilon)\cap\Omega_{\varepsilon}}
  |\nabla u_{\varepsilon}|^2\Big)^{\frac{1}{2}},\qquad
  &F(x)&:=\Big(\dashint_{B(x,\varepsilon)\cap\Omega_{\varepsilon}}
  |f|^2\Big)^{\frac{1}{2}},\\
  W_{B}(x)&:=\Big(\dashint_{B(x,\varepsilon)\cap\Omega_{\varepsilon}}
  |\nabla w_{\varepsilon}|^2\Big)^{\frac{1}{2}},\qquad
  &V_B(x)&:=\Big(\dashint_{B(x,\varepsilon)\cap\Omega_{\varepsilon}}
  |\nabla v_{\varepsilon}
  |^2\Big)^{\frac{1}{2}}
  \end{aligned}
\end{equation*}
for any $x\in\Omega$ and $w_{\varepsilon},v_{\varepsilon}$ will be defined later. We mention that $\tilde{u}_{\varepsilon},\tilde{v}_{\varepsilon}$ and $\tilde{w}_{\varepsilon}$ are corresponding extension functions defined by Lemma \ref{extensiontheory}.

First, we consider the case: $0<r<\varepsilon$. It's fine to fix $W_B=U, V_B=0$ in this case. It's obviously that
\begin{equation}\label{pri:6.5}
\Big(\dashint_{B\cap\Omega}V_{B}^2\Big)^{\frac{1}{2}}\lesssim
\Big(\dashint_{B\cap\Omega}F^2\Big)^{\frac{1}{2}}.
\end{equation}
For any $x\in B\cap\Omega$, we have
\begin{equation}\begin{aligned}
W_B^2(x)=\dashint_{B(x,\varepsilon)\cap\Omega_{\varepsilon}}
  |\nabla u_{\varepsilon}|^2
  &=\dashint_{B(x,\varepsilon)\cap\Omega}l_{\varepsilon}^{+}
  |\nabla \tilde{u}_{\varepsilon}|^2\\
  &\lesssim^{\eqref{pri:6.3}}\dashint_{4B\cap\Omega}\dashint_{B(x,\varepsilon)\cap\Omega}
  l_{\varepsilon}^{+}  |\nabla \tilde{u}_{\varepsilon}|^2=\dashint_{4B\cap\Omega}U^2.
  \end{aligned}
\end{equation}
For any $p\geq2$,  this leads to
\begin{equation}\label{pri:6.7}
\Big(\dashint_{B\cap\Omega}W_{B}^p\Big)^{\frac{1}{p}}
\leq\sup_{x\in B\cap\Omega}|W_B(x)|\lesssim
\Big(\dashint_{4B\cap\Omega}U^2\Big)^{\frac{1}{2}}+
\Big(\dashint_{B\cap\Omega}F^2\Big)^{\frac{1}{2}}.
\end{equation}

Then, we consider the case: $r\geq\varepsilon$. Let $w_{\varepsilon}$ satisfy the following equation:
\begin{equation}\label{pri:6.8}
\left\{\begin{aligned}
\mathcal{L}_{\varepsilon}w_{\varepsilon}&=0&\qquad &\text{in}\quad D_{12r}^{\varepsilon}(x_0),\\
\sigma_{\varepsilon}(w_{\varepsilon})&=0&\qquad&\text{on}\quad \partial D_{12r}^{\varepsilon}(x_{0})|_{D_{12r}(x_{0})}\\
w_{\varepsilon}&=u_{\varepsilon}&\qquad&\text{on}\quad \partial D_{12r}^{\varepsilon}(x_{0})|_{\partial D_{12r}(x_{0})}.
\end{aligned}\right.
\end{equation}
For any $x\in B\cap\Omega$, it follows from boundary and interior Lipschitz estimates \eqref{pri:1.1} and \eqref{pri:1.14} that
\begin{equation}
\begin{aligned}
\dashint_{B(x,\varepsilon)\cap\Omega_{\varepsilon}}
|\nabla w_{\varepsilon}|^2\lesssim
\dashint_{D^{\varepsilon}_{2r}(x_{0})}|\nabla w_{\varepsilon}|^2
&=\dashint_{D_{2r}(x_{0})}l^{+}_{\varepsilon}|\nabla \tilde{w}_{\varepsilon}|^2\\
&\lesssim^{\eqref{pri:6.4}}\dashint_{D_{4r}(x_{0})}
\dashint_{B(x,\varepsilon)\cap\Omega}l^{+}_{\varepsilon}|\nabla \tilde{w}_{\varepsilon}|^2=\dashint_{B_{4r}(x_0)\cap\Omega}
W_B^2,
\end{aligned}
\end{equation}
and this implies
\begin{equation}\label{pri:6.10}
 \sup_{x\in B\cap\Omega}|W_B(x)|^2\lesssim \dashint_{B_{4r}(x_0)\cap\Omega}(U^2+V_B^2).
\end{equation}
Set $v_{\varepsilon}=u_{\varepsilon}-w_{\varepsilon}$, then we have
\begin{equation}
\dashint_{B_{4r}(x_0)\cap\Omega}V_B^2\lesssim^{\eqref{pri:6.4}}
\dashint_{D_{12r}(x_{0})}l^+_{\varepsilon}|\nabla \tilde{v}_{\varepsilon}|^2=\dashint_{D^{\varepsilon}_{12r}(x_{0})}
|\nabla u_{\varepsilon}-\nabla w_{\varepsilon}|^2.
\end{equation}
In view of \eqref{pri:1.15} and \eqref{pri:6.8}, for any $\varphi\in H^{1}(D^{\varepsilon}_{12r}(x_{0});\partial D_{12r}^{\varepsilon}(x_{0})|_{\partial D_{12r}(x_{0})})$, one may have
\begin{equation}\label{pri:6.12}
\int_{D^{\varepsilon}_{12r}(x_{0})}
[A(x/\varepsilon,\nabla u_{\varepsilon})-A(x/\varepsilon,\nabla w_{\varepsilon})] \cdot\nabla\varphi=-\int_{D^{\varepsilon}_{12r}(x_{0})}f\cdot\nabla\varphi.
\end{equation}
By setting $\varphi=u_{\varepsilon}-w_{\varepsilon}$, it follows from \eqref{a:1} that
\begin{equation*}
  \text{LHS~of \eqref{pri:6.12}}\geq\mu_{0}\int_{D^{\varepsilon}_{12r}(x_{0})}
  |\nabla u_{\varepsilon}-\nabla w_{\varepsilon}|^2,
\end{equation*}
and
\begin{equation*}
  \text{RHS~of \eqref{pri:6.12}}\leq\frac{\mu_{0}}{2}\int_{D^{\varepsilon}_{12r}(x_{0})}
  |\nabla u_{\varepsilon}-\nabla w_{\varepsilon}|^2
  +C\int_{D^{\varepsilon}_{12r}(x_{0})}|f|^2,
\end{equation*}
where we employ Young's inequality. The above two inequalities leads to
\begin{equation}\label{pri:6.13}
\dashint_{D^{\varepsilon}_{12r}(x_{0})}
  |\nabla u_{\varepsilon}-\nabla w_{\varepsilon}|^2\lesssim
  \dashint_{D^{\varepsilon}_{12r}(x_{0})}|f|^2
  =\dashint_{D_{12r}(x_{0})}l^+_{\varepsilon}|f|^2\lesssim^{\eqref{pri:6.4}}
  \dashint_{D_{24r}(x_{0})}|F|^2,
\end{equation}
Combing \eqref{pri:6.10} and \eqref{pri:6.13}, for any $p\geq 2$, one may have
\begin{equation}\label{pri:6.14}
\begin{aligned}
\Big(\int_{B\cap\Omega}W_{B}^p\Big)^{\frac{1}{p}}
&\leq\sup_{x\in B\cap\Omega}|W_{B}(x)|
\lesssim \Big(\dashint_{B_{4r}(x_0)\cap\Omega}U^2\Big)^{\frac{1}{2}}
+\Big(\dashint_{B_{24r}(x_0)\cap\Omega}F^2\Big)^{\frac{1}{2}},\\
\Big(\int_{B\cap\Omega}V_{B}^2\Big)^{\frac{1}{2}}
&\lesssim\Big(\dashint_{B_{24r}(x_0)\cap\Omega}F^2\Big)^{\frac{1}{2}}.
\end{aligned}
\end{equation}
Therefore, the condition \eqref{pri:6.1} has been verified by estimates \eqref{pri:6.5},\eqref{pri:6.7} and \eqref{pri:6.14}, and it follows that for any $p\geq 2$,
\begin{equation*}
  \Big(\int_{\Omega}U^p\Big)^{\frac{1}{p}}\lesssim
  \Big(\int_{\Omega}U^2\Big)^{\frac{1}{2}}
  +\Big(\int_{\Omega}F^p\Big)^{\frac{1}{p}}.
\end{equation*}
This together with
\begin{equation*}
\begin{aligned}
  \int_{\Omega}U^2=\int_{\Omega}\dashint_{B(x,\varepsilon)
  \cap\Omega_{\varepsilon}}|\nabla u_{\varepsilon}|^2dx
  &\lesssim\int_{\Omega_{0}}|\nabla \tilde{u}_{\varepsilon}|^2\\
  &\lesssim^{\eqref{pri:2.23}}\int_{\Omega_{\varepsilon}}|\nabla u_{\varepsilon}|^2
  \lesssim \int_{\Omega_{\varepsilon}}|f|^2
  \lesssim^{\eqref{pri:6.4}}\int_{\Omega}|F|^2
  \lesssim\Big(\int_{\Omega}|F|^p\Big)^{\frac{2}{p}}
\end{aligned}
\end{equation*}
implies the desired estimate \eqref{pri:1.16}. We have completed the proof.
\qed

\section{Proofs of lemmas stated in Preliminaries}
\label{sec:8}

\noindent\textbf{The proof of Lemma \ref{lemma:2.1}.}
The estimates $\eqref{pri:2.1}$ and $\eqref{pri:2.2}$ have already been included in \cite{ZR}, while
the estimate $\eqref{pri:2.15}$ seems to be new derived here. The main
ideas can also be found in \cite{WXZ,W} and
we provide proofs in order for the reader's convenience.
Multiplying both sides of $\eqref{pde:1.2}$ by $N(y,\xi)$ and then integrating by parts, we have
\begin{equation*}
\begin{aligned}
0 &= \int_{Y\cap \omega} A(y,\xi+\nabla_y N(y,\xi))\cdot\nabla_yN(y,\xi) dy \\
 &= \int_{Y\cap \omega} A(y,\xi+\nabla_y N(y,\xi))\cdot\big(\xi+\nabla_yN(y,\xi)\big)dy
 - \int_{Y\cap \omega} A(y,\xi+\nabla_y N(y,\xi))dy\cdot\xi \\
 &\geq^{\eqref{a:1}} \mu_0 \int_{Y\cap \omega }|\xi + \nabla_y N(y,\xi)|^2 dy - \mu_1|\xi|\int_{Y\cap \omega}|\xi+\nabla_y N(y,\xi)|dy.
\end{aligned}
\end{equation*}
By Young's inequality,
\begin{equation*}
\int_{Y\cap \omega} |\xi + \nabla_y N(y,\xi)|^2 dy \leq C(\mu_0,\mu_1)|\xi|^2.
\end{equation*}
Thus this together with Poincar\'e's inequality will give the stated estimate $\eqref{pri:2.1}$.

To show the estimate $\eqref{pri:2.2}$, we start with the following identity
\begin{equation}\label{f:2.1}
\begin{aligned}
&\int_{Y\cap \omega} \big[A(y,\xi+\nabla_y N(y,\xi)) - A(y,\xi^\prime+\nabla_y N(y,\xi^\prime))\big]\cdot
\big[\xi-\xi^\prime + \nabla_y N(y,\xi) - \nabla_y N(y,\xi^\prime)\big]dy \\
 &= \int_{Y\cap \omega} \big[A(y,\xi+\nabla_y N(y,\xi)) - A(y,\xi^\prime+\nabla_y N(y,\xi^\prime))\big] dy
\cdot \big(\xi-\xi^\prime\big),
\end{aligned}
\end{equation}
where we use the fact that $N(\cdot,\xi),N(\cdot,\xi^\prime)\in H^1_{per}(Y\cap \omega)$ satisfy the equation
$\eqref{pde:1.2}$ for $\xi,\xi^\prime\in\mathbb{R}^d$, respectively. By the assumption $\eqref{a:2}$,
the left-hand side above is greater than
\begin{equation*}
\mu_0\int_{Y\cap \omega} |\xi-\xi^\prime + \nabla_yN(y,\xi) - \nabla_y N(y,\xi^\prime)|^2 dy,
\end{equation*}
while it follows from Young's inequality that its right-hand side is less than
\begin{equation*}
\frac{\mu_0}{2}\int_{Y\cap \omega }|\xi-\xi^\prime + \nabla_yN(y,\xi) - \nabla_y N(y,\xi^\prime)|^2 dy
+ C(\mu_0,\mu_1)|\xi-\xi^\prime|^2.
\end{equation*}
Thus it is not hard to derive that
\begin{equation}\label{f:2.2}
\Big(\int_{Y\cap \omega} |\nabla_yN(y,\xi) - \nabla_y N(y,\xi^\prime)|^2 dy\Big)^{1/2} \lesssim |\xi-\xi^\prime|,
\end{equation}
and this will give the estimate $\eqref{pri:2.2}$ in a similar way.

Then we proceed to show \eqref{pri:2.15}. Let $u(y,\xi)=N(y,\xi)+y\cdot\xi$ and $\tilde{u}(y,\xi)=u(y,\xi)+\tilde{M}$, in which one may choose
$\tilde{M}$ such that $\tilde{u}$ is positive in $Y\cap \omega$.
%Due to the estimate in Lemma \ref{lemma:7.1}, if $\xi$ is bounded, the existence of $\tilde{M}$ is not hard to see, and such that the boundedness of $\xi$ will be given by $S_{\varepsilon}(\psi_{\varepsilon}\nabla u_{0})$ later.
Note that $\tilde{u}$ still satisfies the equation
\begin{equation*}
  \nabla\cdot A(y,\nabla u(y,\xi))=0, \text{~~in~}Y\cap \omega.
\end{equation*}
Thus, it follows from the local boundedness estimate and the weak Harnack inequality (see Lemma \ref{lemma:7.1}
for the case $B_r\cap \partial\omega \not=\emptyset$, and
\cite[Corollary 3.10, Theorem 3.13]{MZ} for the case
$B_{4r}\subset 2Y\cap \omega$) that
\begin{equation}\label{pri:2.16}
  \sup_{y\in Y\cap \omega\cap B_{r}}\tilde{u}(y,\xi)\lesssim\dashint_{Y\cap \omega\cap B_{\bar{r}}}\tilde{u}(\cdot,\xi)\qquad\text{and}\qquad
  \inf_{y\in Y\cap \omega\cap B_{r}}\tilde{u}(y,\xi)\gtrsim
  \dashint_{Y\cap \omega\cap B_{\bar{r}}}\tilde{u}(\cdot,\xi),
\end{equation}
in which $B_{4r}\subset B_{\bar{r}}\subset 2Y$.
Then for any $y\in Y\cap \omega$ such that $\tilde{u}(y,\xi)-\tilde{u}(y,\xi')>0$, it follows from \eqref{pri:2.16} that
\begin{equation*}
  \tilde{u}(y,\xi)-\tilde{u}(y,\xi')\leq\sup_{y\in Y\cap \omega\cap B_r}\tilde{u}(y,\xi)-\inf_{y\in Y\cap \omega \cap B_r}\tilde{u}(y,\xi')
  \lesssim\dashint_{Y\cap \omega\cap B_{\bar{r}}}|\tilde{u}(\cdot,\xi)-\tilde{u}(\cdot,\xi')|.
\end{equation*}
Similarly, for any $y\in Y\cap\omega$ such that $\tilde{u}(y,\xi')-\tilde{u}(y,\xi)>0$, we may have
\begin{equation*}
  \tilde{u}(y,\xi')-\tilde{u}(y,\xi)\leq\sup_{y\in Y\cap \omega\cap B_r}\tilde{u}(y,\xi')-\inf_{y\in Y\cap \omega \cap B_r}\tilde{u}(y,\xi)
  \lesssim\dashint_{Y\cap \omega\cap B_{\bar{r}}}|\tilde{u}(\cdot,\xi')-\tilde{u}(\cdot,\xi)|.
\end{equation*}

Therefore, for any $y\in Y\cap\omega$, we obtain that
\begin{equation*}
\begin{aligned}
  |\tilde{u}(y,\xi)-\tilde{u}(y,\xi')|
  & \lesssim\dashint_{Y\cap \omega\cap B_{\bar{r}}}|\tilde{u}(\cdot,\xi)-\tilde{u}(\cdot,\xi')|\\
&  \lesssim\dashint_{Y\cap \omega}|N(\cdot,\xi)-N(\cdot,\xi')|+|\xi-\xi'|.
  \end{aligned}
\end{equation*}
This together with
$\eqref{pri:2.2}$ implies \eqref{pri:2.15}, and we have completed the proof.
\qed

\noindent\textbf{The proof of Lemma \ref{lemma:2.55}.}
For any fixed $\xi,\xi'\in\mathbb{R}^{d}$, setting
$P_{1}=\xi+\nabla_{y}N(y,\xi)$ and  $P_{2}=\xi'+\nabla_{y}N(y,\xi')$, we have
\begin{equation*}
  \nabla\cdot[A(y,P_{1})-A(y,P_{2})]=0
  \quad \text{in}~ Y\cap\omega.
\end{equation*}
Under the assumptions $\eqref{a:1},\eqref{a:2}$ and $\eqref{a:3}$,
it is well-known that $P_1,P_2$ are H\"older continuous
(see for example \cite[Theorems 1.1, 1.3]{F}).
In view of the Newton-Leibniz formula,
\begin{equation*}
 \frac{\partial}{\partial y_i}\bigg[\int_{0}^{1}\partial_{\xi_{j}}
  A^i(y,tP_{1}+(1-t)P_{2})dt\cdot(P_{1}^j-P_2^j)\bigg]=0.
\end{equation*}
We write $a_{ij}(y)=\int_{0}^{1}\partial_{\xi_{j}}
  A^i(y,tP_{1}+(1-t)P_{2})dt$. Thus,
  this together with
  $A\in C^1(\mathbb{R}^{d}\times \mathbb{R}^{d})$ implies that $a_{ij}$ is continuous on $\mathbb{R}^d$.
  Setting $\pi=N(y,\xi)-N(y,\xi')$ we have
  \begin{equation*}
    -\nabla\cdot[a(y)\nabla \pi]=\nabla\cdot[a(y)](\xi-\xi')
    \qquad\text{in}\quad Y\cap\omega
  \end{equation*}
with a natural boundary condition
$\vec{n}\cdot a(P_1-P_2) =0$ on $\partial\omega$ and
$\pi$ being periodic on $\partial Y$.
It follows from
the $L^p$ estimate (see
for example \cite[Theorem 1.1]{KFS1})
that for any $p\geq 2$, there holds
\begin{equation*}
  \Big(\dashint_{Y\cap \omega}|\nabla \pi|^pdy\Big)^{\frac{1}{p}}\lesssim
  |\xi-\xi'|+ \Big(\dashint_{Y\cap \omega}|\nabla \pi|^2dy\Big)^{\frac{1}{2}}\lesssim^{\eqref{f:2.2}}|\xi-\xi'|,
\end{equation*}
and the proof is complete.
\qed

\noindent\textbf{The proof of Lemma \ref{lemma:2.2}.}
The proof relies on the extension theorem heavily,
and the idea
is inspired by \cite{PR,ZR}.
Due to the formula $\eqref{f:2.1}$, we have that
\begin{equation}\label{f:2.5}
\begin{aligned}
\big<\widehat{A}(\xi)-\widehat{A}(\xi^\prime),\xi-\xi^\prime\big>
&\geq \mu_0\dashint_{Y\cap\omega} |\xi-\xi^\prime + \nabla_y(N(y,\xi) - N(y,\xi^\prime))|^2 dy.
\end{aligned}
\end{equation}

There are two cases: (1) $\partial Y\cap(\mathbb{R}^d\setminus\omega)=\emptyset$;
(2) $\partial Y\cap(\mathbb{R}^d\setminus\omega) \not=\emptyset$.

For the case (1), it follows from \cite[Lemma 2.6]{APMP} that
there is a linear extension operator from $H^1(Y\cap\omega)$
to $H^1(Y)$ such that the extended function
(denoted by $\tilde{N}(y,\xi)$) satisfies the inequality
\begin{equation}\label{f:3.7}
 \int_{Y\cap\omega}|\nabla_{y}N(y,\xi)|^2dy\geq C\int_{Y}|\nabla_{y}\tilde{N}(y,\xi)|^{2}dy,
 \end{equation}
where $C$ is independent of $N$ and $\xi$. Since
$\tilde{N} = N$ on $\partial Y$ and $N\in H_{\text{per}}^1(Y\cap\omega)$, we have
\begin{equation}\label{f:2.4}
  \int_{\partial Y} n \tilde{N}(\cdot,\xi) dS = 0.
\end{equation}
Thus, one may derive from $\eqref{f:2.4}$ that
\begin{equation}\label{pri:1.3}
 \int_{Y\cap\omega}
 |\nabla_{y}N(y,\xi)+\xi|^2dy\geq C\int_{Y}
 |\nabla_{y}\tilde{N}(y,\xi)+\xi|^{2}dy
 \geq C|\xi|^2,
\end{equation}
where $\xi\in\mathbb{R}^d$ is arbitrary, and
we also employ the facts that the extension operator is
linear and the extension
of a linear function is itself. By the same token, it is
not hard to see
\begin{equation}\label{}
\int_{Y\cap\omega} |\xi-\xi^\prime + \nabla_y(N(y,\xi) - N(y,\xi^\prime))|^2 dy
\geq C|\xi-\xi^\prime|^2,
\end{equation}
and this together with $\eqref{f:2.5}$ implies the first line of $\eqref{pri:2.3}$ in such the case.

For the case (2), let $\tilde{N}(\cdot,\xi)$ be
the extension function of $N(\cdot,\xi)$ in the sense of
\cite[pp.47, Theorem 4.2]{OSY}. Then we have
$\eqref{f:3.7}$ with a different estimated constant.
Moreover, from the construction of the extension operator
in \cite[pp.47, Theorem 4.2]{OSY},
one may infer that $\tilde{N}(\cdot,\xi)\in H_{\text{per}}^1(Y)$, and therefore the equality $\eqref{f:2.4}$ also holds.
Consequently, the first line of $\eqref{pri:2.3}$
would be true following the same computation as in the case (1).

Then we turn to the second line of $\eqref{pri:2.3}$, and
note that
\begin{equation*}
\begin{aligned}
|\widehat{A}(\xi)-\widehat{A}(\xi^\prime)|&\leq
\dashint_{Y\cap\omega}\big|A(y,\xi+\nabla N(y,\xi))-A(y,\xi^\prime+\nabla N(y,\xi^\prime))\big| dy \\
&\leq^{\eqref{a:1}} \mu_1 \dashint_{Y\cap\omega} |\xi-\xi^\prime + \nabla N(y,\xi) - \nabla N(y,\xi^\prime)|dy \\
&\leq^{\eqref{f:2.2}} C|\xi-\xi^\prime|.
\end{aligned}
\end{equation*}
In view of Remark $\ref{remark:2.1}$, we may have the third line of
$\eqref{pri:2.3}$ and the proof is complete.
\qed

\noindent\textbf{The proof of Lemma \ref{lemma:2.5}.}
The proof is quite similar to the linear case
(see for example \cite{S4,ZVVPSE1}) and surprisingly
depends on a linear structure of an auxiliary equation.
It is clear to see that (i) and (ii) follow from the formula $\eqref{eq:1.1}$
and the equation $\eqref{pde:1.2}$, respectively. By (i),
there exists $f_i(\cdot,\xi)\in H_{\text{per}}^{2}(Y)$ such that $\Delta f_i(\cdot,\xi) = b_i(\cdot,\xi)$ in Y.
Let $E_{ji}(y,\xi)= \frac{\partial}{\partial y_j}\big\{f_{i}(y,\xi)\big\}
-\frac{\partial}{\partial y_i}\big\{f_{j}(y,\xi)\big\}$. Thus $E_{ji} = - E_{ij}$, and
one may derive the first expression in $\eqref{eq:2.1}$ from the fact $(\text{ii})$.
Then, the rest thing is to show the estimate $\eqref{pri:2.4}$.
For any $\xi,\xi^\prime\in\mathbb{R}^d$, note that
\begin{equation*}
\begin{aligned}
\int_{Y} |\nabla E_{ji}(y,\xi)- \nabla E_{ji}(y,\xi^\prime)|^2 dy
&\leq 2 \int_{Y} \big|\nabla^2 \big(f(y,\xi)-f(y,\xi^\prime)\big)\big|^2 dy\\
&\lesssim \int_{2Y} \big|b_{i}(y,\xi)-b_{i}(y,\xi^\prime)\big|^2 dy
\lesssim^{\eqref{pri:2.3},\eqref{a:1},\eqref{f:2.2}} |\xi-\xi^\prime|^2,
\end{aligned}
\end{equation*}
where we employ interior $H^2$ theory and energy estimates
for the constructed Poisson's equation in the second step.
This together with Poincar\'e's inequality finally leads to the desired estimate $\eqref{pri:2.4}$.

To show \eqref{pri:2.17}, we claim that
\begin{equation*}
  \|\nabla f(\cdot,\xi)-\nabla f(\cdot,\xi')
  \|_{L^{\infty}(Y)}\lesssim|\xi-\xi'|.
\end{equation*}
According to $\Delta f(y,\xi)=b(y,\xi)$ in $Y$, we have
\begin{equation*}
  \Delta[f(y,\xi)-f(y,\xi')]=b(y,\xi)-b(y,\xi')\text{~in~}Y,
\end{equation*}
and it follows from the assumption
$\eqref{a:1}$ and the estimate $\eqref{pri:2.3}$ that
\begin{equation*}
  |b(y,\xi)-b(y,\xi')|\lesssim |\xi-\xi'|+l^+|\nabla(N(y,\xi)-N(y,\xi'))|.
\end{equation*}
Due to Lemma \ref{lemma:2.55},
it is known that $\|\nabla(N(\cdot,\xi)-N(\cdot,\xi'))
\|_{L^{p}(Y\cap\omega)}\lesssim |\xi-\xi'|$ for any $p\geq 2$,
and this coupled with the above estimate gives
\begin{equation}\label{f:2.3}
\|b(y,\xi)-b(y,\xi')\|_{L^{p}(Y)}
\lesssim |\xi-\xi^\prime|
\end{equation}
for any $p\geq 2$.
By interior Lipschitz's estimates one may derive that
\begin{equation*}
\begin{aligned}
  \|\nabla f(\cdot,\xi)-\nabla f(\cdot,\xi')\|_{L^{\infty}(Y)}
  &\lesssim
  \|\nabla f(\cdot,\xi)-\nabla f(\cdot,\xi')\|_{L^{2}(2Y)} +
  \|b(y,\xi)-b(y,\xi')\|_{L^{q}(2Y)}\\
  &\lesssim \|b(y,\xi)-b(y,\xi')\|_{L^{q}(Y)}\\
  &\lesssim |\xi-\xi'|
\end{aligned}
\end{equation*}
with $q>d$, where we employ the energy
estimate and H\"older's inequality in the second inequality,
and the estimate $\eqref{f:2.3}$ in the last one. Hence,
by the definition of $E_{ij}$, we obtain that
\begin{equation*}
  |E(y,\xi)-E(y,\xi')|\lesssim|\nabla f(\cdot,\xi)-\nabla
  f(\cdot,\xi')|\lesssim|\xi-\xi'|
  \text{~~~for~any~}y,\xi ,\xi'\in \mathbb{R}^{d},
\end{equation*}
which means that $|\nabla_{\xi} E(y,\xi)|\leq C$ for
any $y,\xi\in \mathbb{R}^{d}$, and we have completed the
whole proof.
\qed

\noindent\textbf{The proof of Lemma $\ref{interpolation}$.}
The end-point cases $s=0,1$ has been included in
the process of the proof, while we focus ourselves on
$0<s<1$ in the later proof. The idea relies on
interpolation and duality arguments.
Define
$T_\varepsilon(f) := \varpi(\cdot/\varepsilon)f$ on $\mathbb{R}^d$, and it is not hard to see that
\begin{equation*}
 \|T_\varepsilon (f)\|_{L^2(\Omega)}
 \leq \|\varpi\|_{L^\infty(Y)}\|f\|_{L^2(\Omega)}
 \leq C\|\varpi\|_{W^{1,p}(Y)}\|f\|_{L^2(\Omega)}.
\end{equation*}
This implies $\|T_\varepsilon\|_{L^2\to L^2}\leq C\|\varpi\|_{W^{1,p}(Y)}$ (this in fact proved the result in the case of $s=0$).
Then we obtain
\begin{equation*}
\begin{aligned}
 \|\nabla T_\varepsilon (f)\|_{L^2(\Omega)}
& \leq \varepsilon^{-1}\|\nabla\varpi(\cdot/\varepsilon)f\|_{L^2(\Omega)}
 + \|\varpi(\cdot/\varepsilon)\nabla f\|_{L^2(\Omega)}\\
& \leq \varepsilon^{-1}\|\nabla \varpi(\cdot/\varepsilon)\|_{L^{2_*}(\Omega)}
\|f\|_{L^{2^*}(\Omega)}
+ \|\varpi\|_{L^\infty(Y)}\|\nabla f\|_{L^2(\Omega)}\\
& \leq C\varepsilon^{-1}\Big\{\|\nabla \varpi\|_{L^{2_*}(Y)}
+ \|\varpi\|_{L^\infty(Y)}\Big\}\|\nabla f\|_{L^2(\Omega)}\\
& \leq C\varepsilon^{-1}\|\varpi\|_{W^{1,p}(Y)}\|\nabla f\|_{L^2(\Omega)},
\end{aligned}
\end{equation*}
where $2^*=2d/(d-2)$ and $2_*=d$ if $d>2$;
$2^*=2p/(p-2)$ and $2_*= p$ if $d=2$. Here we merely employ Sobolev's inequality in the last two steps. The above computations lead to
$\|T_\varepsilon\|_{H^1\to H^1}\leq C\varepsilon^{-1}\|\varpi\|_{W^{1,p}(Y)}$ (which has
proved the result for $s=1$).

Thus, on account of the complex interpolation inequality (see for example \cite[Theorem 2.6]{L}),  we have
\begin{equation*}
 \|T_\varepsilon\|_{H^{s}\to H^{s}}
 \leq C\varepsilon^{-s}\|\varpi\|_{W^{1,p}(Y)},
\end{equation*}
where we note that
$H^{s}(\Omega)=[L^2(\Omega),
H^{1}(\Omega)]_{s}$ with $s\in(0,1)$ (see for example \cite[Proposition 2.17]{JK}). This gives the estimate
$\eqref{pri:2.8}$. Consequently,
the desired estimate $\eqref{pri:2.9}$ follows from a duality argument,
\begin{equation*}
\int_{\Omega} \varpi(x/\varepsilon)f\zeta dx
\leq \|f\|_{H^{-s}(\Omega)}
\|\varpi(\cdot/\varepsilon)\zeta\|_{H^s(\Omega)}
\lesssim \varepsilon^{-s}\|\varpi\|_{W^{1,p}(Y)}
\|f\|_{H^{-s}(\Omega)}\|\zeta\|_{H^{s}(\Omega)}
\quad \forall \zeta\in C_0^{\infty}(\Omega),
\end{equation*}
and we have completed the whole proof.
\qed

\noindent\textbf{The proof of Lemma $\ref{extensiontheory}$.}
Since the stated estimates $\eqref{pri:2.23}$ had
been shown in \cite[Theorem 4.3]{OSY}, we focus on
the estimate $\eqref{pri:2.22}$. Before proceeding further,
it is better to outline the core ideas included in \cite[Theorem 4.3]{OSY}, and we take their
terminology like ``perforated domains of type I, II'' (whose definition can be found in \cite[pp.42-44]{OSY}).
Roughly speaking, there are three steps to complete the whole arguments. (1) Due to $u=0$ on $\Gamma_\varepsilon$, it is possible to transfer the perforated domain from type I to type II.
(2) After transferred to the perforated of type II, there are
finite cases that holes intersect with the given torus (its
scale is comparable to that of the holes). (3) Focus on one cell (produced by
torus and periodic holes), and the construction of
the extension map is reduced to build elliptic partial differential equations with mixed boundary value problems, as well as, regularity estimates
(see \cite[Lemma 4.1]{OSY}).
By this way, the estimated constant is consequently independent of $\varepsilon$.

Hence, to establish the desired estimate $\eqref{pri:2.22}$ we need
to improve some estimates addressed in Step (3) above.
Compared to the proof in \cite[Lemma 4.1]{OSY}, the core difference here is that we employ Shen's real approach (see Lemma $\ref{lemma:6.1}$) to obtain $W^{1,p}$ estimates for
the auxiliary equation $\eqref{pri:2.27}$ on Lipschitz domains.
At first, we introduce the same notations as in \cite[Lemma 4.1]{OSY}. Let $G\subset \mathcal{D}\subset \mathbb{R}^d$ and each of the sets $G,\mathcal{D},\mathcal{D}\backslash\overline{G}$ be non-empty bounded Lipschitz domain. Suppose that $\partial G\cap\mathcal{D}$ is non-empty. Denote $W$ as the weak solution of the following equation:
  \begin{equation}\label{pri:2.27}
  \left\{\begin{aligned}
  \nabla\cdot\overline{a}\nabla W&=0&\qquad&\text{in}\quad G,\\
   \vec{n}\cdot\overline{a}\nabla W&=0&\qquad&\text{on}\quad \partial G\cap\partial \mathcal{D},\\
   W&=u&\qquad&\text{on}\quad \partial G\cap\mathcal{D},
  \end{aligned}\right.
  \end{equation}
  where $u\in W^{1,p}(\mathcal{D}\backslash G)$ and $\overline{a}$ is an arbitrary constant matrix satisfying $
  \mu_0|\xi|^2\leq \xi\cdot\overline{a}\xi
  \leq \mu_1|\xi|^2$ for any $\xi\in\mathbb{R}^d$. Then, one may
  set
\begin{equation}
  P(u)=\left\{
\begin{aligned}
 &u(x)&\qquad&\text{for}\quad x\in\mathcal{D}\backslash G,\\
 &W(x)&\qquad&\text{for}\quad x\in G.
\end{aligned}\right.
  \end{equation}
To give the first line of \eqref{pri:2.22}, it's not hard to see that it suffices to show the following estimate
\begin{equation}\label{pri:2.29}
\|P(u)\|_{W^{1,p}(\mathcal{D})}\lesssim \|u\|_{W^{1,p}(\mathcal{D}\backslash G)}.
\end{equation}
According to the classic linear extension theorem
for Sobolev spaces, we may have $\tilde{u}\in W^{1,p}(\mathcal{D})$ from $u\in W^{1,p}(\mathcal{D}\backslash G)$, and
$\|\tilde{u}\|_{W^{1,p}(\mathcal{D})}\leq C \|u\|_{W^{1,p}(\mathcal{D}\backslash G)}$ in which the constant $C$ is independent of $u$. By setting $\tilde{W}=W-\tilde{u}$, we rewrite \eqref{pri:2.27} as follows:
\begin{equation}\label{pri:2.30}
\left\{\begin{aligned}
  \nabla\cdot\overline{a}\nabla \tilde{W}&=-\nabla\cdot\overline{a}\nabla\tilde{u}&\qquad&\text{in}\quad G,\\
   \vec{n}\cdot\overline{a}\nabla \tilde{W}&=-\vec{n}\cdot\overline{a}\nabla \tilde{u}&\qquad&\text{on}\quad \partial G\cap\partial \mathcal{D},\\
   \tilde{W}&=0&\qquad&\text{on}\quad \partial G\cap\mathcal{D}.
  \end{aligned}\right.
\end{equation}
Thus, the remainder of the proof is to establish $W^{1,p}$ estimates
for $\eqref{pri:2.30}$, and we will close it by two steps.

\textbf{Step 1.}
For $\bar{p}=\frac{2d}{d-1}$, we claim that the reverse H\"{o}lder's inequality
\begin{equation}\label{pri:2.25}
\Big(\dashint_{G_{r}}|\nabla \phi|^{\bar{p}}\Big)^{\frac{1}{\bar{p}}}
\lesssim\Big(\dashint_{G_{2r}}|\nabla \phi|^2\Big)^{\frac{1}{2}}
\end{equation}
holds for $\phi$ that satisfies
$\nabla\cdot\bar{a}\nabla \phi=0$ in $G_{4r}$ with $\vec{n}\cdot\overline{a}\nabla \phi=0$ on $G_{4r}^{N}$ and $\phi=0$ on $G^D_{4r}$, in which $G_{r}=G\cap B(x_0,r)$ with
$x_0\in\overline{G},~
G_{4r}^{N}=\partial G_{4r}\cap\Big(\partial G\cap\partial \mathcal{D}\Big)$
and $G_{4r}^{D}=\partial G_{4r}\cap\Big(\partial G\cap\mathcal{D}\Big)$.
Based on the Sobolev embedding theorem and duality arguments (see for example \cite[Remark 9.3]{KFS1}), for $t\in (1,2)$, it follows that
\begin{equation}\label{pri:2.24}
\begin{aligned}
\Big(\int_{G_{tr}}|\nabla \phi|^{\bar{p}}dx\Big)^{1/\bar{p}}
&\lesssim \Big(\int_{\partial G_{tr}}|(\nabla \phi)^*|^{2}dx\Big)^{1/2}\\
&\lesssim \Big(\int_{\partial G_{tr}\backslash\partial G}|\nabla \phi|^{2}dS\Big)^{1/2}+
\Big(\int_{\partial G_{tr}\cap\partial G\cap\mathcal{D}}
|\nabla_{\text{tan}}\phi|^{2}dS\Big)^{1/2}\\
&\quad+
\Big(\int_{\partial G_{tr}\cap\partial G\cap\partial \mathcal{D}}
\Big|\frac{\partial \phi}{\partial \nu}\Big|^{2}dS\Big)^{1/2}\\
&\lesssim
\Big(\int_{\partial G_{tr}\backslash\partial G}|\nabla \phi|^{2}dS\Big)^{1/2},
\end{aligned}
\end{equation}
in which the notation $(\nabla\phi)^*$ represents its
the nontangential maximal function of $|\nabla\phi|$
(see \cite[pp.1220]{B} for the definition).
Here we employ the Rellich's estimate \cite[Theorem 1.5]{B} in the second inequality and the last inequality is due to the assumption on the boundary data. Then, squaring and integrating on both sides of \eqref{pri:2.24}
with respect to $t\in(1,2)$, it follows that
\begin{equation}
\Big(\int_{G_{r}}|\nabla \phi|^{\bar{p}}dx\Big)^{2/\bar{p}}\lesssim
\frac{1}{r}\int_{G_{2r}}|\nabla \phi|^2dx,
\end{equation}
and this implies \eqref{pri:2.25}.

Consequently, a self-improvement property implies that there exists a small parameter $\epsilon>0$, depending on $\mu_{0},\mu_1,d$ and the character of $\Omega$, such that the estimate \eqref{pri:2.25} is still true for the new index $\bar{p}^+:=\frac{2d}{d-1}+\epsilon$.

\textbf{Step 2.} In view of real methods (see Lemma $\ref{lemma:6.1}$),
one may have the $W^{1,p}$ estimates for $2\leq p<\bar{p}^+$, i.e.,
\begin{equation}
\Big(\int_{G}|\nabla \tilde{W}|^pdx\Big)^{1/p}\lesssim
\Big(\int_{G}|\nabla\tilde{u}|^pdx\Big)^{1/p},
\end{equation}
provided $\tilde{W}$ being associated with $\tilde{u}$ by $\eqref{pri:2.30}$.
Now, we decompose equation \eqref{pri:2.30} as follows:
\begin{equation*}
(\text{i})\left\{\begin{aligned}
\nabla\cdot\overline{a}\nabla v
&=-\nabla\cdot(I_{B}\overline{a}\nabla\tilde{u})&\quad&\text{in}\quad G,\\
\vec{n}\cdot\overline{a}\nabla v&=-\vec{n}\cdot(I_{B}\overline{a}\nabla \tilde{u})&\quad&\text{on}\quad \partial G\cap\partial \mathcal{D},\\
   v&=0&\quad&\text{on}\quad \partial G\cap\mathcal{D},
\end{aligned}\right.
\quad
(\text{ii})\left\{\begin{aligned}
\nabla\cdot\overline{a}\nabla w
&=-\nabla\cdot[(1-I_{B})\overline{a}\nabla\tilde{u}]&\quad&\text{in}\quad G,\\
\vec{n}\cdot\overline{a}\nabla w&=-\vec{n}\cdot[(1-I_{B})\overline{a}\nabla\tilde{u}]&\quad&\text{on}\quad \partial G\cap\partial \mathcal{D},\\
   w&=0&\quad&\text{on}\quad \partial G\cap\mathcal{D},
\end{aligned}\right.
\end{equation*}
in which $B:=B(x,r)$ with $r>0$ and $x\in\overline{G}$ are arbitrary. Due to the linearity of the operator, one may easily have
$\tilde{W}=v+w$. For the first equation $(\text{i})$ above, it follows from energy estimate that
\begin{equation}\label{pri:2.36}
\Big(\dashint_{\frac{1}{2}B\cap G}|\nabla v|^2 \Big)^{1/2}
\lesssim\Big(\dashint_{B\cap G}|\nabla \tilde{u}|^2 \Big)^{1/2}.
\end{equation}
On the other hand, we may employ Step 1 for $w$:
\begin{equation}\label{pri:2.37}
\begin{aligned}
\Big(\dashint_{\frac{1}{4}B\cap G}|\nabla w|^{\bar{p}^+}\Big)^{1/\bar{p}^+}
&\lesssim^{\eqref{pri:2.25}}\Big(\dashint_{\frac{1}{2}B\cap G}|\nabla w|^{2}\Big)^{1/2}\\
&\lesssim\Big(\dashint_{\frac{1}{2}B\cap G}|\nabla \tilde{W}|^{2}\Big)^{1/2}+
\Big(\dashint_{\frac{1}{2}B\cap G}|\nabla v|^{2}\Big)^{1/2}\\
&\lesssim^{\eqref{pri:2.36}}\Big(\dashint_{\frac{1}{2}B\cap G}|\nabla \tilde{W}|^{2}\Big)^{1/2}+
\Big(\dashint_{B\cap G}|\nabla \tilde{u}|^{2}\Big)^{1/2}.
\end{aligned}
\end{equation}
 For any $2\leq p<\bar{p}^+$, combing estimates \eqref{pri:2.36} and \eqref{pri:2.37} with Lemma \ref{lemma:6.1}, one may get
\begin{equation}\label{pri:2.38}
\Big(\int_{G}|\nabla \tilde{W}|^{p}\Big)^{1/p}\lesssim
\Big(\int_{G}|\nabla \tilde{W}|^{2}\Big)^{1/2}+\Big(\int_{G}|\nabla \tilde{u}|^{p}\Big)^{1/p}\lesssim \Big(\int_{G}|\nabla \tilde{u}|^{p}\Big)^{1/p},
\end{equation}
where the last step comes from energy estimate and H\"{o}lder's inequality. The desired estimate \eqref{pri:2.29} (for the case $p\geq 2$) follows from \eqref{pri:2.38} and the fact that
$\|\tilde{u}\|_{W^{1,p}(\mathcal{D})}\leq C \|u\|_{W^{1,p}(\mathcal{D}\backslash G)}$ and $\tilde{W}=W-\tilde{u}$.
Then, by duality arguments one may derive the case
of $\frac{2d}{d+1}-\epsilon<p<2$, and we left it to the reader.

We proceed to study the second line of $\eqref{pri:2.22}$.
In fact, $P(c) = c$ for any $c\in\mathbb{R}^d$. Therefore,
\begin{equation}\label{pri:2.19}
\|\nabla P(u)\|_{L^p(\mathcal{D})}
\lesssim \|\nabla P(u-c)\|_{L^p(\mathcal{D})}
\lesssim \|u-c\|_{W^{1,p}(\mathcal{D}\setminus G)}
\lesssim \|\nabla u\|_{L^{p}(\mathcal{D}\setminus G)},
\end{equation}
where we prefer $c=\dashint_{\mathcal{D}\setminus G} u$.
Then we define the extension operator $P_\varepsilon$ by a rescaling argument and
the proof is complete.
\qed

\noindent\textbf{The proof of Lemma $\ref{lemma:2.20}$.}
Roughly, we may separate two cases to talk about the proof.
(1). $r\geq \varepsilon$; (2). $0<r<\varepsilon$.
In fact, in the second case, it is known that
$B(x,r)$ will at most intersect with the finite number of the holes $\{\varepsilon H_k\}_{k=1}^\infty$ according to
the separated assumption $\eqref{g}$. Moreover, the region
$B^{\varepsilon}(x,r)$ is a
bounded connected Lipschitz domain. To avoid losing the
control of the Lipschitz constant of that region, one may
choose $Q$ such that $B(x,r)\subset Q\subset B(x,3r)$ to make sure that
$Q\cap(\varepsilon\omega)$ own a better Lipschitz constant
of the boundary. Then one may appeal to the classical Sobolev-Poincar\'e's inequality
(see for example \cite[Theorem 3.27]{GT}) on this region, and
\begin{equation*}
\|w-c_r\|_{L^q(B^{\varepsilon}_{r}(x))}
\leq \|w-c_r\|_{L^q(Q\cap(\varepsilon\omega)}
\lesssim \|\nabla w\|_{L^p(Q\cap(\varepsilon\omega))}
\leq
\|\nabla w\|_{L^p(B^{\varepsilon}_{3r}(x))}
\end{equation*}
where one may choose $c_r=\dashint_{Q\cap(\varepsilon\omega)} w$, and the up to constant is in dependent of $x$, $r$ and $\varepsilon$. The above estimate gives the desired estimate
$\eqref{pri:2.20}$ in the case of $0<r<\varepsilon$.

Now, we proceed to handle the interesting case $r\geq\varepsilon$.
Let $Y$ be the unite cell, and we define the index set and the related cover region as follows
\begin{equation}\label{}
 T_\varepsilon:=\big\{z\in\mathbb{Z}^d:\varepsilon(z+Y)\cap B(x,r)\not=\emptyset\big\};
 \qquad Y^{*}_{B(x,r)}:= \bigcup_{z\in T_\varepsilon} \varepsilon(z+Y).
\end{equation}
It is not hard to see that $T_\varepsilon\not=\emptyset$ due to $r\geq\varepsilon$. Then, we have two important observations: (1) $B(x,r)\subset Y_{B(x,r)}^*\subset B(x,3r)$; (2) the region $Y_{B(x,r)}^*\cap(\varepsilon\omega)$ is the so-called perforated domains of type II (whose definition can be found in \cite[pp.42-44]{OSY}). Thus, one may appeal to the results
of Lemma $\ref{extensiontheory}$ directly (here we even release from the operation of transferring the perforated domains of type I to type II). In this regard, we may denote
the extension of $w$ by $\tilde{w}$ in the sense of Lemma
$\ref{extensiontheory}$. Thus, we have the following computations.
\begin{equation}\label{f:2.20}
\begin{aligned}
\|w-c\|_{L^q(B^{\varepsilon}_r(x))}
&\leq \|\tilde{w}-c\|_{L^q(B_r(x))}\\
&\lesssim \|\nabla \tilde{w}\|_{L^p(B_r(x))}
\lesssim  \|\nabla w\|_{L^p(B^{\varepsilon}_r(x))}
+
\Big(\int_{B_r(x)\setminus(\varepsilon\omega)}|\nabla \tilde{w}|^p dy\Big)^{1/p},
\end{aligned}
\end{equation}
where we note that $\tilde{w} = w$ on $B^\varepsilon(x,r)$,  and the main job is to estimate the last term above. In fact,
\begin{equation*}
\begin{aligned}
\int_{B_r(x)\setminus(\varepsilon\omega)}|\nabla \tilde{w}|^p dy
\leq \sum_{z\in T_\varepsilon}
\int_{\varepsilon(z+Y)\setminus(\varepsilon\omega)}
|\nabla \tilde{w}|^p dy
\lesssim^{\eqref{pri:2.19}} \sum_{z\in T_\varepsilon}
\int_{\varepsilon(z+Y)\cap(\varepsilon\omega)}
|\nabla w|^p dy
\end{aligned}
\end{equation*}
in which we emphasis that the estimated constant of  $\eqref{pri:2.19}$ is
independent of $w$ and the location (due to the periodicity).
Therefore, we continue to compute the above inequalities, and
\begin{equation}\label{f:2.21}
\int_{B_r(x)\setminus(\varepsilon\omega)}|\nabla \tilde{w}|^p dy
\lesssim
\int_{Y_{B_r(x)}^{*}\cap(\varepsilon\omega)}
|\nabla w|^p dy
\leq \int_{B_{3r}(x)\cap(\varepsilon\omega)}
|\nabla w|^p dy
\end{equation}
Consequently, inserting the estimate $\eqref{f:2.21}$
back into $\eqref{f:2.20}$ we obtain
\begin{equation*}
\|w-c\|_{L^q(B^{\varepsilon}(x,r))}
\lesssim  \|\nabla w\|_{L^p(B^{\varepsilon}(x,r))}
+ \|\nabla w\|_{L^p(B^{\varepsilon}(x,3r))}
\lesssim \|\nabla w\|_{L^p(B^{\varepsilon}(x,3r))}
\end{equation*}
and this closes the proof of $\eqref{pri:2.20}$.
By the same token, one may derive
the estimate $\eqref{pri:2.21}$ without any real difficulty, and left these details to the reader. We have completed
the proof.
\qed

\section{Appendix}\label{sec:7}

\subsection{Fundamental regularities of weak solutions
to homogenization problems}

\begin{lemma}[interior Caccioppoli's inequality]\label{lemma:2.3}
Assume that $\mathcal{L}_\varepsilon$ satisfies the conditions
$\eqref{a:1}$, $\eqref{a:2}$.
Let $u_\varepsilon\in H^1(B^{\varepsilon}_2)$ be a weak solution of $\eqref{pde:7.3}$.
Then for
any $c\in\mathbb{R}$ and $0<r\leq1$, we have
\begin{equation}\label{pri:2.14}
\int_{B^{\varepsilon}_{r}}|\nabla u_\varepsilon|^2 dx
\leq \frac{C}{r^2}\inf_{c\in\mathbb{R}}\int_{B^{\varepsilon}_{2r}}
|u_\varepsilon - c|^2 dx,
\end{equation}
where $C$ depends on $\mu_0,\mu_1$ and $d$.
\end{lemma}

\begin{proof}
It's a classical result and we provide a proof for the reader's convenience.
For $0<r\leq1$, by the definition of the weak solution,
there holds
\begin{equation}\label{pde:2.4}
\int_{B^{\varepsilon}_{2r}}A(x/\varepsilon,\nabla u_\varepsilon)\cdot\nabla\phi dx
= 0
\end{equation}
for any $\phi\in H^1(B^{\varepsilon}_{2r},\partial
B^\varepsilon_{2r}|_{\partial B_{2r}})$ (see Subsection $\ref{subsec:1.2}$).
Set $\phi = \psi_{r}^2(u_\varepsilon - c)$ with any $c\in\mathbb{R}$, where
$\psi_{r}\in C_0^1(B_{2r})$ is a cut-off function, satisfying
$\psi_r = 1$ in $B_r$ and $\psi_r = 0$ outside $B_{2r}$ with
$|\nabla\psi_r|\lesssim 1/r$. The stated estimate $\eqref{pri:2.14}$ follows
from the assumptions $\eqref{a:1}$, $\eqref{a:2}$, as well as, Young's inequality.
\end{proof}
\begin{lemma}[boundary Caccioppoli's inequality]\label{lemma:5.1}
Suppose that the coefficient A satisfies \eqref{a:1} and
$\eqref{a:2}$. Let $u_{\varepsilon}\in H^{1}(D_{4}^{\varepsilon})$ be the weak solution to
$\eqref{pde:7.4}$.
Then, for any $0<r\leq 1$, one may have
\begin{equation}\label{pri:5.1}
 \Big(\dashint_{D^{\varepsilon}_{r}} |\nabla u_\varepsilon|^2  \Big)^{1/2}
 \lesssim \frac{1}{r} \Big(\dashint_{D^{\varepsilon}_{2r}}|u_\varepsilon|^2 \Big)^{1/2}.
\end{equation}
\end{lemma}
\begin{proof}
The proof is standard and similar to Lemma \ref{lemma:2.3}, and we do not repeat it here.
\end{proof}

\begin{theorem}[self-improvement properties]\label{thm:7.1}
Let $\omega$ satisfy the separated property.
Suppose that the coefficient A satisfies \eqref{a:1} and
$\eqref{a:2}$. Let $u_\varepsilon$ satisfy the equation
$\eqref{pde:7.3}$. Then there exists $0<p-2\ll 1$, depending on $\mu_0,\mu_1$ and $d$, such that
\begin{equation}\label{pri:7.6}
\Big(\dashint_{B^{\varepsilon}_r}
|\nabla u_\varepsilon|^p\Big)^{1/p}
\lesssim
\Big(\dashint_{B^{\varepsilon}_{6r}}
|\nabla u_\varepsilon|^2\Big)^{1/2}
\end{equation}
for $\varepsilon\leq r\leq 1/3$. Moreover,
let $\Omega$ be a Lipschitz domain. If
$u_\varepsilon\in H^1(D^\varepsilon_4,\Delta^\varepsilon_4)$
is a weak solution to $\eqref{pde:7.4}$. Then one similarly obtains
\begin{equation}\label{pri:7.7}
\Big(\dashint_{D^{\varepsilon}_{r}}
|\nabla u_\varepsilon|^p\Big)^{1/p}
\lesssim
\Big(\dashint_{D^{\varepsilon}_{2r}}
|\nabla u_\varepsilon|^2\Big)^{1/2},
\end{equation}
where the up to constant and $p$ additionally depends on
the boundary character of $\Omega$.
\end{theorem}

\begin{proof}
The main idea of the proof is
based upon Caccioppoli's inequalities and reverse H\"older's inequalities. Since we require the results to be established on perforated domains, we appeal to Lemma $\ref{lemma:2.20}$
to make the whole arguments workable. We merely describe the
proof of the estimate $\eqref{pri:7.6}$ for the reader's convenience. For any $0<r\leq 1$, it follows from the estimate $\eqref{pri:2.14}$ that
\begin{equation}
\begin{aligned}
\Big(\dashint_{B^{\varepsilon}_r}|\nabla u_\varepsilon|^2\Big)^{\frac{1}{2}}
&\lesssim \frac{1}{r}
\Big(\dashint_{B^{\varepsilon}_{2r}}
|u_\varepsilon - c_r|^2\Big)^{\frac{1}{2}}\\
&\leq
\frac{1}{r}
\Big(\dashint_{B^{\varepsilon}_{2r}}
|u_\varepsilon - c_r|^{\frac{2d}{d-1}}\Big)^{\frac{d-1}{2d}}
\lesssim^{\eqref{pri:2.20}}
\Big(\dashint_{B^{\varepsilon}_{6r}}|\nabla u_\varepsilon|^{\frac{2d}{d+1}}\Big)^{\frac{d+1}{2d}}.
\end{aligned}
\end{equation}
Let $f=|\nabla u_\varepsilon|^{\frac{2d}{d+1}}$, and we rewrite the above estimate as
\begin{equation*}
 \Big(\dashint_{B^{\varepsilon}_r}
 f^{\frac{d+1}{d}}\Big)^{\frac{d}{d+1}}
 \lesssim \dashint_{B^{\varepsilon}_{6r}} f,
\end{equation*}
Then on account of reverse H\"older's inequality
(see for example \cite[Theorem 6.38]{MGLM}) there exists
some $0<\epsilon\ll 1$, depending on $\mu_0,\mu_1,d$, such that for $\frac{d+1}{d}< s\leq \frac{d+1}{d}+\epsilon$, it holds that
\begin{equation*}
 \Big(\dashint_{B^{\varepsilon}_{r}}
 f^{s}\Big)^{\frac{1}{s}}
 \lesssim
 \Big(\dashint_{B^{\varepsilon}_{6r}} f^{\frac{d+1}{d}}
 \Big)^{\frac{d}{d+1}}.
\end{equation*}
By setting $p=\frac{2ds}{d+1}$, one may derive the stated estimate $\eqref{pri:7.6}$, while the estimate $\eqref{pri:7.7}$ follows from the same ingredients and
we left it to the reader. The proof is complete.
\end{proof}

\begin{theorem}[$H^1$ theory]\label{thm:2.1}
Let $\Omega$ be a bounded Lipschitz domain.
Assume that $\mathcal{L}_\varepsilon$ satisfies the conditions
$\eqref{a:1}$, $\eqref{a:2}$. Let
$u_\varepsilon\in H^1(\Omega_{\varepsilon})$ be the solution of $\eqref{pde:1.1}$ with $F\in H^{-1}(\Omega)$. Then we have
\begin{equation}\label{pri:2.12}
 \|\nabla u_\varepsilon\|_{L^2(\Omega_{\varepsilon})}
 \leq C\Big\{\|F\|_{H^{-1}(\Omega)} + \|g\|_{H^{1/2}(\partial\Omega)}\Big\},
\end{equation}
where $C$ depends on $\mu_0,\mu_1,d$ and the character of $\Omega$. Moreover, if $\Omega=\mathbb{R}^d$ and
$u_\varepsilon\in H^1(\Omega_\varepsilon)$ satisfies the regular problem:
$\lambda u_\varepsilon +\mathcal{L}_\varepsilon(u_\varepsilon) = \nabla\cdot f$ in $\Omega_\varepsilon$ and
$\sigma_\varepsilon(u_\varepsilon) = 0$ on
$\partial\Omega_\varepsilon$, where
$\lambda\in(0,\mu_0)$ and $f\in L^2(\mathbb{R}^d;\mathbb{R}^d)$, then there holds
\begin{equation}\label{pri:7.4}
\sqrt{\lambda}\|u_\varepsilon\|_{L^2(\Omega_\varepsilon)}
+ \|\nabla u_\varepsilon\|_{L^2(\Omega_\varepsilon)}
\lesssim \|f\|_{L^{2}(\mathbb{R}^d)},
\end{equation}
where the up to constant is independent of $\lambda$.
\end{theorem}

\subsection{Fundamental regularities of weak solutions
to effective problems}

\begin{theorem}[Meyer's estimates]\label{thm:2.20}
Let $\Omega$ be a bounded Lipschitz domain.
Given $f\in L^{p}(\Omega;\mathbb{R}^d)$ for some $0<p-2\ll1$ and $g\in W^{1-1/p,p}(\partial\Omega)$, let
$u\in H^1(\Omega)$ is the weak solution of
$\mathcal{L}_0(u_0) = \nabla\cdot f$ in $\Omega$ with
$u_0 = g$ on $\partial\Omega$. Then there holds
\begin{equation}\label{pri:2.13}
  \|\nabla u_0\|_{L^{p}(\Omega)}\leq C_{p}
  \Big\{\|f\|_{L^{p}(\Omega)}+
  \|g\|_{W^{1-1/p,p}(\partial\Omega)}\Big\},
\end{equation}
where the constant $C_{p}$ is dependent on $\mu_{0},\mu_{1},d,p$ and the character of $\Omega$.
\end{theorem}
\begin{proof}
  The main idea of the proof is based on reverse
  H\"older's inequality (see for example
  \cite[Theorem 6.38]{MGLM}), and the related details may be found
  in \cite[Theorem 2.13]{WXZ}.
\end{proof}

\begin{remark}
\emph{To obtain the higher regularities of $\nabla u_0$,
it relies on the
smoothness of $\widehat{A}$. We emphasis that
$\widehat{A}$ is merely proved to be Lipschitz continuous.
Thus the later results heavily relies on De Giorgi-Nash-Moser theorem for linearized equations, and therefore we only dare
to say there exists $\alpha\in(0,1)$ such that $u\in C^{1,\alpha}(\bar{\Omega})$ under suitable boundary conditions. Usually, $\alpha$ would be very small and
there is no hope to improve this result unless we master more
information on regularities of $\widehat{A}$, which is, of course, a very interesting problem in nonlinear homogenization theories.}
\end{remark}

\begin{theorem}[$H^2$ theory]\label{thm:2.3}
Let $\Omega$ be a bounded $C^{1,1}$ domain.
Given $g\in H^{3/2}(\partial\Omega)$ and $F\in L^2(\Omega)$,
assume that
$u_0\in H^1(\Omega)$ is the weak solution of
$\mathcal{L}_0(u_0) = F$ in $\Omega$ with $u_0 = g$ on $\partial\Omega$.
Then we have $u_0\in H^2(\Omega)$ satisfying
\begin{equation}\label{pri:2.11}
 \|\nabla^2 u_0\|_{L^2(\Omega)}
 \leq C\Big\{\|F\|_{L^2(\Omega)}
 +\|g\|_{H^{3/2}(\partial\Omega)}\Big\},
\end{equation}
where $C$ depends on $\mu_0,\mu_1,d$ and the character of $\Omega$. Moreover, if $\Omega=\mathbb{R}^d$ and
$u_0$ satisfies the regular problem:
$\lambda u_0 +\mathcal{L}_0(u_0) = F$ in $\mathbb{R}^d$ with
$\lambda\in(0,\mu_0)$, then there holds
\begin{equation}\label{pri:7.3}
\sqrt{\lambda}\|\nabla u_0\|_{L^2(\mathbb{R}^d)}
+ \|\nabla^2 u_0\|_{L^2(\mathbb{R}^d)}
\lesssim \|F\|_{L^2(\mathbb{R}^d)},
\end{equation}
where the up to constant is independent of $\lambda$.
\end{theorem}
\begin{proof}
The main idea is linearization of the equations, coupled with
straightening the boundary arguments, where
we pointed out that the map of the local changing coordinates to flatten out the boundary does not change the type of the operator classes (see for example \cite[Theorem 2.16]{WXZ}).
\end{proof}
\begin{theorem}[interior $C^{1,\alpha}$ estimates]\label{lemma:4.2}
Given $F\in L^{p}(\Omega)$ for some $p>d$,
let $u_0\in H^1(B_{2r})$ be a solution of
$\mathcal{L}_0(u_0) = F$ in $B_{2r}$.
Then
there exists $\alpha\in(0,1)$, and a constant $C>0$
depending on $\mu_0,\mu_1, p, d$, such that
\begin{equation}\label{pri:4.2}
 [\nabla u_0]_{C^{0,\alpha}(B_{r/2})}
 \leq C r^{-\alpha}\bigg\{\frac{1}{r}\Big(\dashint_{B_r}|u_0|^2\Big)^{1/2}
 + r\Big(\dashint_{B_r}|F|^p\Big)^{1/p}\bigg\}.
\end{equation}
\end{theorem}

\begin{proof}
The main idea is linearization.
It is fine to assume $u_0\in H^2(B(0,r))$ and we have the following equation
\begin{equation}\label{pde:4.1}
\int_{B(0,r)}\nabla_{\xi_j}\widehat{A}^i(\nabla u_0)\nabla^2_{jk}u_0\nabla_i\phi dx =
-\int_{B(0,r)} F \nabla_k\phi dx
\end{equation}
for any $\phi\in H_0^1(B(0,r))$, and $k=1,\cdots,d$.
Let $\tilde{a}_{ij}(x) = \nabla_{\xi_j}\widehat{A}^i(\nabla u_0)$, which
will give a linear operator with the uniform ellipticity on account of
$\eqref{pri:2.3}$ and $\eqref{a:4}$. Hence, the De Giorgi-Nash-Moser theorem
tells us that for any $p>d$, there exists $\alpha\in(0,1)$ and $C>1$, depending only on
$\mu_0,\mu_2, d$ and $p$, such that
\begin{equation}
 [\nabla u_0]_{C^{0,\alpha}(B(0,r/2))}
 \leq Cr^{-\alpha}\bigg\{\frac{1}{r}\Big(\dashint_{B(0,r)}|u_0|^2\Big)^{1/2}
 + r\Big(\dashint_{B(0,r)}|F|^p\Big)^{1/p}\bigg\}
\end{equation}
(see for example \cite[Theorem 8.13]{MGLM}).
\end{proof}

\begin{theorem}[boundary $C^{1,\alpha}$ estimates]\label{thm:2.19}
Let $\alpha\in (0,1)$ be obtained as in Theorem $\ref{lemma:4.2}$.
Let $\Omega$ be a bounded $C^{1,1}$ domain.
Given $g\in C^{2}(\Delta_{4r})$,
assume that
$u_{0}\in H^1(D_{4r})$ is the weak solution of
$\mathcal{L}_0(u_0) = 0$ in $D_{4r}$ and $u_0 = g$ on $\Delta_{4r}$ with $g(0) = 0$.
Then we have $\nabla u_0\in C^{0,\alpha}(D_{r}\cup\Delta_{r})$ satisfying
\begin{equation}\label{pri:7.5}
 r^{\alpha}[\nabla u_0]_{C^{0,\alpha}(\overline{D_{r}})}
 \lesssim
  \frac{1}{r}\bigg\{
  \Big(\dashint_{D_{2r}}|u_{0}|^2\Big)^{1/2}
 +r\|\nabla_{\emph{tan}}g\|_{L^{\infty}(\Delta_{2r})}\bigg\}
 +r\|\nabla\nabla_{\emph{tan}}g\|_{L^{\infty}(\Delta_{2r})},
\end{equation}
where $C$ depends on $\mu_0,\mu_1,d$.
\end{theorem}
\begin{proof}
The main idea can be found in Theorem \cite[Theorem 13.2]{GT}
and we provide a proof for the reader's convenience.
Roughly speaking, the proof includes two ingredients. The first one is the so-called flatten boundary arguments, and
then we linearize the transferred equations and appeal to
boundary H\"older estimates for linear equations. However,
to avoid the proof involving ``lower order terms'', we prefer
to flatten boundary in the second step. Although this way
includes flaw, it has already revealed the key information and techniques therein.

\textbf{Step 1.} Consider the equations on the flatten boundary region, i.e., $D_{4r}^{+} := B(0,4r)\cap\{x\in\mathbb{R}^d:x_d>0\}$ and $T_{4r} = B(0,4r)\cap\{x\in\mathbb{R}^d:x_d=0\}$,
\begin{equation}\label{pde:7.1}
\nabla\cdot\widehat{A}(\nabla u_0)=0 \quad\text{in}~D_{4r}^{+},
\qquad u_0 = g \quad \text{on}~ T_{4r}.
\end{equation}

We claim that there exists $\alpha\in(0,1)$ such that the following estimate
\begin{equation}\label{f:7.4}
 r^{\alpha}[\nabla u_0]_{C^{0,\alpha}(\overline{D_{\frac{r}{2}}^{+}})}
 \lesssim \bigg\{
 \frac{1}{r}\Big(
 \dashint_{D_{2r}^{+}}|u_0|^2\Big)^{\frac{1}{2}}
 +\frac{1}{r}\|g\|_{L^\infty(T_{2r})}
 + \|\nabla_{\text{tan}}g\|_{L^\infty(T_{2r})}
 + r\|\nabla\nabla_{\text{tan}}g\|_{L^\infty(T_{2r})}\bigg\}.
\end{equation}
By rescaling techniques one may assume $r=1$. The idea is
to linearize the equation $\eqref{pde:7.1}$ with respective
to the tangential directions, and  we obtain the linearized
equation as follows:
\begin{equation}\label{pde:7.2}
 \nabla\cdot a \nabla w^k = 0
 \quad\text{in}~~D_{4}^{+},
 \qquad w^k = \nabla_k g
 \quad \text{on}~~T_4
 \qquad \text{for~} k=1,\cdots,d-1,
\end{equation}
where $w^k = \nabla_k u_0$ and the coefficient $a=(a_{ij})$ with  $a_{ij}=\partial_{\xi_j}\widehat{A}_i(\nabla u_0)$.
From Lemma $\ref{lemma:2.4}$, it is known $a$ is a uniform elliptic coefficient.
It follows from boundary H\"older estimates (see for example
\cite[Theorem 8.29]{}) that there exists $\alpha\in(0,1)$
(which is usually very small even when the boundary data
is sufficiently smooth) such that
\begin{equation}\label{f:7.1}
 \big[w^k\big]_{C^{0,\alpha}(\overline{D_{1}^{+}})}
 \lesssim \|w^k\|_{L^2(D_{2}^{+})}
 + \|\nabla_{\text{tan}} g\|_{C^{0,\alpha}(T_2)} =: K,
\end{equation}
where we denote $\nabla_k$ on $T_{2}$ by $\nabla_{\text{tan}}$ for any $k=1,\cdots,d-1$. In fact, the estimate
$\eqref{f:7.1}$ revealed that we have controlled the H\"older seminorm for $\nabla_{ij}^{2}u_0$ and
$\nabla_{id}^{2}u_0$ with $i,j=1,\cdots,d-1$. The next job is
to show estimates for $\nabla_{dd}^2 u_0$. This time, we appeal to the equation $\eqref{pde:7.1}$, and it tells us
\begin{equation*}
\sum_{i,j=1}^{d} \partial_{\xi_j}\widehat{A}_i(\nabla u_0)\nabla_{ij}^2 u_0 = 0.
\end{equation*}
By noting $a_{ij} = \partial_{\xi_j}\widehat{A}_i(\nabla u_0)$
the above equality implies
\begin{equation}\label{f:7.2}
 -a_{dd}\nabla_{dd}^2 u_0 =
 \sum_{i=1}^{d-1}a_{id}\nabla^2_{id}u_0+ \sum_{j=1}^{d-1}a_{dj}\nabla^2_{dj}u_0
 +\sum_{i,j=1}^{d-1}a_{ij}\nabla^2_{ij}u_0
\end{equation}
(recalling Remark $\ref{remark:2.4}$ one may have $a_{dd}\geq C_1>0$). To complete the argument, let
$\eta\in C^{1}_0(B(0,2r))$ with $0<r<(1/2)$ be a cut-off function.
Then take $\eta^2 (w^k-c)$ with $c\in\mathbb{R}$ as a test function to multiply
the equation $\eqref{pde:7.2}$, and we have
\begin{equation*}
\begin{aligned}
 \int_{D_{r}^{+}}|\nabla w^k|^2dx
 &\lesssim \frac{1}{r^2}\int_{D_{2r}^{+}}|w^k-c|^2dx
 + \int_{T_{2r}}|\nabla\nabla_{\text{tan}}g| |\nabla_{\text{tan}}g-c| dS\\
 &\lesssim r^{d+2\alpha-2}[w^k]_{C^{0,\alpha}(\overline{D_{1}^{+}})}
 + r^d \|\nabla\nabla_{\text{tan}}g\|_{L^\infty(T_1)}\\
 &\lesssim^{\eqref{f:7.1}} r^{d+2\alpha-2}
 \Big\{K+\|\nabla\nabla_{\text{tan}}g\|_{L^\infty(T_2)}\Big\},
\end{aligned}
\end{equation*}
where we take $c=w^k(0)=\nabla^kg(0)$. Inserting this estimate
back into the right-hand side of $\eqref{f:7.2}$ we obtain
\begin{equation*}
\int_{D_{r}^{+}}|\nabla^2_{dd} u_0|^2dx
\lesssim r^{d+2\alpha-2}
 \Big\{K+\|\nabla\nabla_{\text{tan}}g\|_{L^\infty(T_2)}\Big\}
\end{equation*}
for any $0<r<(1/2)$. Thus,
by Morrey's estimates (see for example \cite[Theorem 7.19]{GT}) we conclude that
$[\nabla u_0]_{C^{0,\alpha}(\overline{D_{\frac{1}{2}}^{+}})}$
for any $i,j=1,\cdots,d$, and
\begin{equation}\label{f:7.3}
[\nabla u_0]_{C^{0,\alpha}(\overline{D_{1/2}^{+}})}
\lesssim \Big\{\|u_0\|_{L^2(D_2^{+})} + \|g\|_{C^{1}(T_2)}
+\|\nabla\nabla_{\text{tan}}g\|_{L^\infty(T_2)}\Big\}.
\end{equation}

The remainder of the proof in this step is appealing to
rescaling arguments. Let $u_0(x)=u_0(ry)$ where $x\in D_{4r}^{+}$ and $y\in D_{4}^{+}$. Let $u_r(y)=\frac{1}{r}u_0(ry)$ and $g_r(y):=\frac{1}{r}g(ry)$.
It is not hard to see that
\begin{equation*}
 0 = \nabla_x\cdot\widehat{A}(\nabla_xu_0)
 =\frac{1}{r}\nabla_y\cdot\widehat{A}(\frac{1}{r}\nabla_y u(ry)) = \frac{1}{r}\nabla_y\cdot\widehat{A}(\nabla_y u_r)
 \quad \Rightarrow
 \quad
 \nabla_y\cdot\widehat{A}(\nabla_y u_r) = 0 \quad\text{in}~D_{4}^{+},
\end{equation*}
and $u_r = g_r$ on $T_4$. Thus, on account of the estimate $\eqref{f:7.3}$, we in fact obtain
\begin{equation*}
[\nabla u_r]_{C^{0,\alpha}(\overline{D_{1/2}^{+}})}
\lesssim \Big\{\|u_r\|_{L^2(D_2^{+})} + \|g_r\|_{C^{1}(T_2)}
+\|\nabla\nabla_{\text{tan}}g_r\|_{L^\infty(T_2)}\Big\}.
\end{equation*}
By noting that $u_r(y) = \frac{1}{r}u_0(ry)$
and $x=ry$, the desired estimate $\eqref{f:7.4}$ simply follows
from the result by changing variables.

\textbf{Step 2.} Flatten out the boundary arguments.
Let $\Psi: D_{4r}\to D_{4}^{+}$ be a boundary flatten map,
which is a $C^{1,1}$ map and its Jacobian matrix
$\nabla\Psi$ is bounded from above and below, which guarantees
that the transferred operator satisfies the same type conditions as $\mathcal{L}_0$ does. Precisely, set
$y=\Psi(x)$ and $v(y) = u_0(\Psi^{-1}(y))$. Thus it is not
hard to obtain $\nabla_x=\nabla\Psi\nabla_y$, and therefore
\begin{equation*}
0 = \nabla_x\cdot\widehat{A}(\nabla_x u_0)
= \nabla\Psi\nabla_y\cdot\widehat{A}(\nabla\Psi\nabla_yv)
= \nabla_y\cdot \widehat{A}^{J}(\nabla_y v),
\end{equation*}
where $\widehat{A}^{J}(\cdot) = J^{t}\widehat{A}(J\cdot)$ with $J=\nabla\Psi$ and $J^{t}$ represents
the transport of $J$. It is not hard to verify that $\widehat{A}^{J}$ satisfies
the coerciveness and growth properties $\eqref{pri:2.3}$ with different character constants. Besides, $v(y) = u_0(x) = g(\Psi^{-1}y)=:\tilde{g}(y)$. Thus, we have transferred the equations into:
$\nabla\cdot \widehat{A}^{J}(\nabla v) = 0$ in $D_{4}^{+}$
with $v = \tilde{g}$ on $T_{4}$. Then apply the estimate
$\eqref{f:7.4}$ to $v$ with $\tilde{g}$ and changing variable
back we finally obtain the desired estimates $\eqref{pri:7.5}$.

We remark that as changing variable back, we will require the map $\Psi$ to be $C^{1,1}$, although this requirement
can not be observed from the most operations
in the second step. Essentially, it is because of the  linearizing of the equations, compared with the related theory for  linear equations.
\end{proof}

\subsection{Local boundary estimates on correctors}

\begin{lemma}[local boundary estimates]\label{lemma:7.1}
Let $\omega$ satisfy the separated property $\eqref{g}$.
Suppose that $A$ satisfy the conditions $\eqref{a:1}$
and $\eqref{a:2}$.
Let $u\in H^{1}_{\text{loc}}(Y\cap \omega)$ be a nonnegative solution of
$\nabla\cdot A(y, \nabla u) = 0$ in
$Y\cap\omega$ with $\vec{n}\cdot A(y, \nabla u) =0$ on $\partial\omega$. Then for any $B_r\subset B_{R}
\subset Y$ centered at
$\partial\omega$ with $0<r<R/4$, there hold the
\emph{local boundedness estimate}
\begin{equation}\label{pri:7.1}
  \sup_{y\in Y\cap \omega\cap B_{r}}u(y)\lesssim
  \Big(\dashint_{Y\cap \omega\cap B_{R}}|u|^p\Big)^{1/p}
\end{equation}
for any $p>0$,
and the weak Harnack inequality
\begin{equation}\label{pri:7.2}
  \inf_{y\in Y\cap \omega\cap B_{r}}u(y)\gtrsim
  \Big(\dashint_{Y\cap \omega\cap B_{R}} |u|^q
  \Big)^{1/q}
\end{equation}
is true for $1<q<\frac{2d}{d-2}$, where the up to constant depends only on $\mu_0,\mu_1,d,p,q$.
\end{lemma}
\begin{proof}
The main ideas had been well presented in \cite{MGLM,GT,MZ}, and
we provide a proof for the sake of the reader's convenience.
There are five steps to complete the whole arguments.

\textbf{Step 1.} We claim that if $u\in H^{1}_{loc}(Y\cap \omega)$ is a solution satisfying
\begin{equation*}
  \int_{Y\cap \omega}A(y,\nabla u)\cdot\nabla vdx=0
\end{equation*}
for any $v\in C^{1}_{0}(B_{R})$ with $B_{R}\subset\subset Y$.
Then $u^{+}=\max\{u,0\}$ is a sub-solution, which means that
\begin{equation}\label{f:5.1}
  \int_{Y\cap \omega}A(y,\nabla u^{+})\cdot\nabla v dx
  =\int_{Y\cap \omega\cap\{u>0\}}A(y,\nabla u)\cdot\nabla vdx\leq 0
\end{equation}
for any $v\geq 0$ and $v\in C^{1}_{0}(B_{R})$.
To see this,
let $v_{k}=\min\{ku^{+},1\}$.
%\begin{equation*}
%v_{k}=\min\{ku^{+},1\}=
%\left\{\begin{aligned}
%& ku^{+}, &~& 0<u^{+}\leq 1/k;\\
%& 1, &~&u^{+}\geq (1/k)
%\end{aligned}\right.
%\end{equation*}
Then for $\varphi\geq 0, \varphi\in C^{1}_{0}(B_{R})$ we have
\begin{equation*}
  0=\int_{Y\cap \omega}A(y,\nabla u)
  \cdot\nabla(\varphi v_{k})dx=\int_{Y\cap \omega}A(y,\nabla u)\cdot\nabla\varphi v_{k}dx+\int_{Y\cap \omega}A(y,\nabla u)\cdot \nabla v_{k}\varphi dx,
\end{equation*}
and this together with $\eqref{a:2}$ implies that
\begin{equation*}
\int_{Y\cap \omega}A(y,\nabla u)\cdot\nabla\varphi
v_{k}dx=-k\int_{Y\cap \omega\cap\{0<u^{+}\leq\frac{1}{k}\}}
A(y,\nabla u)\cdot\nabla u^{+}\varphi dx\leq-k\mu_{0}
\int_{Y\cap \omega\cap\{0<u^{+}\leq\frac{1}{k}\}}
|\nabla u^{+}|^{2}\varphi\leq 0.
\end{equation*}
Hence,
Let $k\to \infty$ one may obtain
\begin{equation*}
\int_{Y\cap \omega}A(y,\nabla u^{+})\cdot
\nabla\varphi dx\leq 0.
\end{equation*}

\textbf{Step 2.}
Let $B_{R}=B_{R}(x_{0})$ with $x_0\in\partial\omega$ and
$D_{R}=B_{R}\cap Y\cap \omega$. Let $\eta\in C_0^1(B_R)$
be a cutoff function such that $\eta=1$ on $B_{r}$ and $\eta=0$ on
$\mathbb{R}^{d}\setminus B_{R}$ with $|\nabla\eta|\leq C/(R-r)$.
For any $\beta \geq 0$, one may establish that
\begin{equation}\label{appendix:5.1}
 \int_{D_{R}}
  \eta^{2}|\nabla u|^{2}u^{\beta}dx
  \leq C(\mu_{0},\mu_{1},d,\beta)
  \int_{D_{R}}|\nabla \eta|^{2}u^{\beta+2}dx.
\end{equation}
To do so, it is firstly known by the assumption that $u=u^+$,
and then we set $v=\eta^{2}u_{M}^{\beta}u>0$, where
\begin{equation*}
  u_{M}=
  \left\{\begin{aligned}
  & u, &~& \text{if~}0<u<M;\\
  & M, &~& \text{if~} u\geq M.
  \end{aligned}\right.
\end{equation*}
Then plugging $v$ back into $\eqref{f:5.1}$ one may obtain
\begin{equation*}
\begin{aligned}
0& \geq\int_{D_{R}}A(y,\nabla u)\cdot\nabla(\eta^{2}u_{M}^{\beta}u)dx\\
& =\int_{D_{R}}\eta^{2} A(y,\nabla u)\cdot(\beta u^{\beta-1}_{M}u
\nabla u_{M}+u_{M}^{\beta}\nabla u)dx+2\int_{D_{R}}
\eta A(y,\nabla u)
\cdot\nabla\eta u^{\beta}_{M}udx
 :=I_{1}+I_{2}.
\end{aligned}
\end{equation*}
It follows from the condition $\eqref{a:2}$ that
\begin{equation*}
\begin{aligned}
 I_{1}
& \geq
\beta\mu_{0}\int_{D_{R}}%\cap\{0<u<M\}
  \eta^{2}|\nabla u_{M}|^{2}u_{M}^{\beta}dx
+\mu_{0}\int_{D_{R}}\eta^{2}|\nabla u|^{2}u_{M}^{\beta}dx \\
 I_{2}
& \geq -2\mu_{1}\int_{D_{R}}\eta|\nabla u|
|\nabla\eta|u_{M}^{\beta}udx
 \geq -\frac{\mu_{0}}{2}
 \int_{D_{R}}\eta^{2}|\nabla u|^{2}u_{M}^{\beta}dx
 -C(\mu_{0},\mu_{1})\int_{D_{R}}|\nabla \eta|^{2}
 u_{M}^{\beta}u^{2}dx,
  \end{aligned}
\end{equation*}
where we use Young's inequality in the last step.
Thus, on account of $I_{1}+I_{2}\leq 0$, we arrive at
\begin{equation*}
  \frac{\mu_{0}}{2}\int_{D_{R}}\eta^{2}
  |\nabla u|^{2}u_{M}^{\beta}dx
  +\beta\mu_{0}\int_{D_{R}\cap\{0<u<M\}}
  \eta^{2}|\nabla u_{M}|^{2}u_{M}^{\beta}dx
  \leq C(\mu_{0},\mu_{1})\int_{D_{R}}
  |\nabla \eta|^{2}u^{\beta+2}dx,
\end{equation*}
and letting $M\to\infty$ leads to the stated estimate
\eqref{appendix:5.1}, which is in fact a good formula
for the later iteration.

\textbf{Step 3.}
In this part, we plan to derive the same formula like
$\eqref{appendix:5.1}$ for the non-negative
supersolution which is defined as follows:
\begin{equation*}
 \int_{D_{R}}A(y,\nabla u)\cdot\nabla v dx
 \geq 0
\end{equation*}
for any $v\in C^{1}_{0}(B_{R})$ with $v\geq 0$.
To achieve our goal, we set
$v=\eta^{2}u_{k}^{\beta}$, where
$u_{k}=u+\frac{1}{k}$ and $\beta<0$. Hence,
\begin{equation*}
  2\int_{D_{R}}\eta A(y,\nabla u)
  \cdot\nabla\eta u_{k}^{\beta}dx+\beta\int_{D_{R}}
  \eta^{2}A(y,\nabla u)\cdot\nabla u u_{k}^{\beta-1}dx\geq 0.
\end{equation*}
In terms of the condition $\eqref{a:2}$, we obtain
\begin{equation*}
\begin{aligned}
-\beta\mu_{0}\int_{D_{R}}\eta^{2}|\nabla u|^{2}u_{k}^{\beta-1}
& \leq 2\mu_{1}\int_{D_{R}}|\nabla u|\eta|\nabla \eta|u_{k}^{\beta}dx\\
& \leq -\frac{\beta \mu_{0}}{2}\int_{D_{R}}|\nabla u|^{2}\eta^{2}u_{k}^{\beta-1}dx+C(\mu_{0},\mu_{1},|\beta|,d)
\int_{D_{R}}|\nabla \eta|^{2}u_{k}^{\beta+1}dx,
\end{aligned}
\end{equation*}
where we employ Young's inequality again, and it implies
\begin{equation*}
  \int_{D_{R}}\eta^{2}|\nabla u|^{2}u_{k}^{\beta-1}dx\leq C(\mu_{0},\mu_{1},|\beta|,d)
\int_{D_{R}}|\nabla \eta|^{2}u_{k}^{\beta+1}dx.
\end{equation*}
Let $k\to\infty$ and $\tilde{\beta}=\beta-1$, we have
\begin{equation}\label{appendix:5.2}
  \int_{D_{R}}\eta^{2}|\nabla u|^{2}u^{\tilde{\beta}}dx\leq
  C(\mu_{0},\mu_{1},|\beta|,d)
\int_{D_{R}}|\nabla \eta|^{2}u^{\tilde{\beta}+2}dx.
\end{equation}

\textbf{Step 4.} We claim that \eqref{appendix:5.1}
implies the local boundedness estimate $\eqref{pri:7.1}$.
We first prove the case $p\geq 2$.
Let $w=u^{\frac{\beta}{2}+1}$, and then the estimate
\eqref{appendix:5.1} may be rewrite as
\begin{equation*}
  \int_{D_{R}}\eta^{2}|\nabla w|^{2}dx
  \lesssim\int_{D_{R}}|\nabla\eta|^{2}w^{2}dx,
\end{equation*}
which together with Sobolev's inequality gives
\begin{equation*}
  \Big(\int_{D_{R}}|\eta w|^{2\chi}dx\Big)^{1/\chi}
  \lesssim\int_{D_{R}}|\nabla \eta|^{2}w^{2}dx,
\end{equation*}
where $\chi= \frac{d}{d-2}$ if $d\geq 3$, and we prefer some
$\chi>2$ in the case of $d=2$.
Recalling $w=u^{\frac{\beta}{2}+1}$, there holds
\begin{equation*}
  \Big(\int_{D_{r}}\big(u^{\beta+2}\big)^{\chi}dx\Big)^{1/\chi}\lesssim
  \frac{1}{(R-r)^{2}}\int_{D_{R}}u^{\beta+2}dx.
\end{equation*}
By setting $\gamma=\beta+2\geq 2$, the above inequality becomes
\begin{equation*}
  \bigg(\int_{D_{r}}u^{\gamma\chi}dx\bigg)^{\frac{1}{\chi\gamma}}
  \lesssim
  \frac{1}{(R-r)^{\frac{2}{\gamma}}}
  \bigg(\int_{D_{R}}u^{\gamma}dx\bigg)^{1/\gamma}.
\end{equation*}
In order to realize the iteration, we prefer
$R_{i}=\frac{R}{2}+\frac{R}{2^{i+1}}$,
$\rho_{i}=2\chi^{i}$ and
$\rho_{i}=\chi\rho_{i-1}, i=0,1,2,\cdots$.
Hence, one may have the formula
\begin{equation*}
  \bigg(\dashint_{D_{R_{i+1}}}u^{\rho_{i+1}}
  \bigg)^{\frac{1}{\rho_{i+1}}}\leq
  C^{\frac{i}{\rho_{i}}}\bigg(\dashint_{D_{R_{i}}}
  u^{\rho_{i}}\bigg)^{\frac{1}{\rho_{i}}}
  \leq C^{\sum\frac{i}{\rho_{i}}}
  \bigg(\dashint_{D_{R}}
  u^{2}\bigg)^{\frac{1}{2}},
\end{equation*}
in which the constant $C$ is independent of $R$.
Consequently,
letting $i\to\infty$, we have proved the desired estimate
$\eqref{pri:7.1}$ for $p\geq 2$.
The case $0<p<2$ easily follows from another
iteration argument and we left it to the readers.

\textbf{Step 5.} We turn to show the estimate
$\eqref{pri:7.2}$ for some $p_0>0$. In terms of
the estimate $\eqref{appendix:5.2}$, it is clear to see that
$u^{-1}$ in fact satisfies the estimate \eqref{appendix:5.1},
which means $u^{-1}$ plays a role as subsolution. Thus, there holds
\begin{equation*}
\sup_{D_{\frac{R}{2}}}u^{-1}
\leq C \Big(\dashint_{D_{R}}u^{-p}\Big)^{\frac{1}{p}},
\end{equation*}
for any $p>0$, and this implies
\begin{equation*}
\begin{aligned}
  \inf_{D_{\frac{R}{2}}}u
 \geq C\Big(\dashint_{D_{R}}u^{-p}\Big)^{-\frac{1}{p}}
 = C\Big(\dashint_{D_{R}}u^{-p}
 \dashint_{D_{R}}u^{p}\Big)^{-\frac{1}{p}}
 \Big(\dashint_{D_{R}}u^{p}\Big)^{\frac{1}{p}}.
  \end{aligned}
\end{equation*}
It's reduced to show for some $p_0>0$, there holds
\begin{equation*}
  \dashint_{D_{R}}u^{-p_0}\dashint_{D_{R}}u^{p_0}\leq C,
\end{equation*}
and it would be done if we proved the following estimate
\begin{equation}\label{appendix:5.5}
 \dashint_{D_{R}}e^{p_0|w|}\leq C,
\end{equation}
where $w=\ln u-\dashint_{B_{R}}\ln u$. To see so, we have the following computation,
\begin{equation*}
\begin{aligned}
\dashint_{D_{R}}e^{p_0 \ln u - p_0\dashint_{D_{R}}\ln u}dx
& =\dashint_{D_{R}}u^{p_0}e^{-\dashint_{D_{R}}p_0\ln u}dx\\
& \geq \dashint_{D_{R}}u^{p_0}\dashint_{D_{R}}
e^{-p_0 \ln u}dx=\dashint_{D_{R}}u^{p_0}\dashint_{D_{R}}
u^{-p_0}dx,
  \end{aligned}
\end{equation*}
where the third step follows from Jensen's inequality.
Now we just need to check \eqref{appendix:5.5}. In fact,
due to John-Nirenberg's inequality it suffices to verify
$w=\ln u-\dashint_{B_{R}}\ln u \in \text{BMO}$. To do so,
Recalling the estimate \eqref{appendix:5.2},
we choose $\beta=-2$ and then
\begin{equation*}
  \int_{D_{R}}\eta^{2}|\nabla u|^{2}u^{-2}dx\leq C\int_{D_{R}}|\nabla\eta|^{2}dx.
\end{equation*}
Noting that $\nabla w=\frac{\nabla u}{u}$, the above estimate
gives
\begin{equation*}
  \int_{D_{r}}|\nabla w|^{2}dx\lesssim r^{d-2}.
\end{equation*}
Thus, it's clear to see
\begin{equation*}
\begin{aligned}
  \dashint_{D_{r}}|w-\dashint_{D_{r}}w|dx
 \leq\Big(\dashint_{D_{r}}|w-\dashint_{D_{r}}w|^{2}dx\Big)^{1/2}
 \lesssim r\Big
 (\dashint_{D_{r}}|\nabla w|^{2}dx\Big)^{1/2}\lesssim 1.
 \end{aligned}
\end{equation*}
Hence, $w\in\text{BMO}$,
and the estimate \eqref{appendix:5.5} follows, and
this leads to the desired estimate $\eqref{pri:7.2}$.
We have completed the whole proof.
\end{proof}

\begin{center}
\textbf{Acknowledgements}
\end{center}
The first author appreciated the hospitality when
she visited the Max Planck Institute for Mathematics
in the Sciences in the winter of 2019. The second
author deeply appreciated Prof. Felix Otto and his lectures,
for the source of the idea of Theorem $\ref{thm:1.3}$.
The research was supported
by the Young Scientists Fund of the National Natural Science
Foundation of China (Grant NO. 11901262),
and supported by the Fundamental Research
Funds for the Central Universities(Grant No. lzujbky-2019-21).

\noindent Li Wang\\
School of Mathematics and Statistics, Lanzhou University,
Lanzhou, 710000, China.\\
E-mail:lwang10@lzu.edu.cn\\

\noindent Qiang Xu\\
Max Planck Institute for Mathematics in the Sciences,
Inselstrasse 22, 04103 Leipzig, Germany\\
E-mail:qiangxu@mis.mpg.de\\

\noindent Peihao Zhao\\
School of Mathematics and Statistics, Lanzhou University,
Lanzhou, 710000, China.\\
E-mail:zhaoph@lzu.edu.cn\\

\end{document}